\documentclass[11pt, twoside]{article}

\usepackage{amsmath}
\usepackage{amssymb}
\usepackage{mathrsfs}
\usepackage{amsthm}
\usepackage{indentfirst}
\usepackage{color}
\usepackage{txfonts}

\allowdisplaybreaks

\def\bint{{\ifinner\rlap{\bf\kern.30em--}
\int\else\rlap{\bf\kern.35em--}\int\fi}\ignorespaces}

\def\sbint{{\ifinner\rlap{\bf\kern.32em--}
\hspace{0.078cm}\int\else\rlap{\bf\kern.45em--}\int\fi}\ignorespaces}

\def\red{\color{red}}

\def\rr{\mathbb{R}}
\def\rn{\mathbb{R}^n}
\def\rnn{\mathbb{R}^{2n}}

\def\cc{\mathbb{C}}
\def\nn{\mathbb{N}}
\def\zz{\mathbb{Z}}

\def\lz{\lambda}

\def\dz{\delta}

\def\ez{\epsilon}

\def\gz{{\gamma}}

\def\tz{\theta}

\def\ls{\lesssim}

\def\fz{\infty}
\def\az{\alpha}

\def\cb{{\mathcal B}}
\def\cd{{\mathcal D}}

\def\cl{{\mathcal L}}
\def\cm{{\mathcal M}}

\def\co{{\mathcal O}}

\def\cx{{\mathcal X}}
\def\cy{{\mathcal Y}}

\def\BMO{\mathrm{\,BMO\,}}
\def\bmo{\mathrm{\,bmo\,}}

\def\CMO{\mathop\mathrm{\,CMO\,}}
\def\cmo{\mathop\mathrm{\,cmo\,}}
\def\VMO{\mathop\mathrm{\,VMO\,}}
\def\vmo{\mathop\mathrm{\,vmo\,}}
\def\XMO{\mathop\mathrm{\,XMO\,}}
\def\xmo{\mathop\mathrm{\,xmo\,}}
\def\MMO{\mathop\mathrm{\,MMO\,}}
\def\Lpwa{{L_{w_1}^{p_1}(\rn)}}
\def\Lpwb{{L_{w_2}^{p_2}(\rn)}}
\def\Lpw{{L_w^p(\rn)}}
\def\xMO{\mathop\mathrm{\,X_1MO\,}}

\def\r{\right}
\def\lf{\left}

\def\noz{{\nonumber}}

\def\dis{\displaystyle}

\def\r{\right}
\def\lf{\left}

\def\at{{\mathop\mathrm{\,at\,}}}
\def\supp{{\mathop\mathrm{\,supp\,}}}

\def\loc{{\mathop\mathrm{\,loc\,}}}

\def\eqref#1{(\ref{#1})}

\newtheorem{theorem}{Theorem}[section]
\newtheorem{lemma}[theorem]{Lemma}
\newtheorem{corollary}[theorem]{Corollary}
\newtheorem{proposition}[theorem]{Proposition}

\theoremstyle{definition}
\newtheorem{remark}[theorem]{Remark}
\newtheorem{definition}[theorem]{Definition}

\newtheorem{question}[theorem]{Question}

\numberwithin{equation}{section}

\begin{document}

\title{\bf\Large  A Survey on Function Spaces of John--Nirenberg Type
\footnotetext{\hspace{-0.35cm} 2020 {\it
Mathematics Subject Classification}. Primary 42B35; Secondary 42B30, 46E35. \endgraf
{\it Key words and phrases}. Euclidean space, cube, congruent cube,
BMO, $JN_p$, (localized) John--Nirenberg--Campanato space,
Riesz--Morrey space, vanishing John--Nirenberg space, duality,
commutator.
 \endgraf
This project is supported by the National
Natural Science Foundation of China (Grant Nos. 11971058, 12071197 and 11871100)
and the National Key Research
and Development Program of China
(Grant No. 2020YFA0712900).}}
\date{ }
\author{Jin Tao, Dachun Yang\footnote{Corresponding author,
E-mail: dcyang@bnu.edu.cn/{\red  July 20, 2021}/Final version.}\ \ and Wen Yuan}
\maketitle

\vspace{-0.8cm}

\begin{center}
\begin{minipage}{13cm}
{\small {\bf Abstract}\quad
In this article, the authors give a survey on the
recent developments of both the John--Nirenberg space $JN_p$
and the space BMO as well as their vanishing subspaces such as VMO, XMO, CMO, $VJN_p$, and $CJN_p$
on $\mathbb{R}^n$ or a given cube $Q_0\subset\mathbb{R}^n$ with finite side length.
In addition, some related open questions are also presented.}
\end{minipage}
\end{center}

\vspace{0.2cm}

\tableofcontents

\section{Introduction}

In this article, a \emph{cube $Q$} means
that it has finite side length and
all its sides parallel to the coordinate axes,
but $Q$ is not necessary to be open or closed.
Moreover, we always let $\mathcal{X}$ be $\rn$
or a given cube of $\rn$.
Recall that the \emph{Lebesgue space} $L^q(\cx)$
with $q\in[1,\fz]$ is defined to be the set of
all measurable functions $f$ on $\cx$ such that
$$\|f\|_{L^q(\mathcal{X})}:=
\begin{cases}
	\displaystyle{\lf[\int_{\cx}|f(x)|^q\,dx\right]^\frac1q}
	&{\rm when}\quad q\in [1,\fz),\\
	\displaystyle{\mathop{\mathrm{ess\,sup}}_{x\in\cx}|f(x)|}
	&{\rm when}\quad q=\fz
\end{cases}$$
is finite.
In what follows, we use $\mathbf{1}_E$ to denote
the \emph{characteristic function} of a set $E\subset\rn$,
and, for any given $q\in[1,\fz)$, $L_{\loc}^q(\cx)$
the set of all measurable functions $f$ on $\cx$ such that
$f{\mathbf 1}_E\in L^q(\cx)$
for any bounded measurable set $E\subset\mathcal{X}$.

It is well known that $L^p(\cx)$ with $p\in[1,\fz]$
plays a leading role in the modern analysis of mathematics.
In particular, when $p\in(1,\fz)$, the space $L^p(\cx)$
enjoys some elegant properties,
such as the reflexivity and the separability,
which no longer hold true in $L^\fz(\cx)$.
Thus, many studies related to $L^p(\cx)$
need some modifications when $p=\fz$;
for instance, the boundedness of Calder\'on--Zygmund operators.
Recall that the Calder\'on--Zygmund operator $T$ is bounded on $L^p(\rn)$
for any given $p\in(1,\fz)$, but not bounded on $L^\fz(\rn)$.
Indeed, $T$ maps $L^\fz(\rn)$ into
$$\BMO(\rn):=\lf\{f\in L^1_\loc(\rn):\,\,
\|f\|_{\BMO(\rn)}:=\sup_{{\rm cube}\,Q}\fint_Q\lf|f(x)-f_Q\r|\,dx<\fz\r\}$$
which was introduced by John and Nirenberg \cite{JN61} in 1961
to study the functions of \emph{bounded mean oscillation},
here and thereafter,
$$f_Q:=\fint_Q f(y)\,dy:=\frac{1}{|Q|}\int_Q f(y)\,dy$$
and the supremum is taken over all cubes $Q$ of $\rn$.
This implies that $\BMO(\cx)$ is a fine substitute of $L^\fz(\cx)$.
Also, it should be mentioned that, in the sense modulo constants,
$\BMO(\cx)$ is a Banach space, but, for simplicity,
we regard $f\in\BMO(\cx)$ as a function rather than
an equivalent class $f+\cc:=\{f+c:\ c\in\cc\}$
if there exists no confusion.
Moreover, the space $\BMO(\cx)$ and its numerous variants
as well as their vanishing subspaces
have attracted a lot of attentions since 1961.
For instance, Fefferman and Stein \cite{FS72} proved that
the dual space of the Hardy space $H^1(\rn)$ is $\BMO(\rn)$;
Coifman et al. \cite{CRW1976} showed an equivalent characterization
of the boundedness of Calder\'on--Zygmund commutators via $\BMO(\rn)$;
Coifman and Weiss introduced the space of homogeneous type and
studied the Hardy space and the BMO space in this context;
Sarason \cite{Sarason75} obtained the equivalent characterization
of $\VMO(\rn)$, the closure in $\BMO(\rn)$ of uniformly continuous functions,
and used it to study stationary stochastic processes satisfying the
strong mixing condition and the algebra $H^\fz+C$;
Uchiyama \cite{U78} established an equivalent characterization
of the compactness of Calder\'on--Zygmund commutators via $\CMO(\rn)$
which is defined to be the closure in $\BMO(\rn)$ of
infinitely differentiable functions on $\rn$ with compact support;
Nakai and Yabuta \cite{ny85} studied pointwise multipliers
for functions on $\rn$ of bounded mean oscillation;
Iwaniec \cite{I92} used the compactness theorem
in Uchiyama \cite{U78} to study linear complex Beltrami equations and
the $L^p(\cc)$-theory of quasiregular mappings.
All these classical results have wide generalizations as well as applications,
and inspire a myriad of further studies in recent years;
see, for instance, the references \cite{J78,cl99,BT13,BDMT15}
for their applications in singular integral operators as well as their commutators,
the references \cite{n93,n97,ny97,n16,n17,lny18}
for their applications in pointwise multipliers,
the references \cite{CDS99,ny12,tyyPA} for their applications in partial differential equations,
and the references \cite{cds05,ABBF16,bd20,dg20,dgl20,dgy21}
for more variants and properties of $\BMO(\rn)$.
In particular, we refer the reader to Chang and Sadosky \cite{cs06} for an
instructive survey on functions of bounded mean oscillation,
and also Chang et al. \cite{cds05} for
BMO spaces on the Lipschitz domain of $\rn$.

Naturally, $\BMO(\cx)$ extends $L^\fz(\cx)$, in the sense that
$L^\fz(\cx)\subsetneqq \BMO(\cx)$ and
$\|\cdot\|_{\BMO(\cx)}\le 2\|\cdot\|_{L^\fz(\cx)}$.
Similarly, such extension exists as well for any $L^p(\cx)$ with $p\in(1,\fz)$.
Indeed, John and Nirenberg \cite{JN61} also introduced
a generalized version of the BMO condition which
was subsequently used to define the so-called John--Nirenberg space $JN_p(Q_0)$ with
exponent $p\in(1,\fz)$ and $Q_0$ being any given cube of $\rn$.
Recall that, for any given $p\in(1,\infty)$ and any given
cube $Q_0$ of $\rn$, the \emph{John--Nirenberg space  $JN_p(Q_0)$}
is defined to be the set of all $f\in L^1(Q_0)$ such that
\begin{equation}\label{jnp}
\|f\|_{JN_p(Q_0)}:=\sup \lf[ \sum_i|Q_i|\lf\{\fint_{Q_i}\lf|f(x)-f_{Q_i}\r|\,dx
   \r\}^p \r]^\frac1p<\infty,
\end{equation}
where the supremum is taken over all collections of
\emph{interior pairwise disjoint} cubes $\{Q_i\}_i$ of $Q_0$.
It is easy to see that
the limit of  $JN_p(Q_0)$ when $p\to\fz$ is just $\BMO(Q_0)$;
see also Corollary \ref{p=8} below.
Moreover, the John--Nirenberg space is closely related to
the Lebesgue space $L^p(Q_0)$ and the weak Lebesgue space $L^{p,\fz}(Q_0)$
which is defined as in Definition \ref{def-weak} below.
Precisely, let $p\in(1,\infty)$. On one hand,
the inequality obtained in \cite[Lemma 3]{JN61}
(see also Theorem \ref{JN61Lem3} below) implies that
$JN_p(Q_0)\subset L^{p,\infty}(Q_0)$;
also, by \cite[Example 3.5]{ABKY11}, we further know that
$JN_p(Q_0)\subsetneqq L^{p,\fz}(Q_0)$.
On the other hand, it is obvious that $L^p(Q_0)\subset JN_p(Q_0)$
with $\|\cdot\|_{JN_p(Q_0)}\le2\|\cdot\|_{L^p(Q_0)}$,
but the striking nontriviality was showed very recently by
Dafni et al. \cite[Proposition 3.2 and Corollary 4.2]{DHKY18}
which says that $L^p(Q_0)\subsetneqq JN_p(Q_0)$.
Combining these facts, we conclude that
\begin{align}\label{Lp<JNp<WLp}
L^p(Q_0)\subsetneqq JN_p(Q_0)\subsetneqq L^{p,\fz}(Q_0).
\end{align}
Therefore, John--Nirenberg spaces are new spaces between Lebesgue spaces
and weak Lebesgue spaces, which motivates us to study the properties of $JN_p$.
Furthermore, various John--Nirenberg-type spaces have attracted a lot of attentions
as well in recent years; see, for instance,
\cite{hmv14,BKM16,m16,DHKY18,BE20,SXY19,tyyBJMA}
for the Euclidean space case, \cite{fpw98,mp98,ABKY11,ms16}
for the metric measure space case.

It should be mentioned that the mean oscillation truly makes a difference
in both $\BMO$ and $JN_p$; for instance,
\begin{enumerate}
\item[{\rm(i)}]
via the characterization of distribution functions,
we know that $\BMO$ is closely related to the space $L_{\exp}$
whose definition [see \eqref{Orlicz} below]
is similar to an equivalent expression
of $\BMO$ but with $f-f_{Q}$ replaced by $f$
(see Proposition \ref{BMO-exp} below);

\item[{\rm(ii)}]
there exists an interesting observation presented by Riesz \cite{R1910},
which says that, in \eqref{jnp}, if we replace $f-f_{Q_i}$ by $f$,
then $JN_p(Q_0)$ turns to be $L^p(Q_0)$.
Moreover, this conclusion also holds true when $Q_0$ is replaced by $\rn$;
see Proposition \ref{Riesz} below.
\end{enumerate}

The main purpose of this article is to give a survey on some
recent developments of both the John--Nirenberg space $JN_p$
and the space BMO, mainly including their several generalized
(or related) spaces and some vanishing subspaces.
We warm up in Section \ref{Sec-JNp} by recalling
some definitions and basic properties of BMO and $JN_p$.
Section \ref{Sec-(L)JNC} summarizes some recent developments of
the John--Nirenberg--Campanato space,
the localized John--Nirenberg--Campanato space,
and the special John--Nirenberg--Campanato space via congruent cubes.
Section \ref{Sec-Riesz} focuses on the Riesz-type space
which differs from the John--Nirenberg space in subtracting integral means,
and its congruent counterpart.
In Section \ref{Sec-Van}, we pay attention to
some vanishing subspaces of aforementioned John--Nirenberg-type spaces,
such as VMO, XMO, CMO, $VJN_p$, and $CJN_p$
on $\mathbb{R}^n$ or any given cube $Q_0$ of $\rn$.
In addition, several related open questions are also summarized
in this survey.

More precisely, the remainder of this survey is organized as follows.

Section \ref{Sec-JNp} is split into two subsections.
In Subsection \ref{Subsec-BMO},
via recalling the definitions of distribution functions and
some related function spaces (including the weak Lebesgue space,
the Morrey space, and the space $L_{\exp}$),
we present the relation
$$L^\fz(Q_0)\subsetneqq \BMO(Q_0)\subsetneqq L_{\exp}(Q_0)$$
in Proposition \ref{LfzBMOLexp} below,
which is a counterpart of \eqref{Lp<JNp<WLp} above,
and also show two equivalent Orlicz-type norms on $\BMO(\rn)$
in Proposition \ref{BMO-exp} below;
moreover, corresponding results for the localized BMO space are also obtained
in Corollary \ref{bmo-exp} below.
Subsection \ref{Subsec-JNp} is devoted to some significant results of $JN_p$,
including the famous John--Nirenberg inequality
(see Theorem \ref{JN61Lem3} below), and the accurate relations
of $JN_p$ and $L^p$ as well as $L^{p,\fz}$
(see Remark \ref{LpJNpWLp} below).
Furthermore, some recent progress of $JN_p$ is also briefly listed
at the end of this subsection.

Section \ref{Sec-(L)JNC} is split into three subsections.
In Subsection \ref{Subsec-JNC},
we first recall the notions of the John--Nirenberg--Campanato space
(for short, JNC space), the corresponding Hardy-type space,
and their basic properties which include
the limit results and the relations with other classical spaces.
Then we review the dual theorem between these two spaces,
and the independence over the second sub-index of
JNC spaces and Hardy-type spaces.
Subsection \ref{Subsec-LJNC} is devoted to the localized counterpart of
Subsection \ref{Subsec-JNC}.
The aim of Subsection \ref{Subsec-ConJNC} is the summary of
the special JNC space defined via congruent cubes
(for short, congruent JNC space), including their basic properties
corresponding to those in Subsection \ref{Subsec-JNC}.
Also, some applications about the boundedness of operators
on congruent spaces are mentioned as well.

In Section \ref{Sec-Riesz}, via subtracting integral means in the JNC space,
we first give the definition of the Riesz-type space appearing in \cite{tyyBJMA},
and then present some basic facts about this space in Subsection \ref{Subsec-RM}.
Moreover, the predual space (namely, the block-type space)
and the corresponding dual theorem
of the Riesz-type space are also displayed in this subsection.
Subsection \ref{Subsec-RM-Appl} is devoted to the congruent counterpart of
the Riesz-type space and the boundedness of some important operators.

Section \ref{Sec-Van} is split into three subsections.
Subsection \ref{Subsec-VanBMO} is devoted to several vanishing subspaces of $\BMO(\rn)$,
including $\VMO(\rn)$, $\CMO(\rn)$, $\MMO(\rn)$, $\XMO(\rn)$, and $\xMO(\rn)$.
We first recall their definitions, and then review
their [except $\MMO(\rn)$] mean oscillation characterizations,
respectively, in Theorems \ref{VMO-char}, \ref{CMO-char}, and \ref{xMO-char} below.
Meanwhile, an open question on the corresponding equivalent characterization
of $\MMO(\rn)$ is also listed in Question \ref{openQ-MMO} below.
Then we further review the compactness theorems of the Calder\'on--Zygmund
commutators $[b,T]$ where $b$ belongs to the vanishing subspaces
$\CMO(\rn)$ as well as $\XMO(\rn)$,
and propose an open question on $[b,T]$ with $b\in\XMO(\rn)$.
Moreover, the characterizations via Riesz transforms
of $\BMO(\rn)$, $\VMO(\rn)$, and $\CMO(\rn)$,
as well as the localized results of these vanishing subspaces are presented.
Also, some open questions are listed in this subsection.
Subsection \ref{Subsec-VanJNp} devotes to the vanishing subspaces of JNC spaces.
We first recall the definition of the vanishing JNC space on cubes
in Definition \ref{def-VJNp},
and then review its equivalent characterization as well as its dual result,
respectively, in Theorems \ref{VJNp-BB} and \ref{Vdual-BB}.
Moreover, for the case of $\rn$, we review the corresponding results for
$VJN_p(\rn)$ and $CJN_p(\rn)$, which are, respectively, counterparts of
$\VMO(\rn)$ and $\CMO(\rn)$
(see Theorems \ref{VJNp-char} and \ref{CJNp-char} below).
As before, some open questions are also listed at the end of this subsection.
Subsection \ref{Subsec-VanConJNp} is devoted to the congruent counterpart of
Subsection \ref{Subsec-VanJNp}, some similar conclusions are listed in this subsection;
meanwhile, some open questions in the JNC space have affirmative answers
in the congruent setting; see Proposition \ref{nontrivial} below.

Finally, we make some conventions on notation.
Let $\nn:=\{1,2,\ldots\}$, $\zz_+:=\nn\cup\{0\}$,
and $\zz_+^n:=(\zz_+)^n$. We always denote by $C$ and $\widetilde{C}$
positive constants which are independent of the main parameters,
but they may vary from line to line.
Moreover, we use $C_{(\gamma,\ \beta,\ \ldots)}$
to denote a positive constant depending on the indicated
parameters $\gamma,\ \beta,\ \ldots$. Constants with subscripts,
such as $C_{0}$ and $A_1$, do not change in different occurrences.
Moreover, the  {symbol} $f\lesssim g$ represents that
$f\le Cg$ for some positive constant $C$.
If $f\lesssim g$ and $g\lesssim f$,
we then write $f\sim g$. If $f\le Cg$ and $g=h$ or $g\le h$,
we then write $f\ls g\sim h$
or $f\ls g\ls h$, \emph{rather than} $f\ls g=h$
or $f\ls g\le h$.
For any $p\in[1,\fz]$, let $p'$ be its \emph{conjugate index},
that is, $p'$ satisfies  $1/p+1/p'=1$.
We use ${\mathbf 1}_E$ to denote the \emph{characteristic function}
of a set $E\subset\rn$
and $|E|$ the \emph{Lebesgue measure} when $E\subset\rn$ is measurable,
and $\mathbf{0}$ the \emph{origin} of $\rn$.
For any function $f$ on $\rn$, let
$\supp(f):=\{x\in\rn:\,\,f(x)\neq0\}$.
Let $\mathbb X$ be a normed linear space. We use $(\mathbb X)^\ast$
to denote its dual space.

\section{BMO and $JN_p$}\label{Sec-JNp}

It is well known that the space BMO has played an important role in
harmonic analysis, partial differential equations,
and other mathematical fields
since it was introduced by John and Nirenberg in the celebrated article
\cite{JN61}. However, in the same article \cite{JN61},
another mysterious space appeared as well,
which is nowadays called the John--Nirenberg space $JN_p$.
Indeed, BMO can be viewed as the limit
space of $JN_p$ as $p\to\fz$; see Proposition \ref{ptofz} and
Corollary \ref{p=8} below with $\az:=0$.
To establish the relations of BMO and $JN_p$,
and also summarize some recent works of John--Nirenberg-type spaces,
we first recall some basic properties of BMO and $JN_p$
in this section.

This section is devoted to some well-known results
of $\BMO(\cx)$ and $JN_p(\cx)$, respectively,
in Subsections \ref{Subsec-BMO} and \ref{Subsec-JNp}.
In addition, it is trivial to find that
all the results in Subsection \ref{Subsec-BMO} also hold true
with the cube $Q_0$ replaced by the ball $B_0$ of $\rn$.

\subsection{(Localized) BMO and $L_{\exp}$}\label{Subsec-BMO}

This subsection is devoted to several equivalent norms of
the spaces BMO and localized BMO.
To this end, we begin with the \emph{distribution function}
\begin{align}\label{distr}
\mathfrak{D}(f;\cx)(t):=|\{x\in\cx:\,\,|f(x)|>t\}|,
\end{align}
where $f\in L^1_\loc(\cx)$ and $t\in(0,\fz)$.
Recall that the distribution function is closely related to the
following weak Lebesgue space.
\begin{definition}\label{def-weak}
Let $p\in(0,\fz)$.
The \emph{weak Lebesgue space} $L^{p,\fz}(\cx)$ is defined by setting
$$L^{p,\fz}(\cx):=\lf\{f\ {\rm is\ measurable\ on\ }\cx:\,\,
\|f\|_{L^{p,\fz}(\cx)}<\fz \r\},$$
where, for any measurable function $f$ on $\cx$,
$$\|f\|_{L^{p,\fz}(\cx)}:=\sup_{t\in(0,\fz)}
\lf[t |\{x\in\cx:\,\,|f(x)|>t\}|^\frac1p\r].$$
\end{definition}

Moreover, the distribution function also features $\BMO(\cx)$,
which is exactly the famous result obtained by John and Nirenberg \cite[Lemma 1']{JN61}:
there exist positive constants $C_1$ and $C_2$, depending only on the dimension $n$,
such that, for any given $f\in\BMO(\cx)$, any given cube $Q\subset\cx$,
and any $t\in(0,\fz)$,
\begin{align}\label{JN-BMO}
\lf|\lf\{x\in Q:\ |f(x)-f_Q|>t \r\} \r|
\le C_1 e^{-\frac{C_2}{\|f\|_{\BMO(\cx)}}t}|Q|.
\end{align}
The main tool used in the proof of \eqref{JN-BMO} is the following
well-known \emph{Calder\'on--Zygmund decomposition};
see, for instance, \cite[p.\,34, Theorem 2.11]{Duo} and also
\cite[p.\,150, Lemma 1]{Stein93}.
\begin{theorem}\label{CZ-decom}
For a given function $f$ which is integrable and non-negative on $\cx$,
and a given positive number $\lz$,
there exists a sequence $\{Q_j\}_j$ of disjoint dyadic cubes of $\cx$ such that
\begin{itemize}
\item [{\rm(i)}]
$f(x)\le\lz\ $ for almost every $x\in\cx\setminus\bigcup_j Q_j$;

\item [{\rm(ii)}]
$|\bigcup_j Q_j |\le\frac1\lz\|f\|_{L^1(\cx)}$;

\item [{\rm(iii)}]
$\lz<\fint_{Q_j}f(x)\,dx\le 2^n\lz$.
\end{itemize}
\end{theorem}
As an application of \eqref{JN-BMO},
we find that, for any given $q\in(1,\fz)$,
$f\in\BMO(\rn)$ if and only if $f\in L^1_\loc(\rn)$ and
\begin{align*}
\|f\|_{{\rm BMO}_q(\rn)}:=\sup_{{\rm cube}\,Q\subset\rn}
\lf[\fint_Q\lf|f(x)-f_Q\r|^q\,dx\r]^\frac1q<\fz;
\end{align*}
meanwhile, $\|\cdot\|_{\BMO(\rn)}\sim\|\cdot\|_{{\rm BMO}_q(\rn)}$;
see, for instance, \cite[p.\,125, Corollary 6.12]{Duo}.
Recently, B\'enyi el al. \cite{BMMST19} gave a comprehensive approach for the
boundedness of weighted commutators via a new equivalent Orlicz-type norm
\begin{align}\label{BMO-OrliczType}
\|f\|_{\mathcal{BMO}(\cx)}:=\sup_{{\rm cube}\,Q\subset\cx}
\|f-f_Q\|_{L_{\exp}(Q)}
\end{align}
(this equivalence is proved in Proposition \ref{BMO-exp} below),
here and thereafter, for any given cube $Q$ of $\rn$,
and any measurable function $g$,
the \emph{locally normalized Orlicz norm} $\|g\|_{L_{\exp}(Q)}$
is defined by setting
\begin{align}\label{Orlicz}
\|g\|_{L_{\exp}(Q)}:=\inf\lf\{\lz\in(0,\fz):\
\fint_Q\lf[e^{\frac{|g(x)|}{\lz}}-1\r]\,dx\le1\r\}.
\end{align}
Moreover, for any given cube $Q$ of $\rn$,
the \emph{space} $L_{\exp}(Q)$ is defined by setting
$$L_{\exp}(Q):=\lf\{f\ {\rm is\ measurable\ on\ }Q:\,\,
\exists\,\lz\in(0,\fz)\ {\rm such\ that}\
\fint_Q e^{\frac{|f(x)|}{\lz}}\,dx<\fz \r\}.$$
The space $L_{\exp}(Q)$ was studied in the interpolation of operators
(see, for instance, \cite[p.\,243]{BS88})
and it is closely related to the space $\BMO(Q)$
(see Proposition \ref{BMO-exp} below).

On the Orlicz function in \eqref{Orlicz}, we have the following properties.
\begin{lemma}\label{type1}
For any $t\in[0,\fz)$, let $\Phi(t):=e^{t}-1$.
Then
\begin{enumerate}
\item [{\rm(i)}]
$\Phi$ is of \emph{lower type} 1,
namely, for any $s\in(0,1)$ and $t\in(0,\fz)$,
$$\Phi(st)\le s\Phi(t);$$

\item [{\rm(ii)}]
$\Phi$ is of \emph{critical lower type} 1,
namely, there exists no $p\in(1,\fz)$ such that,
for any $s\in(0,1)$ and $t\in(0,\fz)$,
$$\Phi(st)\le C s^p\Phi(t)$$
holds true for some constant $C\in[1,\fz)$ independent of $s$ and $t$.
\end{enumerate}
\end{lemma}
\begin{proof}
We first show (i).
For any $s\in(0,1)$ and $t\in(0,\fz)$, let
$$h(s,t):=\Phi(st)-s\Phi(t)
=e^{st}-1-s(e^t-1).$$
Then
$$\frac{\partial}{\partial t}h(s,t)=s e^{st}-s e^t=s(e^{st}-e^t).$$
From this and $s\in(0,1)$, we deduce that, for any $t\in(0,\fz)$,
$\frac{\partial}{\partial t}h(s,t)<0$ and hence
$h(s,t)\le h(s,0)=0$, which shows that $\Phi$ is of lower type 1
and hence completes the proof of (i)

Next, we show that $\Phi$ is of critical lower type 1.
Suppose that there exist a $p\in(1,\fz)$ and a constant $C\in[1,\fz)$ such that,
for any $s\in(0,1)$ and $t\in(0,\fz)$,
$\Phi(st)\le C s^p\Phi(t)$, namely,
\begin{align}\label{crit1}
e^{st}-1\le C s^p(e^t-1).
\end{align}
From $p\in(1,\fz)$ and the L'Hospital rule, we deduce that
$$\lim_{s\to0^+}\frac{\Phi(st)}{s^p\Phi(t)}
=\lim_{s\to0^+}\frac{e^{st}-1}{s^p(e^t-1)}
=\lim_{s\to0^+}\frac{t e^{st}}{p s^{p-1}(e^t-1)}=\fz,$$
which contradicts to \eqref{crit1}, and hence $\Phi$ is of critical lower type $1$.
Here and thereafter, $s\to0^+$ means $s\in(0,1)$ and $s\to0$.
This finishes the proof of (ii) and hence of Lemma \ref{type1}.
\end{proof}

Before showing the equivalent Orlicz-type norms of $\BMO(\cx)$,
we first prove the following equivalent characterizations of $\BMO(\cx)$.
These characterizations might be well known. But, to the best of our knowledge,
we did not find a complete proof. For the convenience of the reader, we present
the details here.
\begin{proposition}\label{equBMO}
The following three statements are mutually equivalent:
\begin{itemize}
\item [{\rm(i)}] $f\in\BMO(\cx)$;

\item [{\rm(ii)}] $f\in L^1_\loc(\cx)$ and
there exist positive constants $C_3$ and $C_4$
such that, for any cube $Q\subset\cx$ and any $t\in(0,\fz)$,
$$\lf|\lf\{x\in Q:\ |f(x)-f_Q|>t \r\} \r|
\le C_3 e^{-C_4 t}|Q|;$$

\item [{\rm(iii)}] $f\in L^1_\loc(\cx)$ and
there exists a $\lz\in(0,\fz)$ such that
$$\sup_{{\rm cube}\,Q\subset \cx}\fint_Qe^{\frac{|f(x)-f_Q|}{\lz}}\,dx<\fz.$$
\end{itemize}
\end{proposition}
\begin{proof}
We prove this proposition via showing
$\rm (i)\Longrightarrow(ii)\Longrightarrow(iii)\Longrightarrow(i)$.

First, the implication $\rm (i)\Longrightarrow(ii)$ was
proved by John and Nirenberg in \cite[Lemma 1']{JN61};
see \eqref{JN-BMO} above.

Next, we show the implication $\rm (ii)\Longrightarrow(iii)$.
Suppose that $f$ satisfies (ii).
Then there exist positive constants $C_3$ and $C_4$ such that,
for any cube $Q\subset \cx$ and any $t\in(0,\fz)$,
$$\lf|\lf\{x\in Q:\ |f(x)-f_Q|>t \r\} \r|\le C_3 e^{-C_4 t}|Q|$$
and hence
\begin{align}\label{C3C4}
&\fint_Q e^{\frac{C_4}{2}|f(x)-f_Q|}\,dx\noz\\
&\quad=\frac1{|Q|}\int_0^\fz \lf|\lf\{x\in Q:\
e^{\frac{C_4}{2}|f(x)-f_Q|}>t \r\} \r|\,dt\noz\\
&\quad=\frac1{|Q|}\lf(\int_0^1+\int_1^\fz\r)\lf|\lf\{x\in Q:\
e^{\frac{C_4}{2}|f(x)-f_Q|}>t \r\} \r|\,dt\notag\\
&\quad\le1+\frac1{|Q|}\int_1^\fz\lf|\lf\{x\in Q:\
|f(x)-f_Q|>2C_4^{-1}\log t \r\} \r|\,dt\notag\\
&\quad\le1+\frac1{|Q|}\int_1^\fz C_3 e^{-C_4 2C_4^{-1}\log t}|Q|\,dt\notag\\
&\quad=1+C_3\int_1^\fz t^{-2}\,dt=1+C_3,
\end{align}
which implies that $f$ satisfies (iii).
This shows the implication $\rm (ii)\Longrightarrow(iii)$.

Finally, we show the implication $\rm (iii)\Longrightarrow(i)$.
Suppose that $f$ satisfies (iii).
Then there exists a $\lz\in(0,\fz)$ such that
$$\sup_{Q\subset \cx}\fint_Q e^{\frac{|f(x)-f_Q|}{\lz}}\,dx<\fz.$$
From this and the basic inequality $x\le e^x-1$ for any $x\in\rr$,
we deduce that
\begin{align*}
\sup_{{\rm cube}\,Q\subset \cx}\fint_Q \lf|f(x)-f_Q\r|\,dx
\le\lz\sup_{{\rm cube}\,Q\subset \cx}\fint_Q \lf[e^{\frac{|f(x)-f_Q|}{\lz}}-1\r]\,dx<\fz,
\end{align*}
which implies that $f$ satisfies (i),
and hence the implication $\rm (iii)\Longrightarrow(i)$ holds true.
This finishes the proof of Proposition \ref{equBMO}.
\end{proof}

In what follows, for any normed space $\mathbb{Y}(\cx)$,
equipped with the norm $\|\cdot\|_{\mathbb{Y}(\cx)}$,
whose elements are measurable functions on $\cx$, let
$$\mathbb{Y}(\cx)/\cc:=\lf\{f\ {\rm is\ measurable\ on\ }\cx:\
\|f\|_{\mathbb{Y}(\cx)/\cc}
:=\inf_{c\in\cc}\|f+c\|_{\mathbb{Y}(\cx)}<\fz\r\}.$$

\begin{proposition}\label{LfzBMOLexp}
Let $Q_0$ be a cube of $\rn$.
Then
$$\lf[L^\fz(Q_0)/\cc\r]\subsetneqq \BMO(Q_0)
\subsetneqq \lf[L_{\exp}(Q_0)/\cc\r].$$
\end{proposition}
\begin{proof}
Indeed, on one hand, from
$$\fint_Q \lf|f(x)-f_Q\r|\,dx\le2\fint_Q \lf|f(x)+c\r|\,dx
\le 2\|f+c\|_{L^\fz(Q_0)}$$
for any $c\in\cc$,
we deduce that $[L^\fz(Q_0)/\cc]\subset \BMO(Q_0)$.
Moreover, let $g(\cdot):=\log|\cdot-c_0|$, where $c_0$ is the center of $Q_0$.
Then $g\in \BMO(Q_0)\setminus [L^\fz(Q_0)/\cc]$;
see \cite[Example 3.1.3]{GTM250} for this fact.
To sum up, we have $[L^\fz(Q_0)/\cc]\subsetneqq \BMO(Q_0)$.

On the other hand, by Proposition \ref{equBMO}(iii),
we easily find that
$\BMO(Q_0)\subset [L_{\exp}(Q_0)/\cc]$.
Moreover, without loss of generality,  we may assume that $Q_0:=(-1,1)$
and let
\begin{align*}
g(x):=
\begin{cases}
-\log(-x),& x\in(-1,0),\\
0, & x=0,\\
\log(x),& x\in(0,1).
\end{cases}
\end{align*}
We claim that $g\in [L_{\exp}(Q_0)/\cc]\setminus\BMO(Q_0)$.
Indeed, for any $\ez\in(0,1)$, let $I_\ez:=(-\ez,\ez)$.
Then
$$\fint_{I_\ez}\lf|g(x)-g_{I_\ez}\r|\,dx
=\frac1{2\ez}\int_{-\ez}^\ez\lf|\log|x|\r|\,dx
=-\frac1\ez\int_0^\ez\log(x)\,dx=1-\log(\ez)\to\fz$$
as $\ez\to0^+$, which implies that $g\notin \BMO(Q_0)$.
However,
$$\int_{Q_0}e^{\frac12 |g(x)|}\,dx=2\int_0^1 e^{-\frac12 \log(x)}\,dx
=2\int_0^1 x^{-\frac12}\,dx=4<\fz,$$
which implies that $g\in L_{\exp}(Q_0)$.
Therefore, $\BMO(Q_0)\subsetneqq [L_{\exp}(Q_0)/\cc]$,
which completes the proof of Proposition \ref{LfzBMOLexp}.
\end{proof}

Now, we show that the two Orlicz-type norms,
\eqref{BMO-OrliczType} and
\begin{align*}
\|f\|_{\widetilde{L_{\exp}}(\cx)}
:=\inf\lf\{\lz\in(0,\fz):\,\,\sup_{{\rm cube}\,Q\subset\cx}
\fint_Q\lf[e^{\frac{|f(x)-f_Q|}{\lz}}-1\r]\,dx\le1\r\}
\end{align*}
for any $f\in L^1_\loc(\cx)$, are equivalent norms of $\BMO(\cx)$.

\begin{proposition}\label{BMO-exp}
The following three statements are mutually equivalent:
\begin{itemize}
\item [{\rm(i)}] $f\in\BMO(\cx)$;
	
\item [{\rm(ii)}] $f\in L^1_\loc(\cx)$ and $\|f\|_{\mathcal{BMO}(\cx)}<\fz$;

\item [{\rm(iii)}] $f\in L^1_\loc(\cx)$ and $\|f\|_{\widetilde{L_{\exp}}(\cx)}<\fz$.
\end{itemize}
Moreover, $\|\cdot\|_{\BMO(\cx)}\sim\|\cdot\|_{\mathcal{BMO}(\cx)}
\sim\|\cdot\|_{\widetilde{L_{\exp}}(\cx)}$.
\end{proposition}

\begin{proof} To prove this proposition, we only need to prove that,
for any $f\in L^1_\loc(\cx)$,
$$\|f\|_{\BMO(\cx)}\sim\|f\|_{\mathcal{BMO}(\cx)}
\sim\|f\|_{\widetilde{L_{\exp}}(\cx)}.$$

We first show that, for any $f\in L^1_\loc(\cx)$, $\|f\|_{\BMO(\cx)}
\le\|f\|_{\mathcal{BMO}(\cx)}$
and $\|f\|_{\BMO(\cx)}\le\|f\|_{\widetilde{L_{\exp}}(\cx)}$.
To this end, let $f\in L^1_\loc(\cx)$.
For any cube $Q\subset\cx$ and any $\lz\in(0,\fz)$,
by $t\le e^t-1$ for any $t\in(0,\fz)$, we have
$$\fint_Q \lf|f(x)-f_Q\r|\,dx\le\lz\fint_Q \lf[e^{\frac{|f(x)-f_Q|}{\lz}}-1\r]\,dx
\le\lz,$$
which implies that
$$\fint_Q \lf|f(x)-f_Q\r|\,dx\le\|f-f_Q\|_{L_{\exp}(Q)}$$
and hence
$$\|f\|_{\BMO(\cx)}\le\|f\|_{\mathcal{BMO}(\cx)}.$$
Moreover, to show $\|f\|_{\BMO(\cx)}\le\|f\|_{\widetilde{L_{\exp}}(\cx)}$,
it suffices to assume that $f\in\widetilde{L_{\exp}}(\cx)$,
otherwise $\|f\|_{\widetilde{L_{\exp}}(\cx)}=\fz$ and hence
the desired inequality automatically holds true.
Then, by $t\le e^t-1$ for any $t\in(0,\fz)$, we conclude that,
for any $n\in\nn$ and any cube $Q\subset\cx$,
\begin{align}\label{Lexp1}
\fint_Q\frac{|f(x)-f_Q|}{\|f\|_{\widetilde{L_{\exp}}(\cx)}+\frac1n}\,dx
\le\fint_Q \lf[e^{\frac{|f(x)-f_Q|}{\|f\|_{\widetilde{L_{\exp}}(\cx)}+\frac1n}}-1\r]\,dx.
\end{align}
From the definition of $\|\cdot\|_{\widetilde{L_{\exp}}(\cx)}$,
we deduce that, for any $n\in\nn$, there exists a
$$\lz_n\in\lf(\|f\|_{\widetilde{L_{\exp}}(\cx)},\|f\|_{\widetilde{L_{\exp}}(\cx)}+\frac1n\r)$$
such that
$$\sup_{{\rm cube}\,Q\subset\cx}
\fint_Q \lf[e^{\frac{|f(x)-f_Q|}{\lz_n}}-1\r]\,dx\le1.$$
By this, \eqref{Lexp1}, and the monotonicity of $e^{(\cdot)}-1$,
we conclude that, for any $n\in\nn$ and any cube $Q\subset\cx$,
$$\fint_Q\frac{|f(x)-f_Q|}{\|f\|_{\widetilde{L_{\exp}}(\cx)}+\frac1n}\,dx
\le 1$$
and hence
$$\fint_Q|f(x)-f_Q|\,dx\le \|f\|_{\widetilde{L_{\exp}}(\cx)}+\frac1n.$$
Letting $n\to\fz$, we then obtain
$$\|f\|_{\BMO(\cx)}=
\sup_{{\rm cube}\,Q\subset\cx}\fint_Q|f(x)-f_Q|\,dx
\le\|f\|_{\widetilde{L_{\exp}}(\cx)}.$$
To sum up, we have, for any $f\in L^1_\loc(\cx)$,
\begin{align}\label{LMBJZS}
\|f\|_{\BMO(\cx)}\le\|f\|_{\mathcal{BMO}(\cx)}
\ \ {\rm and}\ \
\|f\|_{\BMO(\cx)}\le\|f\|_{\widetilde{L_{\exp}}(\cx)}.
\end{align}

Next, we show that the reverse inequalities hold true for any
$f\in L^1_\loc(\cx)$, respectively. Actually, we may assume that
$f\in\BMO(\cx)$ because, otherwise, the desired inequalities
automatically hold true.
Now, let $f\in\BMO(\cx)$. Then, for any cube $Q\subset\cx$ and any
$\lz\in(C_2^{-1}\|f\|_{\BMO(\cx)},\fz)$,
by \eqref{JN-BMO} and the calculation of \eqref{C3C4}, we obtain
\begin{align*}
&\fint_Q e^{\frac{|f(x)-f_Q|}{\lz}}\,dx\\
&\quad\le1+\frac1{|Q|}\int_1^\fz
\lf|\lf\{x\in Q:\ |f(x)-f_Q|>\lz\log t \r\} \r|\,dt\noz\\
&\quad\le1+\frac1{|Q|}\int_1^\fz C_1 e^{-\frac{C_2}{\|f\|_{\BMO(\cx)}}\lz \log t}|Q|\,dt\\
&\quad=1+C_1\int_1^\fz t^{-\frac{C_2 \lz}{\|f\|_{\BMO(\cx)}}}\,dt=1+C_1
\end{align*}
and hence
$$\fint_Q \lf[e^{\frac{|f(x)-f_Q|}{\lz}}-1\r]\,dx\le C_1,$$
where $C_1\in(1,\fz)$ is as in \eqref{JN-BMO}.
From this and Lemma \ref{type1}(i) with $s$ replaced by $1/C_1$, we deduce that
\begin{align}\label{indepQ}
\fint_Q \lf[e^{\frac{|f(x)-f_Q|}{\lz C_1}}-1\r]\,dx
\le\frac{1}{ C_1}\fint_Q \lf[e^{\frac{|f(x)-f_Q|}{\lz}}-1\r]\,dx
\le1.
\end{align}
On one hand, by \eqref{indepQ} and
$$\frac{ C_1}{C_2}\|f\|_{\BMO(\cx)}<\lz C_1<\fz,$$
we conclude that
\begin{align*}
\|f-f_Q\|_{L_{\exp}(Q)}
&=\inf\lf\{\widetilde{\lz}>0:\
\fint_Q\lf[e^{\frac{|f(x)-f_Q|}{\widetilde{\lz}}}-1\r]\,dx\le1\r\}\\
&\le\frac{ C_1}{C_2}\|f\|_{\BMO(\cx)}
\end{align*}
and hence
\begin{align}\label{cBMOBMO}
\|f\|_{\mathcal{BMO}(\cx)}
=\sup_{{\rm cube}\,Q\subset\cx}\|f-f_Q\|_{L_{\exp}(Q)}
\le \frac{ C_1}{C_2}\|f\|_{\BMO(\cx)}.
\end{align}
On the other hand, by \eqref{indepQ}, we conclude that
$$\sup_{Q\subset\cx}\fint_Q \lf[e^{\frac{|f(x)-f_Q|}{\lz C_1}}-1\r]\,dx
\le1.$$
From this and
$$\frac{ C_1}{C_2}\|f\|_{\BMO(\cx)}<\lz C_1<\fz,$$
we deduce that
\begin{align*}
\|f\|_{\widetilde{L_{\exp}}(\cx)}
&=\inf\lf\{\lz\in(0,\fz):\,\,\sup_{{\rm cube}\,Q\subset\cx}
\fint_Q\lf[e^{\frac{|f(x)-f_Q|}{\lz}}-1\r]\,dx\le1\r\}\\
&\le \frac{ C_1}{C_2}\|f\|_{\BMO(\cx)}.
\end{align*}
Combining this with \eqref{cBMOBMO},
we have, for any $f\in\BMO(\cx)$,
\begin{align*}
\|f\|_{\mathcal{BMO}(\cx)}\le\frac{ C_1}{C_2}\|f\|_{\BMO(\cx)}
\ \ {\rm and}\ \
\|f\|_{\widetilde{L_{\exp}}(\cx)}\le\frac{ C_1}{C_2}\|f\|_{\BMO(\cx)}.
\end{align*}
This, together with \eqref{LMBJZS}, then finishes
the proof of Proposition \ref{BMO-exp}.
\end{proof}

\begin{remark}\label{LexpStar}
There exists another norm on $L_{\exp}(Q_0)$,
defined by the distribution functions as follows.
Let $f$ be a measurable function on $Q_0$.
The \emph{decreasing rearrangement} $f^\ast$ of $f$ is defined by setting, for any $u\in[0,\fz)$,
$$f^\ast(u):=\inf\{t\in(0,\fz):\ |\{x\in Q_0:\,\,|f(x)|>t\}|\le u\}.$$
Moreover, for any $v\in(0,\fz)$, let
$$f^{\ast\ast}(v):=\frac1v\int_0^v f^\ast(u)\,du.$$
Then $f\in L_{\exp}(Q_0)$ if and only if $f$ is measurable on $Q_0$ and
$$\|f\|_{L_{\exp}^\ast(Q_0)}:=
\sup_{v\in(0,|Q_0|]}\frac{f^{\ast\ast}(v)}{1+\log(\frac{|Q_0|}{v})}<\fz;$$
meanwhile, $\|\cdot\|_{L_{\exp}^\ast(Q_0)}$ is a norm of $L_{\exp}(Q_0)$;
see \cite[p.\,246, Theorem 6.4]{BS88} for more details.
Furthermore, from \cite[p.\,7, Corollary 1.9]{BS88}, we deduce that
$\|\cdot\|_{L_{\exp}^\ast(Q_0)}$ and $\|\cdot\|_{L_{\exp}(Q_0)}$
are equivalent.
Notice that  $f^\ast$ and $f^{\ast\ast}$ are fundamental tools in the theory
of \emph{Lorentz spaces}; see \cite[p.\,48]{GTM249} for more details.
\end{remark}

Recently, Izuki et al. \cite{ins} obtained both the John--Nirenberg inequality and
the equivalent characterization of $\BMO(\rn)$ on the ball Banach function space
which contains Morrey spaces, (weighted, mixed-norm, variable) Lebesgue spaces,
and Orlicz-slice spaces as special cases;
see \cite[Definition 2.8]{ins} and also \cite{tyyz21} for the related definitions.
Precisely, let $X$ be a ball Banach function space
satisfying the additional assumption that
the Hardy--Littlewood maximal operator $M$ is bounded on $X'$
(the associate space of $X$; see \cite[Definition 2.9]{ins} for its definition), and,
for any $b\in L^1_\loc(\rn)$,
$$\|b\|_{\BMO_X}:=\sup_B\frac1{\|\mathbf1_B\|_X}
\big\| |b-b_B|\mathbf1_B\big\|_X,$$
where the supremum is taken over all balls $B$ of $\rn$.
It is obvious that
$\|\cdot\|_{\BMO_{L^1(\rn)}}=\|\cdot\|_{\BMO(\rn)}$.
Moreover, in \cite[Theorem 1.2]{ins}, Izuki et al. showed that,
under the above assumption of $X$, $b\in\BMO(\rn)$ if and only if
$b\in L^1_\loc(\rn)$ and $\|b\|_{\BMO_X}<\fz$; meanwhile,
$$\|\cdot\|_{\BMO_X}\sim\|\cdot\|_{\BMO(\rn)}.$$
Furthermore, the John--Nirenberg inequality on $X$ was also obtained
in \cite[Theorem 3.1]{ins} which shows that
there exists some positive constant $\widetilde{C}$ such that,
for any ball $B\subset\rn$ and any $\tau\in[0,\fz)$,
$$\lf\|\mathbf1_{\{x\in B:\ |b(x)-b_B|>\tau2^{n+2}\|b\|_{\BMO(\rn)}\}}\r\|_X
\le \widetilde{C} 2^{-\frac{\tau}{1+2^{n+4}\|M\|_{X'\to X'}}} \lf\| \mathbf1_B\r\|_X,$$
where $\|M\|_{X'\to X'}$ denotes the \emph{operator norm} of $M$ on $X'$.
Later, these results were applied in \cite{tyyz21} to establish
the compactness characterization of commutators
on ball Banach function spaces.

Now, we come to the localized counterpart.
The local space $\BMO(\rn)$, denoted by ${\rm bmo}\,(\rn)$,
was originally introduced by Goldberg \cite{G79}.
In the same article, Goldberg also introduced the localized Campanato space
$\Lambda_{\alpha}(\rn)$ with $\alpha\in(0,\fz)$,
which proves the dual space of the localized Hardy space. Later,
Jonsson et al. \cite{JSW84} constructed the localized Hardy space and the localized
Campanato space on the subset of $\rn$;
Chang \cite{C94} studied the localized Campanato space on bounded Lipschitz domains;
Chang et al. \cite{CDS99}
studied the localized Hardy space and its dual space on smooth domains
as well as their applications to boundary value problems;
Dafni and Liflyand \cite{dl19} characterized
the localized Hardy space in the sense of Goldberg, respectively,
by means of the localized Hilbert transform and localized molecules.
In what follows, for any cube $Q$ of $\rn$,
we use $\ell(Q)$ to denote its side length,
and let $\ell(\rn):=\fz$.
Recall that
$$\bmo(\cx):=\lf\{f\in L^1_\loc(\cx):\
\|f\|_{\bmo(\cx)}<\fz \r\},$$
where
$$\|f\|_{\bmo(\cx)}
:=\sup_Q\fint_Q\lf|f(x)-f_{Q,c_0}\r|\,dx$$
with
\begin{align}\label{fQc0}
f_{Q,c_0}:=
\begin{cases}
f_Q &{\rm if}\ \ell(Q)\in(0,c_0),\\
0 &{\rm if}\ \ell(Q)\in[c_0,\ell(\cx))
\end{cases}
\end{align}
for some given $c_0\in(0,\ell(\cx))$,
and the supremum taken over all cubes $Q$ of $\cx$.
Also, a well-known fact is that
$\bmo(\cx)$ is independent of the choice of $c_0$;
see, for instance, \cite[Lemma 6.1]{dy12}.

\begin{proposition}\label{bmoBMO}
Let $\cx$ be $\rn$ or a cube $Q_0$ of $\rn$.
Then
\begin{equation}\label{2.11w}
\lf[L^\fz(\cx)/\cc\r]\subset \lf[\bmo(\cx)/\cc\r]\subset \BMO(\cx)
\end{equation}
and
\begin{align}\label{Bb8}
\|\cdot\|_{\BMO(\cx)}\le 2\inf_{c\in\cc}\|\cdot+c\|_{\bmo(\cx)}
\le 4\inf_{c\in\cc}\|\cdot+c\|_{L^\fz(\cx)}.
\end{align}
Moreover,
\begin{align}\label{8bB}
\lf[L^\fz(\rn)/\cc\r]\subsetneqq \lf[\bmo(\rn)/\cc\r]\subsetneqq \BMO(\rn)
\end{align}
and, for any cube $Q_0$ of $\rn$,
\begin{align}\label{8bBe}
\lf[L^\fz(Q_0)/\cc\r]\subsetneqq \lf[\bmo(Q_0)/\cc\r]
=\BMO(Q_0)\subsetneqq \lf[L_{\exp}(Q_0)/\cc\r]
\end{align}
with
$$\|\cdot\|_{\BMO(Q_0)}\le2\inf_{c\in\cc}\|\cdot+c\|_{\bmo(Q_0)}
\le4\|\cdot\|_{\BMO(Q_0)}.$$
\end{proposition}
\begin{proof}
First, we prove \eqref{Bb8}.
To this end, let $f\in L^1_\loc(\cx)$.
Then, for any $c\in\cc$ and any cube $Q$ of $\cx$,
\begin{align*}
\fint_{Q}\lf|f(x)-f_{Q}\r|\,dx
&=\fint_{Q}\lf|[f(x)+c]-(f+c)_{Q}\r|\,dx\\
&\le2\fint_{Q}\lf|f(x)+c\r|\,dx
\le2\|f+c\|_{L^\fz(Q)}.
\end{align*}
From this and the definitions of $\|\cdot\|_{\BMO(\cx)}$
and $\|\cdot\|_{\bmo(\cx)}$,
it follows that \eqref{Bb8} holds true, which further implies \eqref{2.11w}.

We now show \eqref{8bB}.
Indeed, let
\begin{align*}
	g_1(x):=
	\begin{cases}
		\log(|x|) &{\rm if}\ x\in\rn\setminus\{\mathbf{0}\},\\
		0 &{\rm if}\  x=\mathbf{0}.
	\end{cases}
\end{align*}
From \cite[Example 3.1.3]{GTM250}, we deduce that
$g_1\in\BMO(\rn)$.
However, $g_1\notin \bmo(\rn)$ because, for any $M>\max\{c_0,1\}$,
by the sphere coordinate changing method, we have
$$\fint_{B(\mathbf{0},M)}\lf|\log(|x|) \r|\,dx
\sim\log(M),$$
which tends to infinity as $M\to \fz$.
Thus, $g_1\in \BMO(\rn)\setminus [\bmo(\rn)/\cc]$
and hence we have $[\bmo(\rn)/\cc]\subsetneqq \BMO(\rn)$.
Moreover, define
\begin{align*}
g_2(x):=
\begin{cases}
\log(|x|) &{\rm if}\ |x|\in(0,1),\\
0 &{\rm if}\  |x|\in\{0\}\bigcup[1,\fz).
\end{cases}
\end{align*}
Notice that $g_2\notin L^\fz(\rn)$ and $g_2=\max\{g_1,0\}\in\BMO(\rn)$.
Then, for any cube $Q\subset\rn$, if $\ell(Q)\in(0,c_0)$, then
$$\fint_{Q}\lf|g_2(x)-(g_2)_Q \r|\,dx\le\|g_2\|_{\BMO(\rn)};$$
if $\ell(Q)\in [c_0,\fz)$, then
\begin{align*}
\fint_{Q}\lf|g_2(x) \r|\,dx
&\le\fint_{B(\mathbf{0},1)}\log(|x|)\,dx\sim\|g_2\|_{L^1(\rn)}\sim1.
\end{align*}
To sum up, $\|g_2\|_{\bmo(\rn)}\ls 1+\|g_2\|_{\BMO(\rn)}$
which implies that $g_2\in\bmo(\rn)$ and hence $L^\fz(\rn)\subsetneqq \bmo(\rn)$.
This shows \eqref{8bB}.

We next prove \eqref{8bBe}.
By the above example $g_2$, we conclude that
$L^\fz(Q_0)\subsetneqq\bmo(Q_0)$.
Meanwhile, $\BMO(Q_0)\subsetneqq [L_{\exp}(Q_0)/\cc]$
was obtained in Proposition \ref{LfzBMOLexp}.
Moreover, for any given $f\in\BMO(Q_0)$, we have $f\in L^1(Q_0)$ and hence
\begin{align*}
&\inf_{c\in\cc}\|f-c\|_{\bmo(Q_0)}\\
&\quad=
\begin{cases}
\displaystyle{\fint_{Q}\lf|f(x)-f_Q \r|\,dx\le\|f\|_{\BMO(Q_0)} }
&{\rm if\ } \ell(Q)\in(0,c_0),\\
\displaystyle{\inf_{c\in\cc}\fint_{Q}|f(x)-c|\,dx\le 2\|f\|_{\BMO(Q_0)} }
&{\rm if\ } \ell(Q)\in [c_0,\ell(Q_0)),
\end{cases}\\
&\quad\le2 \|f\|_{\BMO(Q_0)}.
\end{align*}
Combining this with the observations that $[\bmo(Q_0)/\cc]\subset\BMO(Q_0)$
and that, for any $c\in\cc$,
$$\|f\|_{\BMO(Q_0)}=\|f+c\|_{\BMO(Q_0)}\le2\|f+c\|_{\bmo(Q_0)},$$
we find that $[\bmo(Q_0)/\cc]=\BMO(Q_0)$ and
$$\|f\|_{\BMO(Q_0)}\le2\inf_{c\in\cc}\|f+c\|_{\bmo(Q_0)}
\le4\|f\|_{\BMO(Q_0)}.$$
To sum up, we obtain \eqref{8bBe}.
This finishes the proof of Proposition \ref{bmoBMO}.
\end{proof}

Let $f\in L^1_\loc(\cx)$. Similarly to Proposition \ref{BMO-exp}, let
\begin{align}\label{bmo1-Orlicz}
\|f\|_{{\rm bmo}_1(\cx)}:=\sup_{{\rm cube}\,Q\subset\cx}
\lf\|f-f_{Q,c_0}\r\|_{L_{\exp}(Q)}
\end{align}
and
\begin{align}\label{bmo2-Orlicz}
\|f\|_{{\rm bmo}_2(\cx)}:=\inf\lf\{\lz\in(0,\fz):\,\,\sup_{{\rm cube}\,Q\subset\cx}
\fint_Q\lf[e^{\frac{|f(x)-f_{Q,c_0}|}{\lz}}-1\r]\,dx\le1\r\},
\end{align}
where $c_0\in(0,\ell(\cx))$ and $f_{Q,c_0}$ is as in \eqref{fQc0}.
To show that they are equivalent norms of $\bmo(\cx)$,
we first establish the following John--Nirenberg inequality for $\bmo(\cx)$,
namely, Proposition \ref{JNineq-bmo} below.
In what follows, for any given cube $Q$ of $\rn$,
$(a_1,\dots,a_n)$ denotes the \emph{left and lower vertex} of $Q$,
which means that, for any $(x_1,\dots,x_n)\in Q$,
$x_i\ge a_i$ for any $i\in\{1,\dots,n\}$.
Recall that, for any given cube $Q$ of $\rn$,
the \emph{dyadic system} $\mathscr{D}_{Q}$ of $Q$
is defined by setting
\begin{align}\label{dyadicQ0}
\mathscr{D}_{Q}:=\bigcup_{j=0}^\fz\mathscr{D}_{Q}^{(j)},
\end{align}
where, for any $j\in\{0,1,\dots\}$,
$\mathscr{D}_{Q}^{(j)}$ denotes the set of all
$(x_1,\dots,x_n)\in Q$ such that, for any $i\in\{1,\dots,n\}$, either
$$x_i\in\lf[a_i+k_i2^{-j}\ell(Q),a_i+(k_i+1)2^{-j}\ell(Q)\r)$$
for some $k_i\in\{0,1,\dots,2^j-2\}$, or
$$x_i\in\lf[a_i+(1-2^{-j})\ell(Q),a_i+\ell(Q)\r].$$
\begin{proposition}\label{JNineq-bmo}
Let $f\in\bmo(\cx)$ and $c_0\in(0,\ell(\cx))$.
Then there exist positive constants
$C_5$ and $C_6$ such that, for any given cube $Q\subset\cx$,
and any $t\in(0,\fz)$,
\begin{align}\label{JN-bmo}
	\lf|\lf\{x\in Q:\ |f(x)-f_{Q,c_0}|>t \r\} \r|
	\le C_5 e^{-\frac{C_6}{\|f\|_{\bmo(\cx)}}t}|Q|.
\end{align}
\end{proposition}
\begin{proof}
Indeed, this proof is a slight modification of
the proof of \cite[Lemma 1]{JN61} or \cite[Theorem 6.11]{Duo}.
We give some details here again for the sake of completeness.

Let $f\in\bmo(\cx)$.
Then, from Proposition \ref{bmoBMO}, we deduce that
$f\in\BMO(\cx)$ and $\|f\|_{\BMO(\cx)}\le2\|f\|_{\bmo(\cx)}$,
which further implies that, for any cube $Q\subset\cx$ with $\ell(Q)<c_0$,
and any $t\in(0,\fz)$,
\begin{align*}
	\mathfrak{D}\lf(f-f_{Q,c_0};Q\r)(t)
	&=\mathfrak{D}\lf(f-f_Q;Q\r)(t)
	\le C_1 e^{-\frac{C_2}{\|f\|_{\BMO(\cx)}}t}|Q|\\
	&\le C_1 e^{-\frac{C_2}{2\|f\|_{\bmo(\cx)}}t}|Q|,
\end{align*}
where $C_1$ and $C_2$ are as in \eqref{JN-BMO}, and
the distribution function $\mathfrak{D}$ is defined as in \eqref{distr}.
Therefore, to show \eqref{JN-bmo}, it remains to prove that,
for any given cube $Q$ with $\ell(Q)\ge c_0$,
and any $t\in(0,\fz)$,
\begin{align*}
	\lf|\lf\{x\in Q:\ |f(x)|>t \r\} \r|
	\le C_5 e^{-\frac{C_6}{\|f\|_{\bmo(\cx)}}t}|Q|.
\end{align*}
Notice that, in this case, there exists a unique $m_0\in\zz_+$
such that $2^{-(m_0+1)}\ell(Q)<c_0\le2^{-m_0}\ell(Q)$.
Moreover, since inequality \eqref{JN-bmo} is not altered
when we multiply both $f$ and $t$ by the same constant,
without loss of generality, we may assume that $\|f\|_{\bmo(\cx)}=1$.
Let $Q_0$ be any given dyadic subcube of $Q$ with level $m_0$,
namely, $Q_0\in\mathscr{D}_{Q}^{(m_0)}$.
Then, by $c_0\le2^{-m_0}\ell(Q)=\ell(Q_0)$ and the definition of $\|f\|_{\bmo(\cx)}$,
we have
\begin{align}\label{bmo<1}
	\fint_{Q_0}|f(x)|\,dx\le\|f\|_{\bmo(\cx)}=1.
\end{align}
From the Calder\'on--Zygmund decomposition (namely, Theorem \ref{CZ-decom}) of $f$
with height $\lz:=2$, we deduce that
there exists a family $\{Q_{1,j}\}_j\subset\mathscr{D}_{Q_0}^{(1)}$ such that,
for any $j$,
$$2<\fint_{Q_{1,j}}|f(x)|\,dx\le2^{n+1}$$
and $|f(x)|\le2$ when $x\in Q\setminus\bigcup_j Q_{1,j}$.
By this and \eqref{bmo<1}, we conclude that
$$\sum_j \lf|Q_{1,j}\r|\le\frac12\sum_j\int_{Q_{1,j}}|f(x)|\,dx
\le\frac12\int_{Q_0}|f(x)|\,dx\le\frac{1}2|Q_0|$$
and, for any $j$,
$$\lf|f_{Q_{1,j}} \r|\le\lf|\fint_{Q_{1,j}}f(x)\,dx\r|\le2^{n+1}.$$

Moreover, for any $j$, from the Calder\'on--Zygmund decomposition of $f-f_{Q_{1,j}}$
with height $2$, we deduce that there exists a family
$\{Q_{1,j,k}\}_{k}\subset\mathscr{D}_{Q_{1,j}}^{(1)}$ such that, for any $k$,
$$2<\fint_{Q_{1,j,k}}|f(x)-f_{Q_{1,j}}|\,dx\le2^{n+1}$$
and $|f(x)-f_{Q_{1,j}}|\le2$ when $x\in Q\setminus\bigcup_k Q_{1,j,k}$.
Meanwhile, by the construction of $\{Q_{1,j}\}_j$,
we know that $\ell(Q_{1,j})=\frac12\ell(Q_0)=2^{-(m_0+1)}\ell(Q)$ which,
combined with the facts $\|f\|_{\bmo(\cx)}=1$ and $2^{-(m_0+1)}\ell(Q)<c_0$,
further implies that
$$\fint_{Q_{1,j}}\lf|f(x)-f_{Q_{1,j}}\r|\,dx\le\|f\|_{\bmo(\cx)}=1.$$
Thus, we  obtain, for any $j$,
\begin{align*}
	\sum_k \lf|Q_{1,j,k}\r|
	&\le\frac12\sum_j\int_{Q_{1,j,k}}|f(x)-f_{Q_{1,j}}|\,dx\\
	&\le\frac12\int_{Q_{1,j}}|f(x)-f_{Q_{1,j}}|\,dx\le\frac{1}2|Q_{1,j}|
\end{align*}
and, for any $k$,
$$\lf|f_{Q_{1,j,k}}-f_{Q_{1,j}} \r|\le
\fint_{Q_{1,j,k}}|f(x)-f_{Q_{1,j}}|\,dx\le2^{n+1}.$$
Rewrite $\bigcup_{j,k}\{Q_{1,j,k}\}=:\bigcup_j\{Q_{2,j}\}$.
Then we have
$$\sum_j\lf|Q_{2,j}\r|\le\frac12\sum_j\lf|Q_{1,j}\r|\le\frac14|Q_0|$$
and, for any $x\in Q\setminus\bigcup_j Q_{2,j}$,
$$|f(x)|\le\lf|f(x)-f_{Q_{1,j}}\r|+\lf|f_{Q_{1,j}}\r|\le2+2^{n+1}\le 2\cdot2^{n+1}.$$

Repeating this process, then, for any $T\in\nn$,
we obtain a family $\{Q_{T,j}\}_j\subset\mathscr{D}_{Q_0}$
of disjoint dyadic cubes such that
$$\sum_j\lf|Q_{T,j} \r|\le2^{-T}|Q_0|$$
and, for any $x\in Q_0\setminus\bigcup_j Q_{T,j}$,
$$|f(x)|\le T2^{n+1}.$$
Notice that, for any $t\in[2^{n+1},\fz)$,
there exists a unique $T\in\nn$ such that
$T2^{n+1}\le t< (T+1)2^{n+1}\le T2^{n+2}$.
Therefore, we obtain
\begin{align}\label{t-large}
	|\{x\in Q_0:\,\,|f(x)|>t\}|
	&\le\sum_j \lf|Q_{T,j} \r|\le 2^{-T}|Q_0|\noz\\
	&=e^{-T\log2}|Q_0|\le e^{-C_6 t}|Q_0|,
\end{align}
where $C_6:=2^{-(n+2)}\log2$. Also, observe that, if $t\in(0,2^{n+1})$,
then $C_6 t<2^{-1}\log2$ and hence
$$|\{x\in Q_0:\,\,|f(x)|>t\}|\le|Q_0|
\le e^{2^{-1}\log2-C_6 t}|Q_0|=C_5 e^{-C_6 t}|Q_0|,$$
where $C_5:=\sqrt2$.
By this, \eqref{t-large}, and the arbitrariness of $Q_0\in\mathscr{D}_{Q}^{(m_0)}$,
we conclude that, for any $t\in(0,\fz)$,
\begin{align*}
|\{x\in Q:\,\,|f(x)|>t\}|
&=\sum_{Q_0\in\mathscr{D}_{Q}^{(m_0)}}
|\{x\in Q_0:\,\,|f(x)|>t\}|\\
&\le C_5 e^{-C_6 t}\sum_{Q_0\in\mathscr{D}_{Q}^{(m_0)}}|Q_0|
=C_5 e^{-C_6 t}|Q|
\end{align*}
and hence \eqref{JN-bmo} holds true.
This finishes the proof of Proposition \ref{JNineq-bmo}.
\end{proof}

As a corollary of Proposition \ref{JNineq-bmo},
we have the following result, namely, $\|\cdot\|_{{\rm bmo}_1(\cx)}$ in \eqref{bmo1-Orlicz}
and $\|\cdot\|_{{\rm bmo}_2(\cx)}$ in \eqref{bmo2-Orlicz}
are equivalent norms of $\bmo(\cx)$.
The proof of Corollary \ref{bmo-exp} is just a repetition
of the proof of Proposition \ref{BMO-exp}
with \eqref{JN-BMO} replaced by \eqref{JN-bmo};
we omit the details here.

\begin{corollary}\label{bmo-exp}
The following three statements are mutually equivalent:
\begin{itemize}
\item [{\rm(i)}] $f\in\bmo(\cx)$;
	
\item [{\rm(ii)}] $f\in L^1_\loc(\cx)$ and $\|f\|_{{\rm bmo}_1(\cx)}<\fz$;
	
\item [{\rm(iii)}] $f\in L^1_\loc(\cx)$ and $\|f\|_{{\rm bmo}_2(\cx)}<\fz$.
\end{itemize}
Moreover, $\|\cdot\|_{\bmo(\cx)}\sim\|\cdot\|_{{\rm bmo}_1(\cx)}
\sim\|\cdot\|_{{\rm bmo}_2(\cx)}$.
\end{corollary}

\subsection{John--Nirenberg space $JN_p$}\label{Subsec-JNp}

Although there exist a lot of fruitful studies of the space BMO in recent years,
but, as was mentioned before, the structure of $JN_p$ is largely a mystery
and there still exist many unsolved problems on $JN_p$.
The first well-known property of $JN_p$ is the following \emph{John--Nirenberg inequality}
obtained in \cite[Lemma 3]{JN61} which says that $JN_p(Q_0)$ is embedded into
the weak Lebesgue space $L^{p,\infty}(Q_0)$ (see Definition \ref{def-weak}).
\begin{theorem}[John--Nirenberg]\label{JN61Lem3}
Let $p\in(1,\infty)$ and $Q_0$ be a cube of $\rn$.
If $f\in JN_p(Q_0)$, then
$f-f_{Q_0}\in L^{p,\infty}(Q_0)$ and there exists a positive constant
$C_{(n,p)}$, depending only on $n$ and $p$, but independent of $f$, such that
$$\lf\|f-f_{Q_0}\r\|_{L^{p,\infty}(Q_0)}
\le C_{(n,p)}\|f\|_{JN_p(Q_0)}.$$
\end{theorem}
It should be mentioned that the proof of Theorem \ref{JN61Lem3}
relies on the Calder\'on--Zygmund decomposition
(namely, Theorem \ref{CZ-decom}) as well.
Moreover, as an application of Theorem \ref{JN61Lem3},
Dafni et al. recently showed in \cite[Proposition 5.1]{DHKY18} that,
for any given $p\in(1,\fz)$ and $q\in[1,p)$,
$f\in JN_p(Q_0)$ if and only if $f\in L^1(Q_0)$ and
$$\|f\|_{JN_{p,q}(Q_0)}:=\sup \lf[\sum_i|Q_i|
\lf(\fint_{Q_i}\lf|f(x)-f_{Q_i}\r|^q\,dx\r)^{\frac pq} \r]^\frac1p<\fz,$$
where the supremum is taken in the same way as in \eqref{jnp};
meanwhile, $\|\cdot\|_{JN_p(Q_0)}\sim\|\cdot\|_{JN_{p,q}(Q_0)}$.
Furthermore, in \cite[Proposition 5.1]{DHKY18}, Dafni et al. also showed that,
for any given $p\in(1,\fz)$ and $q\in [p,\infty)$,
the spaces $JN_{p,q}(Q_0)$ and $L^q(Q_0)$ coincide as sets.
\begin{remark}\label{LpJNpWLp}
\begin{itemize}
\item[{\rm(i)}]
As a counterpart of Proposition \ref{LfzBMOLexp},
for any given $p\in(1,\fz)$ and any given cube $Q_0$ of $\rn$,
we have
$$L^p(Q_0)\subsetneqq JN_p(Q_0)\subsetneqq L^{p,\fz}(Q_0).$$
Indeed, $L^p(Q_0)\subset JN_p(Q_0)$ is obvious from their definitions;
$JN_p(Q_0)\subset L^{p,\fz}(Q_0)$ is just Theorem \ref{JN61Lem3};
$JN_p(Q_0)\subsetneqq L^{p,\fz}(Q_0)$ was showed in \cite[Example 3.5]{ABKY11}
and the desired function is just $x^{-1/p}$ on $[0,2]$.
However, the fact $L^p(Q_0)\subsetneqq JN_p(Q_0)$ is extremely non-trivial
and was obtained in \cite[Proposition 3.2 and Corollary 4.2]{DHKY18}
via constructing a nice fractal function based on skillful dyadic techniques.
Moreover, in \cite[Theorem 1.1 and Remark 2.4]{DHKY18},
Dafni et al. showed that,
for any given $p\in(1,\fz)$ and any given interval $I_0\subset\rr$
which is no matter bounded or not,
monotone functions are in $JN_p(I_0)$ if and only if they are also in $L^p(I_0)$.
Thus, $JN_p(\cx)$ may be very ``close'' to $L^p(\cx)$ for any given $p\in(1,\fz)$.

\item[{\rm(ii)}]
$JN_1(Q_0)$ coincides with $L^1(Q_0)$.
To be precise, let $Q_0$ be any given cube of $\rn$, and
$$JN_1(Q_0):=\lf\{f\in L^1(Q_0): \|f\|_{JN_1(Q_0)}<\fz\r\},$$
where $\|f\|_{JN_1(Q_0)}$ is defined as in \eqref{jnp}
with $p$ replaced by $1$.
Then we claim that $JN_1(Q_0)= [L^1(Q_0)/\cc]$
with equivalent norms.
Indeed, for any $f\in JN_1(Q_0)$,
by the definition of $\|f\|_{JN_1(Q_0)}$,
we have
$$\|f\|_{JN_1(Q_0)}
\ge \lf\|f-f_{Q_0}\r\|_{L^1(Q_0)}
\ge \inf_{c\in\cc}\|f+c\|_{L^1(Q_0)}
=:\|f\|_{L^1(Q_0)/\cc}.$$
Conversely, for any given $f\in L^1(Q_0)$ and any $c\in\cc$, we have
\begin{align*}
\|f\|_{JN_1(Q_0)}
&=\sup\sum_i \int_{Q_i}\lf|f(x)-f_{Q_i}\r|\,dx\\
&\le2\sup\sum_i \int_{Q_i}\lf|f(x)+c\r|\,dx \\
&\le2\|f+c\|_{L^1(Q_0)},
\end{align*}
which implies that $\|f\|_{JN_1(Q_0)}\le\|f\|_{L^1(Q_0)/\cc}$
and hence the above claim holds true.
Moreover, the relation between $JN_1(\rr)$ and $L^1(\rr)$
was studied in \cite[Proposition 2]{BE20}.

\item[{\rm(iii)}]
Garsia and Rodemich in \cite[Theorem 7.4]{gr74}
showed that, for any given $p\in(1,\fz)$,
$f\in L^{p,\fz}(Q_0)$ if and only if $f\in L^1(Q_0)$ and
$$\|f\|_{{\rm GaRo}_p(Q_0)}
:=\sup \frac1{(\sum_{i}|Q_i|)^{1/p'}}
\sum_i \frac1{|Q_i|}\int_{Q_i}\int_{Q_i}
\lf|f(x)-f(y)\r|\,dx\,dy<\fz,$$
where the supremum is taken in the same way as in \eqref{jnp};
meanwhile,
$$\|\cdot\|_{L^{p,\fz}(Q_0)}\sim\|\cdot\|_{{\rm GaRo}_p(Q_0)};$$
see also \cite[Theorem 5(ii)]{m16} for this equivalence.
Moreover, in \cite[Theorem 5(i)]{m16}, Milman showed that
$\|\cdot\|_{{\rm GaRo}_p(Q_0)}\le2\|\cdot\|_{JN_p(Q_0)}$.
\end{itemize}
\end{remark}

Recall that the predual space of $\BMO(\cx)$ is the Hardy space $H^1(\cx)$;
see, for instance, \cite[Theorem B]{CW77}.
Similarly to this duality,
Dafni et al. \cite{DHKY18} also obtained the predual space of $JN_p(Q_0)$
for any given $p\in(1,\fz)$,
which is denoted by the Hardy kind space $HK_{p'}(Q_0)$,
here and thereafter, $1/p+1/p'=1$.
Later, these properties, including equivalent norms and duality,
were further studied on several John--Nirenberg-type spaces,
such as John--Nirenberg--Campanato spaces,
localized John--Nirenberg--Campanato spaces, and
congruent John--Nirenberg--Campanato spaces
(see Section \ref{Sec-(L)JNC} for more details),
and Riesz-type spaces (see Section \ref{Sec-Riesz} for more details).

Finally, let us briefly recall some other related studies concerning
the John--Nirenberg space $JN_p$,
which would not be stated in details in this survey
while all of them are quite instructive:
\begin{enumerate}
\item[$\bullet$]
Stampacchia \cite{S65} introduced the space $N^{(p,\lambda)}$,
which coincides with $JN_{(p,1,0)_\alpha}(Q_0)$ in Definitions \ref{JNpqa}
if we write $\lambda=p\az$ with $p\in(1,\fz)$ and $\az\in(-\fz,\fz)$,
and applied them to the context of interpolation of operators.

\item[$\bullet$]
Campanato \cite{Campanato66} also used the John--Nirenberg spaces
to study the interpolation of operators.

\item[$\bullet$]
In the context of doubling metric spaces,
$JN_p$ and median-type $JN_p$ were studied, respectively, by
Aalto et al. in \cite{ABKY11} and Myyryl\"ainen in \cite{m21}.

\item[$\bullet$]
Hurri-Syrj\"anen et al. \cite{hmv14} established a local-to-global result for the
space $JN_p(\Omega)$ on an open subset $\Omega$ of $\rn$.
More precisely, it was proved that the norm
$\|\cdot\|_{JN_p(\Omega)}$ is dominated by its local version
$\|\cdot\|_{JN_{p,\tau}(\Omega)}$ modulus constants;
here, $\tau\in[1,\fz)$, for any open subset $\Omega$ of $\rn$,
the related ``norm" $\|\cdot\|_{JN_p(\Omega)}$ is
defined in the same way as $\|\cdot\|_{JN_p(Q_0)}$ in \eqref{jnp}
with $Q_0$ replaced by $\Omega$,
and $\|\cdot\|_{JN_{p,\tau}(\Omega)}$ is defined in the same way as  $\|\cdot\|_{JN_p(\Omega)}$
with an additional requirement $\tau Q\subset\Omega$ for all chosen cubes $Q$ in the definition
of $\|\cdot\|_{JN_p(\Omega)}$.

\item[$\bullet$]
Marola and  Saari \cite{ms16} studied the corresponding results
of Hurri-Syrj\"anen et al. \cite{hmv14} on metric measure spaces,
and obtained the equivalence between the local and the global $JN_p$ norms.
Moreover, in both articles \cite{hmv14,ms16},
a global John--Nirenberg inequality for $JN_p(\Omega)$ was established.

\item[$\bullet$]
Berkovits et al. \cite{BKM16} applied the dyadic variant of $JN_p(Q_0)$
in the study of self-improving properties of some Poincar\'e-type inequalities.
Later, the dyadic $JN_p(Q_0)$ was further studied by
Kinnunen and Myyryl\"ainen in \cite{km21}.

\item[$\bullet$]
A. Brudnyi and Y. Brudnyi \cite{BB20} introduced
a class of function spaces $V_\kappa([0,1]^n)$
which coincides with $JN_{(p,q,s)_\az}([0,1]^n)$ defined below
for suitable range of indices;
see \cite[Proposition 2.9]{tyyNA} for more details.
Very recently, Dom\'inguez and Milman \cite{dm21} further introduced
and studied sparse Brudnyi and John--Nirenberg spaces.

\item[$\bullet$]
Blasco and Espinoza-Villalva \cite{BE20}
computed the concrete value of $\|\mathbf{1}_{A}\|_{JN_p(\rr)}$
for any given $p\in[1,\fz]$ and any measurable set $A\subset\rr$ of
positive and finite Lebesgue measure, where $JN_\fz(\rr):=\BMO(\rr)$.

\item[$\bullet$]
The $JN_p(Q_0)$-type norm $\|\cdot\|_{{\rm GaRo}_p(Q_0)}$ in Remark \ref{LpJNpWLp}(iii)
was further generalized and studied in Astashkin and Milman \cite{am20}
via the Str\"omberg--Jawerth--Torchinsky local maximal operator.
\end{enumerate}

\section{John--Nirenberg--Campanato space}\label{Sec-(L)JNC}

The main target of this section  is to summarize the main
results of John--Nirenberg--Campanato spaces,
localized John--Nirenberg--Campanato spaces,
and congruent John--Nirenberg--Campanato spaces
obtained, respectively, in \cite{tyyNA,SXY19,jtyyz21}.
Moreover, at the end of each part,
we list some open questions which are still unsolved so far.
Now, we first recall some definitions
of some basic function spaces.

\begin{enumerate}
\item[$\bullet$]
For any $s\in\zz_+$ (the set of all non-negative integers),
let $\mathcal{P}_s(Q)$ denote the set of all polynomials
of degree not greater than $s$ on the cube $Q$, and
$P_Q^{(s)}(f)$ the unique polynomial of degree not greater than $s$ such that
\begin{align}\label{PQsf}
\int_Q \lf[f(x)-P_Q^{(s)}(f)(x)\r]x^\gz\,dx=0,\quad \forall\,|\gz|\le s,
\end{align}
where $\gz:=(\gz_1,\ldots,\gz_n)\in \zz_+^n:=(\zz_+)^n$,
$|\gz|:=\gz_1+\cdots+\gz_n$, and
$x^\gz:=x_1^{\gz_1}\cdots x_n^{\gz_n}$ for any $x:=(x_1,\ldots,x_n)\in\rn$.

\item[$\bullet$]
Let $q\in[1,\fz]$ and $Q_0$ be a cube of $\rn$.
For any measurable function $f$, let
$$\|f\|_{L^q(Q_0,|Q_0|^{-1}dx)}:=\lf[\fint_{Q_0}|f(x)|^q\,dx\r]^\frac1q.$$

\item[$\bullet$]
Let $q\in(1,\fz)$, $s\in\zz_+$, and $Q_0$ be a cube of $\rn$.
The \emph{space $L^q(Q_0,|Q_0|^{-1}dx)/\mathcal{P}_s(Q_0)$} is defined by setting
$$L^q(Q_0,|Q_0|^{-1}dx)/\mathcal{P}_s(Q_0)
:=\lf\{f\in L^q(Q_0):\,\,\|f\|_{L^q(Q_0,|Q_0|^{-1}dx)/\mathcal{P}_s(Q_0)}<\fz \r\},$$
where
$$\|f\|_{L^q(Q_0,|Q_0|^{-1}dx)/\mathcal{P}_s(Q_0)}
:=\inf_{m\in\mathcal{P}_s(Q_0)}\|f+m\|_{L^q(Q_0,|Q_0|^{-1}dx)}.$$

\item[$\bullet$]
For any given $v\in[1,\infty]$ and $s\in\zz_+$,
and any measurable subset $E\subset\rn$, let
\begin{align*}
L_s^v(E)&:=\lf\{f\in L^v(E):\
  \int_E f(x)x^\gz\,dx=0,\ \forall\,\gz\in\zz_+^n,\ |\gz|\le s\r\}.
\end{align*}
\end{enumerate}

Let $Q$ be any given cube of $\rn$.
It is well known that $P_Q^{(0)}(f)=f_Q$ and, for any $s\in\zz_+$,
there exists a constant $C_{(s)}\in[1,\infty)$, independent of $f$ and $Q$, such that
\begin{align}\label{Cs}
  \lf|P_Q^{(s)}(f)(x)\r|\le C_{(s)} \fint_Q |f(x)|\,dx,\quad \forall\,x\in Q.
\end{align}
Indeed, let $\{\varphi_Q^{(\gz)}:\ \gz\in\zz^n_+,\,|\gz|\le s\}$ denote the Gram--Schmidt
orthonormalization of $\{x^\gz:\ \gz\in\zz^n_+,\,|\gz|\le s\}$ on the cube $Q$ with respect to
the weight $1/|Q|$, namely, for any $\gz$, $\nu$, $\mu\in\zz^n_+$
with $|\gz|\le s$, $|\nu|\le s$, and $|\mu|\le s$,
$\varphi_Q^{(\gz)}\in \mathcal{P}_s(Q)$ and
$$\langle\varphi_Q^{(\nu)},\varphi_Q^{(\mu)} \rangle
:=\frac1{|Q|}\int_Q \varphi_Q^{(\nu)}(x)\varphi_Q^{(\mu)}(x)\,dx=
\begin{cases}
  1,&\nu=\mu,\\
  0,&\nu\neq\mu.
\end{cases}$$
Then
$$P_Q^{(s)}(f)(x):=\sum_{\{\gz\in\zz^n_+:\ |\gz|\le s\}}\langle\varphi_Q^{(\gz)},f \rangle
\varphi_Q^{(\gz)}(x),\quad\forall\,x\in Q,$$
and we can choose
$C_{(s)}:=\sum_{\{\gz\in\zz^n_+:\ |\gz|\le s\}}\|\varphi_Q^{(\gz)}\|^2_{L^\fz (Q)}$
satisfying \eqref{Cs}; see \cite[p.\,83]{TaiblesonWeiss80} and
\cite[p.\,54, Lemma 4.1]{Lu95} for more details.

\subsection{John--Nirenberg--Campanato spaces}\label{Subsec-JNC}

In this subsection, we recall the definitions of Campanato spaces,
John--Nirenberg--Campanato spaces (for short, JNC spaces),
and Hardy-type spaces, respectively, in
Definitions \ref{Campanato}, \ref{JNpqa}, and \ref{HKuvb} below.
Moreover, we review some properties of JNC spaces and Hardy-type spaces,
including their limit spaces (Proposition \ref{ptofz} and Corollary \ref{p=8} below),
relations with the Lebesgue space (Propositions \ref{JNpq=Lq} and \ref{HKqq=Lq} below),
the dual result (Theorem \ref{duality} below),
the monotonicity over the first sub-index (Proposition \ref{mono} below),
the John--Nirenberg-type inequality (Theorem \ref{JohnNirenberg} below),
and the equivalence over the second sub-index
(Propositions \ref{JNpqa=JNp1a} and \ref{HKuvb=HKu8b} below).

A general dual result for Hardy spaces was given by
Coifman and Weiss \cite{CW77} who proved that,
for any given $p\in(0,1]$, $q\in[1,\infty]$, and $s$ being the non-negative
integer not smaller than $n(\frac1p-1)$,
the dual space of the Hardy space $H^p(\rn)$ is
the Campanato space $\mathcal{C}_{\frac1p-1,\,q,\,s}(\rn)$
which was introduced by Campanato \cite{Campanato64} and coincides with
$\BMO(\rn)$ when $p=1$.

\begin{definition}\label{Campanato}
Let $\alpha\in[0,\infty)$, $q\in[1,\infty)$, and $s\in\zz_+$.
\begin{itemize}
\item[{\rm(i)}]
The \emph{Campanato space} $\mathcal{C}_{\alpha,q,s}(\cx)$
is defined by setting
$$\mathcal{C}_{\alpha,q,s}(\cx):=\lf\{f\in L^q_{\loc}(\cx):\,\,\|f\|_{\mathcal{C}_{\alpha,q,s}(\cx)}<\infty\r\},$$
where $$\|f\|_{\mathcal{C}_{\alpha,q,s}(\cx)}:=
\sup|Q|^{-\alpha}\lf[\fint_{Q}\lf|f-P_{Q}^{(s)}(f)\r|^q\r]
 ^{\frac 1q}$$
and the supremum is taken over all cubes $Q$ of $\cx$.
In addition, the ``norm'' $\|\cdot\|_{\mathcal{C}_{\alpha,q,s}(\cx)}$
is defined modulo polynomials and,
for simplicity, the space $\mathcal{C}_{\alpha,q,s}(\cx)$
is regarded as the quotient space
$\mathcal{C}_{\alpha,q,s}(\cx)/\mathcal{P}_s(\cx)$.

\item[{\rm(ii)}]
The \emph{dual space $(\mathcal{C}_{\alpha,q,s}(\cx))^\ast$} of $\mathcal{C}_{\alpha,q,s}(\cx)$
is defined to be the set of all continuous linear functionals on $\mathcal{C}_{\alpha,q,s}(\cx)$
equipped with the weak-$\ast$ topology.
\end{itemize}
\end{definition}

In what follows, for any $\ell\in(0,\fz)$,
$Q(\mathbf{0},\ell)$ denotes the cube centered at the origin $\mathbf{0}$
with side length $\ell$.

\begin{remark}\label{rem-Morrey}
	Let $0<q\le p\le \fz$.
	The \emph{Morrey space} $ M^p_q(\rn)$,
	introduced by Morrey in \cite{m38}, is defined by setting
	$$ M^p_q(\rn):=\lf\{f\in L^q_\loc(\rn):\,\,
	\|f\|_{ M^p_q(\rn)}<\fz\r\},$$
	where, for any $f\in L^q_\loc(\rn)$,
	$$\|f\|_{ M^p_q(\rn)}:=\sup_{{\rm cube}\,Q\subset\rn}
	|Q|^{\frac1p}\lf[\fint_Q |f(y)|^q\,dy\r]^\frac1q.$$
	From Campanato \cite[Theorem 6.II]{Campanato64}, it follows that,
	for any given $q\in[1,\infty)$ and $\alpha\in[-\frac1q,0)$,
	and any $f\in\mathcal{C}_{q,\az,0}(\cx)$,
	\begin{align}\label{Mor=Cam}
	\|f\|_{\mathcal{C}_{q,\az,0}(\cx)}
	\sim\lf\|f-\sigma(f)\right\|_{M_{q}^{-1/\az}(\cx)},
	\end{align}
	where the positive equivalence constants are independent of $f$, and
	$$\sigma(f):=
	\begin{cases}
		\displaystyle{\lim_{\ell\to\fz}\frac1{|Q({\mathbf 0},\ell)|}
			\int_{Q({\mathbf 0},\ell)}f(x)\,dx}
		& {\rm if\ }\cx=\rn,\\
		\displaystyle{\frac{1}{|Q_0|}\int_{Q_0}f(x)\,dx}
		&{\rm if\ }\cx=Q_0;
	\end{cases}$$
	see also Nakai \cite[Theorem 2.1 and Corollary 2.3]{n16}
	for this conclusion on spaces of homogeneous type.
	In addition, a surprising result says that,
	in the definition of supremum $\|\cdot\|_{ M^p_q(\rn)}$,
	if ``cubes'' were changed into ``measurable sets'',
	then the Morrey norm $\|\cdot\|_{ M^p_q(\rn)}$
	becomes an equivalent norm of the weak Lebesgue space
	(see Definition \ref{def-weak}). To be precise,
	for any given $0<q<p<\fz$,
	$f\in L^{p,\fz}(\rn)$ if and only if $f\in L^q_\loc(\rn)$ and
	$$\|f\|_{\widetilde{M^p_q}(\rn)}:=\sup_{A\subset\rn,\,|A|\in(0,\fz)}
	|A|^{\frac1p}\lf[\fint_A |f(y)|^q\,dy\r]^\frac1q<\fz;$$
	moreover,
	\begin{align*}
		\|\cdot\|_{L^{p,\fz}(\rn)}\le \|\cdot\|_{\widetilde{M^p_q}(\rn)}
		\le\lf(\frac{p}{p-q}\r)^\frac1q \|\cdot\|_{L^{p,\fz}(\rn)};
	\end{align*}
	see, for instance, \cite[p.\,485, Lemma 2.8]{GR85}.
	Another interesting $JN_p$-type equivalent norm of the weak Lebesgue space
	was presented in Remark \ref{LpJNpWLp}(iii).
\end{remark}

Inspired by the relation between BMO and the Campanato space,
as well as the relation between BMO and $JN_p$,
Tao el al. \cite{tyyNA} introduced a Campanato-type space $JN_{(p,q,s)_\alpha}(\cx)$
in the spirit of the John--Nirenberg space
$JN_p(Q_0)$, which contains $JN_p(Q_0)$ as a special case.
This John--Nirenberg--Campanato space is defined not only on any cube $Q_0$
but also on the whole space $\rn$.

\begin{definition}\label{JNpqa}
Let $p$, $q\in[1,\infty)$, $s\in\zz_+$, and $\alpha\in\rr$.
\begin{itemize}
\item[{\rm(i)}]
The \emph{John--Nirenberg--Campanato space} (for short, JNC \emph{space})
$JN_{(p,q,s)_\alpha}(\cx)$
is defined by setting
$$JN_{(p,q,s)_\alpha}(\cx):=\lf\{f\in L^q_{\loc}(\cx):\,\,
\|f\|_{JN_{(p,q,s)_\alpha}(\cx)}<\infty\r\},$$
where
$$\|f\|_{JN_{(p,q,s)_\alpha}(\cx)}:=
\sup\lf\{\sum_i|Q_i|\lf[|Q_i|^{-\alpha}\lf\{\fint_{Q_i}\lf|f(x)-P_{Q_i}^{(s)}(f)(x)\r|^q\,dx\r\}
   ^{\frac 1q}\r]^p\r\}^{\frac 1p},$$
$P_{Q_i}^{(s)}(f)$ for any $i$ is as in \eqref{PQsf} with $Q$ replaced by $Q_i$,
and the supremum is taken over all collections of interior pairwise disjoint cubes
$\{Q_i\}_i$ of $\cx$.
Also, the ``norm'' $\|\cdot\|_{JN_{(p,q,s)_\alpha}(\cx)}$
is defined modulo polynomials and,
for simplicity, the space $JN_{(p,q,s)_\alpha}(\cx)$ is regarded as the quotient space
$JN_{(p,q,s)_\alpha}(\cx)/\mathcal{P}_s(\cx)$.

\item[{\rm(ii)}]
The \emph{dual space $(JN_{(p,q,s)_\alpha}(\cx))^\ast$} of $JN_{(p,q,s)_\alpha}(\cx)$
is defined to be the set of all continuous linear functionals on $JN_{(p,q,s)_\alpha}(\cx)$
equipped with the weak-$\ast$ topology.
\end{itemize}
\end{definition}

\begin{remark}\label{rem-ainrr}
In \cite{tyyNA}, the JNC space was introduced only for any given $\az\in[0,\fz)$
to study its relation with the Campanato space in Definition \ref{Campanato},
and for any given $p\in(1,\fz)$ due to Remark \ref{LpJNpWLp}(ii).
However, many results in \cite{tyyNA} also hold true when $\az\in\rr$ and $p=1$
just with some slight modifications of their proofs.
Thus, in this survey, we introduce the JNC space for any given $\az\in\rr$ and $p\in[1,\fz)$,
and naturally extend some related results with some identical proofs omitted.
\end{remark}

The following proposition,
which is just \cite[Proposition 2.6]{tyyNA},
means that the classical Campanato space serves as
a limit space of $JN_{(p,q,s)_\alpha}(\cx)$,
similarly to the Lebesgue spaces
$L^\fz(\cx)$ and $L^p(\cx)$ when $p\to\fz$.

\begin{proposition}\label{ptofz}
Let $\alpha\in[0,\infty)$, $q\in[1,\infty)$, and $s\in\zz_+$. Then
$$\lim_{p\to\fz}JN_{(p,q,s)_\alpha}(\cx)=\mathcal{C}_{\alpha,q,s}(\cx)$$
in the following sense: for any $f\in\bigcup_{r\in[1,\infty)}\bigcap_{p\in[r,\infty)}
JN_{(p,q,s)_\alpha}(\cx)$,
$$\lim_{p\to\infty}\|f\|_{JN_{(p,q,s)_\alpha}(\cx)}
=\|f\|_{\mathcal{C}_{\alpha,q,s}(\cx)}.$$
\end{proposition}
In Proposition \ref{ptofz}, if we take $\cx=Q_0$,
we then have following corollary which is just \cite[Corollary 2.8]{tyyNA}.
\begin{corollary}\label{p=8}
Let $q\in[1,\infty)$, $\alpha\in[0,\infty)$, $s\in\zz_+$,
and $Q_0$ be a cube of $\rn$. Then
$$\mathcal{C}_{\alpha,q,s}(Q_0)=
 \lf\{ f\in\bigcap_{p\in[1,\infty)}JN_{(p,q,s)_\alpha}(Q_0):\,\,
      \lim_{p\to\infty}\|f\|_{JN_{(p,q,s)_\alpha}(Q_0)}<\infty\r\}$$
and, for any $f\in \mathcal{C}_{\alpha,q,s}(Q_0)$,
$$\|f\|_{\mathcal{C}_{\alpha,q,s}(Q_0)}=\lim_{p\to\fz}\|f\|_{JN_{(p,q,s)_\alpha}(Q_0)}.$$
\end{corollary}

\begin{remark}
\begin{itemize}
\item[{\rm(i)}]
Let $p\in(1,\fz)$ and $Q_0$ be a cube of $\rn$.
It is easy to show that
$$\BMO(Q_0)\subset JN_p(Q_0).$$
However, we claim that
$$\BMO(\rn)\nsubseteq JN_p(\rn).$$
Indeed,  for the simplicity of the presentation,
without loss of generality, we may show this claim only in $\rr$.
Let $g(x):=\log(|x|)$ for any $x\in\rr\setminus\{0\}$, and $g(0):=0$.
Then $g\in\BMO(\rr)$ due to \cite[Example 3.1.3]{GTM250},
and hence it suffices to prove that
$g\notin JN_p(\rr)$ for any given $p\in(1,\fz)$.
To do this, let $I_t:=(0,t)$ for any $t\in(0,\fz)$.
Then, by some simple calculations, we obtain
$$g_{I_t}=\fint_{I_t}g(x)\,dx
=\frac1t\int_{0}^t \log(x)\,dx=\log(t)-1$$
and hence
\begin{align*}
&\lf| \lf\{x\in I_t:\ \lf|g(x)-g_{I_t}\r|>\frac12 \r\}\r|\\
&\quad=\lf| \lf\{x\in (0,t):\ \lf|\log(x)-\lf[\log(t)-1\r]\r|>\frac12 \r\}\r|\\
&\quad\ge t-te^{-\frac12}=t\lf(1-e^{-\frac12}\r)\to\fz
\end{align*}
as $t\to\fz$. But, the John--Nirenberg inequality of $JN_p(I_t)$
in Theorem \ref{JN61Lem3} implies that, for any $t\in(0,\fz)$,
$$\lf| \lf\{x\in I_t:\ \lf|g(x)-g_{I_t}\r|>\frac12 \r\}\r|
\ls\lf[\frac{\|g\|_{JN_p(I_t)}}{\frac12}\r]^p
\ls\|g\|_{JN_p(\rr)}^p$$
with the implicit positive constants depending only on $p$.
Thus, $g\notin JN_p(\rr)$ and hence the above claim holds true.

\item[{\rm(ii)}]
The predual counterpart of Corollary \ref{p=8} is still unclear so far;
see Question \ref{openQ-HK} below for more details.
\end{itemize}
\end{remark}

Obviously, $JN_{(p,q,0)_0}(Q_0)$ is just $JN_{p,q}(Q_0)$.
From this and \cite[Proposition 5.1]{DHKY18},
we deduce that, when $p\in(1,\fz)$ and $q\in [1,p)$,  $JN_{(p,q,0)_0}(Q_0)$
coincides with $JN_p(Q_0)$ in the sense of equivalent norms and,
when $p\in(1,\fz)$ and $q\in [p,\infty)$,
$JN_{(p,q,0)_0}(Q_0)$ and $L^q(Q_0)$ coincide as sets.
Moreover, by adding a particular weight of $|Q_0|$,
the authors of this article showed that aforementioned coincidence (as sets)
can be modified into equivalent norms;
see Proposition \ref{JNpq=Lq} below, which is just \cite[Proposition 2.5]{tyyNA}.
In what follows, for any given positive constant $A$ and
any given function space $(\mathbb{X},\|\cdot\|_{\mathbb{X}})$,
we write $A\mathbb{X}:=\{Af:\,\,f\in\mathbb{X}\}$
with its norm defined by setting, for any $Af\in A\mathbb{X}$,
$\|Af\|_{A\mathbb{X}}:=A\|f\|_{\mathbb{X}}$.

\begin{proposition}\label{JNpq=Lq}
Let $p\in[1,\fz)$, $q\in[p,\fz)$, $s\in\zz_+$, $\alpha=0$,
and $Q_0$ be a cube of $\rn$.
Then
$$\lf[|Q_0|^{-\frac1p} JN_{(p,q,s)_\alpha}(Q_0)\r]
=\lf[L^q(Q_0,|Q_0|^{-1}dx)/\mathcal{P}_s(Q_0)\r]$$
with equivalent norms, namely,
\begin{align*}
\|f\|_{L^q(Q_0,|Q_0|^{-1}dx)/\mathcal{P}_s(Q_0)}
&\le|Q_0|^{-\frac1p}\|f\|_{JN_{(p,q,s)_0}(Q_0)}\\
&\le2^{p-\frac pq}\lf[1+C_{(s)}\r]^\frac pq
\|f\|_{L^q(Q_0,|Q_0|^{-1}dx)/\mathcal{P}_s(Q_0)},
\end{align*}
where $C_{(s)}$ is as in \eqref{Cs}.
\end{proposition}

It is a very interesting open question to find a counterpart of
Proposition \ref{JNpq=Lq} when $\alpha\in\rr\setminus\{0\}$;
see Question \ref{openQ-JN} below for more details.

Now, we review the predual of the John--Nirenberg--Campanato space
via introducing atoms, polymers, and Hardy-type spaces in order,
which coincide with the same notation as in \cite{DHKY18}
when $u\in(1,\fz)$, $v\in(u,\fz]$, and $\az=0=s$;
see \cite[Remarks 3.4 and 3.8]{tyyNA} for more details.
In particular, when $\az=0$, the $(u,v,s)_0$-atom below is just
the classic atom of the Hardy space; see \cite[Remark 3.2]{tyyNA}.

\begin{definition}\label{uvas-atom}
Let $u$, $v\in[1,\infty]$, $s\in\zz_+$, and $\alpha\in\rr$.
A function $a$ is called a \emph{$(u,v,s)_\alpha$-atom} on a cube $Q$ if
\begin{enumerate}
\item[{\rm(i)}]
$\supp (a):=\{x\in\rn:\,\,a(x)\neq0\}\subset Q$;
\item[{\rm(ii)}]
$\|a\|_{L^v(Q)}\le |Q|^{\frac1v-\frac1u-\alpha}$;
\item[{\rm(iii)}]
$\int_Q a(x)x^\gz\,dx=0$ for any $\gz\in\zz_+^n$ with $|\gz|\le s$.
\end{enumerate}
\end{definition}

In what follows, for any $u\in[1,\infty]$, let $u'$ denote its \emph{conjugate index},
namely, $1/u+1/{u'}=1$, and, for any $\{\lz_j\}_j\subset\cc$, let
\begin{align}\label{lp-norm}
\lf\|\{\lz_j\}_j\r\|_{\ell^u}:=
\begin{cases}
	\lf(\dis\sum_j|\lz_j |^u\right)^\frac1u &{\rm when}\ u\in[1,\fz),\\
	\dis\sup_j |\lz_j | &{\rm when}\ u=\fz.
\end{cases}
\end{align}

\begin{definition}\label{uvas-polymer}
Let $u$, $v\in[1,\infty]$, $s\in\zz_+$, and $\alpha\in\rr$.
The \emph{space of $(u,v,s)_\alpha$-polymers},
denoted by $\widetilde{HK}_{(u,v,s)_\alpha}(\cx)$,
is defined to be the set of all $g\in(JN_{(u',v',s)_\alpha}(\cx))^\ast$
satisfying that there exist $(u,v,s)_\alpha$-atoms $\{a_j\}_j$ supported,
respectively, in interior pairwise disjoint cubes $\{Q_j\}_j$ of $\cx$,
and $\{\lz_j\}_j\subset\cc$ with $|\lambda_j|^u<\fz$ such that
$$g=\sum_j \lambda_j a_j$$
in $(JN_{(u',v',s)_\alpha}(\cx))^\ast$.
Moreover, any $g\in\widetilde{HK}_{(u,v,s)_\alpha}(\cx)$ is called a
\emph{$(u,v,s)_\az$-polymer} with its norm
$\|g\|_{\widetilde{HK}_{(u,v,s)_\alpha}(\cx)}$ defined by setting
$$\|g\|_{\widetilde{HK}_{(u,v,s)_\alpha}(\cx)}
:=\inf\lf\|\{\lz_j\}_j\r\|_{\ell^u},$$
where the infimum is taken over all decompositions of $g$ as above.
\end{definition}

\begin{definition}\label{HKuvb}
Let $u$, $v\in[1,\infty]$, $s\in\zz_+$, and $\alpha\in\rr$.
The \emph{Hardy-type space} $HK_{(u,v,s)_\alpha}(\cx)$ is defined by setting
\begin{align*}
HK_{(u,v,s)_\alpha}(\cx):=&\lf\{g\in(JN_{(u',v',s)_\alpha}(\cx))^\ast:\,\,
g=\sum_i g_i \,\,\mathrm{in}\,\, (JN_{(u',v',s)_\alpha}(\cx))^\ast,\,\,
\r.\\
&\quad\quad\quad\lf.\{g_i\}_i\subset\widetilde{HK}_{(u,v,s)_\alpha}(\cx),\ \ \mathrm{and\,\,}
\sum_i\lf\|g_i\r\|_{\widetilde{HK}_{(u,v,s)_\alpha}(\cx)}<\infty\r\}
\end{align*}
and, for any $g\in HK_{(u,v,s)_\alpha}(\cx)$, let
$$\|g\|_{HK_{(u,v,s)_\alpha}(\cx)}:=\inf\sum_i
\|g_i\|_{\widetilde{HK}_{(u,v,s)_\alpha}(\cx)},$$
where the infimum is taken over all decompositions of $g$ as above.
Moreover, the \emph{finite atomic Hardy-type space
$HK_{(u,v,s)_\alpha}^{\mathrm{fin}}(\mathcal{X})$} is defined to be
the set of all finite summations $\sum_{m=1}^M \lambda_{m}a_{m}$,
where $M\in\nn$, $\{\lambda_{m}\}_{m=1}^M\subset\cc$, and
$\{a_{m}\}_{m=1}^M$ are $(u,v,s)_\alpha$-atoms.
\end{definition}

The significant dual relation between $JN_{(p,q,s)_\alpha}(\cx)$ and
$HK_{(p',q',s)_\alpha}(\cx)$ reads as follows, which is just
\cite[Theorem 3.9]{tyyNA} with $\az\in[0,\fz)$ replaced by $\az\in\rr$
(this makes sense because the crucial lemma, \cite[Lemma 3.12]{tyyNA},
still holds true with the corresponding replacement).

\begin{theorem}\label{duality}
Let $p$, $q\in (1,\fz)$, $1/p=1/p'=1=1/q+1/q'$, $s\in\zz_+$, and $\alpha\in\rr$.
Then $(HK_{(p',q',s)_\alpha}(\cx))^\ast=JN_{(p,q,s)_\alpha}(\cx)$
in the following sense:
\begin{enumerate}
\item[{\rm(i)}]If $f\in JN_{(p,q,s)_\alpha}(\cx)$, then $f$ induces a linear
functional $\mathcal{L}_f$ on $HK_{(p',q',s)_\alpha}(\cx)$ and
$$\|\mathcal{L}_f\|_{(HK_{(p',q',s)_\alpha}(\cx))^\ast}
  \le C\|f\|_{JN_{(p,q,s)_\alpha}(\cx)},$$
where $C$ is a positive constant independent of $f$.
\item[{\rm(ii)}] If $\mathcal{L}\in(HK_{(p',q',s)_\alpha}(\cx))^\ast$,
then there exists an $f\in JN_{(p,q,s)_\alpha}(\cx)$ such that,
for any $g\in HK_{(p',q',s)_\alpha}^{\mathrm{fin}}(\cx)$,
$$\mathcal{L}(g)=\int_{\cx}f(x)g(x)\,dx,$$
and
$$\|\mathcal{L}\|_{(HK_{(p',q',s)_\alpha}(\cx))^\ast}
\sim\|f\|_{JN_{(p,q,s)_\alpha}(\cx)}$$
with the positive equivalence constants independent of $f$.
\end{enumerate}
\end{theorem}

When $\cx:=Q_0$, $\az=0=s$, and $q\in[1,p)$,
by \cite[Remark 3.10]{tyyNA} and Proposition \ref{JNpqa=JNp1a},
we know that Theorem \ref{duality} in this case coincides with \cite[Theorem 6.6]{DHKY18}.
As an application of Theorem \ref{duality},
the authors obtained the following atomic characterization of $L^{q'}_s(Q_0)$
for any given $q'\in(1,\fz)$ and $s\in\zz_+$,
which is just \cite[Corollary 3.13]{tyyNA}.

\begin{proposition}\label{HKqq=Lq}
Let $p\in(1,\fz)$, $q\in[p,\fz)$, $1/p=1/p'=1=1/q+1/q'$, $s\in\zz_+$,
and $Q_0$ be a cube of $\rn$. Then
$$L^{q'}_s(Q_0,|Q_0|^{q'-1}dx)
=|Q_0|^\frac1p HK_{(p',q',s)_0 }(Q_0)$$
with equivalent norms.
\end{proposition}

From Theorem \ref{JN61Lem3} and \cite[p.\,14, Exercise 1.1.11]{GTM249},
we deduce that, for any $1<p_1<p_2<\fz$,
\begin{align*}
	JN_{p_2}(Q_0)\subset L^{p_2,\fz}(Q_0)\subset L^{p_1}(Q_0)\subset JN_{p_1}(Q_0).
\end{align*}
Moreover, it is easy to show following monotonicity over the first sub-index
of both $JN_{(p,q,s)_\alpha}(Q_0)$ and $HK_{(u,v,s)_\az}(Q_0)$.

\begin{proposition}\label{mono}
	Let $s\in\zz_+$ and $Q_0$ be a cube of $\rn$.
	\begin{enumerate}
		\item[{\rm(i)}] Let $1<u_1<u_2<\fz$. If $v\in(1,\fz)$ and  $\alpha\in\rr$,
		or $v=\fz$ and $\az\in[0,\fz)$, then
		$$HK_{(u_2,v,s)_\az}(Q_0)\subset HK_{(u_1,v,s)_\az}(Q_0)$$
		and
		$$\|\cdot\|_{HK_{(u_1,v,s)_\az}(Q_0)}\le
		|Q_0|^{\frac1{u_1}-\frac1{u_2}}\|\cdot\|_{HK_{(u_2,v,s)_\az}(Q_0)}.$$
		
		\item[{\rm(ii)}] Let $1<p_1<p_2<\fz$. If $q\in(1,\fz)$ and $\az\in\rr$,
		or $q=1$ and $\az\in[0,\fz)$, then
		$$JN_{(p_2,q,s)_\az}(Q_0)\subset JN_{(p_1,q,s)_\az}(Q_0)$$
		and there exists some positive constant $C$ such that
		$$\|\cdot\|_{JN_{(p_1,q,s)_\az}(Q_0)}\le C
		|Q_0|^{\frac1{p_1}-\frac1{p_2}}\|\cdot\|_{JN_{(p_2,q,s)_\az}(Q_0)}.$$
	\end{enumerate}
\end{proposition}

\begin{proof}
	(i) is a direct corollary of the fact that, for any $(u_2,v,s)_\az$-atom
	$a$ on the cube $Q$,
	$$|Q|^{\frac1{v_2}-\frac1{v_1}}a$$
	is a $(u_1,v,s)_\az$-atom;
	see \cite[Remark 5.5]{SXY19} for more details.
	
	(ii) is a direct consequence of the Jensen inequality;
	see, for instance, \cite[Remark 4.2(ii)]{tyyNA}.
	This finishes the proof of Proposition \ref{mono}.
\end{proof}

Now, we consider the independence over the second sub-index,
which strongly relies on the John--Nirenberg inequality as in the BMO case.
The following John--Nirenberg-type inequality is just \cite[Theorem 4.3]{tyyNA},
which coincides with Theorem \ref{JN61Lem3} when $\az=0=s$.

\begin{theorem}\label{JohnNirenberg}
Let $p\in(1,\infty)$, $s\in\zz_+$, $\alpha\in[0,\infty)$, and $Q_0$
be a cube of $\rn$.
If $f\in JN_{(p,1,s)_\alpha}(Q_0)$, then
$f-P_{Q_0}^{(s)}(f)\in L^{p,\infty}(Q_0)$ and there exists a positive constant
$C_{(n,p,s)}$, depending only on $n$, $p$, and $s$, but independent of $f$, such that
\begin{align*}
\lf\|f-P_{Q_0}^{(s)}(f)\r\|_{L^{p,\infty}(Q_0)}\le C_{(n,p,s)} \lf| Q_0 \r|^\alpha
 \|f\|_{JN_{(p,1,s)_\alpha}(Q_0)}.
\end{align*}
\end{theorem}
It should be mentioned that the main tool used
in the proof of Theorem \ref{JohnNirenberg}
is the following \emph{good-$\lambda$ inequality}
(namely, Lemma \ref{goodlambda} below)
which is just \cite[Lemma 4.6]{tyyNA};
see also \cite[Lemma 4.5]{ABKY11} when $s=0$.
Recall that, for any given cube $Q_0$ of $\rn$,
the \emph{dyadic maximal operator} $\mathcal{M}_{Q_0}^{({\rm d})}$
is defined by setting, for any given $g\in L^1(Q_0)$ and any $x\in Q_0$,
$$\mathcal{M}_{Q_0}^{({\rm d})}(g)(x)
:=\sup_{Q\in\cd_{Q_0},\,Q\ni x}\frac1{|Q|}\int_{Q}|g(x)|\,dx,$$
where $\cd_{Q_0}$ is as in \eqref{dyadicQ0} with $Q$ replaced by $Q_0$,
and the supremum is taken over all dyadic cubes $Q\in\mathcal{D}_{Q_0}$
and $Q\ni x$.
\begin{lemma}\label{goodlambda}
Let $p\in(1,\infty)$, $s\in\zz_+$, $C_{(s)}\in[1,\fz)$ be as in \eqref{Cs},
$\theta\in(0,2^{-n}C_{(s)}^{-1})$,
$Q_0$ be a cube of $\rn$, and $f\in JN_{(p,1,s)_0}(Q_0)$.
Then, for any real number $\lambda>\frac1\theta\fint_{Q_0}|f-P_{Q_0}^{(s)}(f)|$,
\begin{align*}
&\lf| \lf\{ x\in Q_0:\,\,\mathcal{M}_{Q_0}^{({\rm d})}\lf(f-P_{Q_0}^{(s)}(f)\r)(x)
>\lambda \r\} \r|\noz\\
&\quad \le \frac{\|f\|_{JN_{(p,1,s)_0}(Q_0)}}{[1-2^n\tz C_{(s)}]\lz}
\lf| \lf\{ x\in Q_0:\,\,\mathcal{M}_{Q_0}^{({\rm d})}\lf(f-P_{Q_0}^{(s)}(f)\r)(x)
>\theta\lambda \r\} \r|^\frac1{p'}.
\end{align*}
\end{lemma}

Moreover, based on Theorem \ref{JohnNirenberg}, in \cite[Proposition 4.1]{tyyNA},
Tao et al. further obtained the following independence
over the second sub-index of $JN_{(p,q,s)_\alpha}(\cx)$.

\begin{proposition}\label{JNpqa=JNp1a}
Let $1\le q<p<\infty$, $s\in\zz_+$, and $\alpha\in[0,\infty)$.
Then
$$JN_{(p,q,s)_\alpha}(\cx)=JN_{(p,1,s)_\alpha}(\cx)$$
with equivalent norms.
\end{proposition}

Furthermore, the following independence over the second sub-index of
$HK_{(u,v,s)_\alpha}(\cx)$ is just \cite[Proposition 4.7]{tyyNA},
whose proof is based on Theorem \ref{duality} and Proposition \ref{JNpqa=JNp1a}.

\begin{proposition}\label{HKuvb=HKu8b}
Let $1<u<v\le\infty$, $s\in\zz_+$, and $\alpha\in[0,\fz)$.
Then
$$HK_{(u,v,s)_\alpha}(\cx)=HK_{(u,\infty,s)_\alpha}(\cx)$$
with equivalent norms.
\end{proposition}

In particular, when $\az=0=s$, Propositions \ref{JNpqa=JNp1a} and \ref{HKuvb=HKu8b}
were obtained, respectively, in \cite[Propositions 5.1 and 6.4]{DHKY18}.

Combining Theorem \ref{duality} and Propositions \ref{JNpqa=JNp1a} and \ref{HKuvb=HKu8b},
we immediately have the following corollary; we omit the details here.

\begin{corollary}\label{dual-q=1}
Let $p\in(1,\fz)$, $s\in\zz_+$, and $\alpha\in[0,\infty)$.
Then $(HK_{(p',\fz,s)_\alpha}(\cx))^\ast=JN_{(p,1,s)_\alpha}(\cx)$.
\end{corollary}

Finally, we list some open questions.
\begin{question}\label{openQ-JN}
For any given cube $Q_0$ of $\rn$, by \cite[Remark 4.2(ii)]{tyyNA}
with slight modifications, we know that
\begin{itemize}
\item[{\rm(i)}] for any given $p\in[1,\fz)$ and $s\in\zz_+$,
$$JN_{(p,q,s)_0}(Q_0)=
\begin{cases}
JN_{(p,1,s)_0}(Q_0),& q\in [1,p),\\
JN_{(q,q,s)_0}(Q_0),& q\in [p,\fz);
\end{cases}$$

\item[{\rm(ii)}] for any given $p\in[1,\fz)$, $q\in [p,\fz)$,
$s\in\zz_+$, and $\az\in\rr$,
$$JN_{(q,q,s)_\az}(Q_0)\subset JN_{(p,q,s)_\az}(Q_0)$$
and
\begin{align*}
\lf[|Q_0|^{-\frac1p}\|f\|_{JN_{(p,q,s)_\az}(Q_0)}\r]
\le \lf[|Q_0|^{-\frac1q}\|f\|_{JN_{(q,q,s)_\az}(Q_0)}\r];
\end{align*}

\item[{\rm(iii)}] for any given $p\in[1,\fz)$, $q\in [p,\fz)$, $s\in\zz_+$, and
$\alpha\in (\frac{s+1}n,\infty)$,
$$JN_{(q,q,s)_\az}(Q_0)=\mathcal{P}_s(Q_0)=JN_{(p,q,s)_\az}(Q_0).$$
\end{itemize}
However, letting $RM_{p,q,\az}(\cx)$ denote the Riesz--Morrey space
in Definition \ref{def-Riesz}, it is still unknown whether or not
\begin{itemize}
\item[{\rm(i)}]for any given $p\in[1,\fz)$, $q\in [p,\fz)$, $s\in\zz_+$, and
$\alpha\in(-\fz,\frac{s+1}n]\setminus\{0\}$,
$$JN_{(p,q,s)_\alpha}(Q_0)=JN_{(q,q,s)_\alpha}(Q_0)
\ {\rm or}\
JN_{(p,q,s)_\alpha}(Q_0)=\lf[RM_{p,q,\az}(Q_0)/\mathcal{P}_s(Q_0)\r]$$
holds true;

\item[{\rm(ii)}] for any given $p\in[1,\fz)$, $q\in [p,\fz)$, $s\in\zz_+$, and
$\alpha\in\rr$,
$$JN_{(p,q,s)_\alpha}(\rn)=JN_{(q,q,s)_\alpha}(\rn)
\ {\rm or}\
JN_{(p,q,s)_\alpha}(\rn)=\lf[RM_{p,q,\az}(\rn)/\mathcal{P}_s(\rn)\r]$$
holds true, where $\mathcal{P}_s(\rn)$ denotes the set of all
polynomials of degree not greater than $s$ on $\rn$.
\end{itemize}
\end{question}

\begin{question}\label{openQ-HK}
Let $1<u_1<u_2<\fz$, $v\in(1,\fz]$, $s\in\zz_+$,
and $Q_0$ be a cube of $\rn$.
From Proposition \ref{mono}(i), we deduce that
$$HK_{(u_2,v,s)_0}(Q_0)\subset HK_{(u_1,v,s)_0}(Q_0)$$
and
$$\|\cdot\|_{HK_{(u_1,v,s)_0}(Q_0)}\le
\lf[|Q_0|^{\frac1{u_1}-\frac1{u_2}}\|\cdot\|_{HK_{(u_2,v,s)_0}(Q_0)}\r].$$
Moreover, by \cite[Remark 4.2(iii)]{tyyNA} and \cite[Proposition 5.7]{SXY19},
we find that, for any $u\in[1,\fz)$,
$$HK_{(u,v,s)_0}(Q_0)\subset H^{1,v,s}_{\mathrm{at}}(Q_0)$$
and, for any $g\in\bigcup_{u\in[1,\fz)}HK_{(u,v,s)_0}(Q_0)$,
\begin{align*}
\|g\|_{H^{1,v,s}_{\mathrm{at}}(Q_0)}
\le \liminf_{u\to1^+} \|g\|_{HK_{(u,v,s)_0}(Q_0)},
\end{align*}
where $H^{1,v,s}_\at(\cx)$ denotes the \emph{atomic Hardy space};
see Coifman and Weiss \cite{CW77}, and also \cite[Remark 3.2(ii)]{tyyNA},
for its definition.
Here and thereafter, $u\to1^+$ means $u\in(1,\fz)$ and $u\to1$.
However, for any given $v\in(1,\fz]$, $s\in\zz_+$, $\alpha\in[0,\fz)$, and
any given cube $Q_0$ of $\rn$,
\begin{itemize}
\item[\rm{(i)}] it is still unknown whether or not,
for any $g\in\bigcup_{u\in[1,\fz)}HK_{(u,v,s)_\az}(Q_0)$,
$$\|g\|_{H^{\frac1{\az+1},v,s}_{\mathrm{at}}(Q_0)}
=\lim_{u\to1^+} \|g\|_{HK_{(u,v,s)_\az}(Q_0)}$$
holds true;

\item[\rm{(ii)}]
it is interesting to clarify the relation between
$\bigcup_{u\in[1,\fz)}HK_{(u,v,s)_\az}(Q_0)$
and $H^{\frac1{\az+1},v,s}_{\mathrm{at}}(Q_0)$.
\end{itemize}
\end{question}

The last question in this subsection is on an interpolation result in \cite{S65}.
We first recall some notation in \cite{S65}.
Let $p\in(1,\fz)$, $\lambda\in\rr$, and $Q_0$ be a cube of $\rn$.
The \emph{space} $N^{(p,\lambda)}(Q_0)$ is defined by setting
$$N^{(p,\lambda)}(Q_0):=\lf\{u\in L^1(Q_0):\,\,[u]_{N^{(p,\lambda)}(Q_0)}<\fz\r\},$$
where
$$[u]_{N^{(p,\lambda)}}(Q_0):=\sup \lf\{\sum_i\lf|\int_{Q_i}|u(x)-u_{Q_i}|\,dx \r|^p
|Q_i|^{1-p-\lambda} \r\}^{1/p}$$
and the supremum is taken over all collections of
interior pairwise disjoint cubes $\{Q_i\}_i$ of $Q_0$,
and $u_{Q_i}$ is the mean of $u$ over $Q_i$ for any $i$.
Let $\mathcal{F}(Q_0)$ denote the set of all simple functions on $Q_0$.
\begin{definition}(\cite[Definition 3.1]{S65})\label{st}
A linear operator $T$ defined on $\mathcal{F}(Q_0)$ is said to be of \emph{strong type}
$N[p,(q,\mu)]$ if there exists a positive constant $K$ such that,
for any $u\in \mathcal{F}(Q_0)$,
$$[Tu]_{N^{(q,\mu)}(Q_0)}\le K\|u\|_{L^p(Q_0)};$$
the smallest of the constant $K$ for which the above inequality holds true is called the
\emph{strong $N[p,(q,\mu)]$-norm}.
\end{definition}

\begin{theorem}(\cite[Theorem 3.1]{S65})\label{interpolation-thm}
Let $[p_i,q_i,\mu_i]$ be real numbers such that $p_i$, $q_i\in[1,\fz)$ for any $i\in\{1,2\}$.
If $T$ is a linear operator which is simultaneously of strong type $N[p_i,(q_i,\mu_i)]$
with respective norms $K_i$ $(i\in\{1,2\})$, then $T$ is of strong type $N[p_t,(q_t,\mu)]$ where
$$\begin{cases}
\displaystyle{\frac1{p_t}:=\frac{1-t}{p_1}+\frac{t}{p_2},\quad \frac1{q_t}:=\frac{1-t}{q_1}+\frac{t}{q_2}}\\
\displaystyle{\frac{\mu}{q}=(1-t)\frac{\mu_1}{q_1} \quad\mathrm{for}\quad t\in[0,1].}
\end{cases}$$
Moreover, for any $t\in[0,1]$,
$$[Tu]_{N[p_t,(q_t,\mu)]}\le K_1^{1-t}K_2^t \|u\|_{L^p(Q_0)}.$$
The theorem holds true also in the limit case $p_1=\fz$ and $\frac1{q_1}=\mu_1=0$.
\end{theorem}

\begin{question}\label{openQ-interpolation}
In the proof of Theorem \ref{interpolation-thm}, lines 1-3 of \cite[p.\,454]{S65},
the author applied \cite[Lemma 2.3]{S65} with
$$F[u,v,S]:=\sum_i\int_{Q_i}\lf[u(y)-u_{Q_i}\r]v\,dy |Q_i|^{-\lambda/p_t}$$
replaced by
$$\Phi(S,t):=\sum_i\int_{Q_i}\lf[ T(\widetilde{u}(y,t))
-\lf(T\widetilde{u}\r)_{Q_i}\r]\widetilde{v}(y,t)\,dy
\lf| Q_i\r|^{-\mu(t)\beta(t)}.$$
Therefore, by the proof of \cite[Lemma 2.3]{S65},
we need to choose a function $\widetilde{v}$
satisfying that, for any $i$, there exists some constant $c_i$ such that
\begin{align}\label{interpolation-1}
\widetilde{v}(y,t)=c_i\overline{\lf\{\mathrm{sign} \lf[T(\widetilde{u}(y,t))
-\lf(T\widetilde{u}\r)_{Q_i}\r]\r\}}
\end{align}
in $Q_i$.
Meanwhile, from the definition of $\widetilde{v}$
(see line -3 of \cite[p.\,452]{S65}),
it follows that
\begin{align}\label{interpolation-2}
  \widetilde{v}(y,t)=|v(y)|^{[1-\beta(t)]q_t'}e^{i\mathrm{arg}v(y)}
\end{align}
for some simple function $v\in\mathcal{F}(Q_0)$, where $1/q_t+1/q_t'=1$.
To sum up, we need to find a simple function $v$ such that
both \eqref{interpolation-1} and \eqref{interpolation-2} hold true,
which seems unreasonable because
$T\widetilde{u}$ may behave so bad even though both $u$ and
$\widetilde{u}$ are simple functions.
Thus, the proof of Theorem \ref{interpolation-thm} in \cite{S65}
seems problematic.
It is interesting to check whether or not
Theorem \ref{interpolation-thm} is really true.
\end{question}

\subsection{Localized John--Nirenberg--Campanato spaces}\label{Subsec-LJNC}

As a combination of the JNC space and
the localized BMO space in Subsection \ref{Subsec-BMO},
Sun et al. \cite{SXY19} studied the localized
John--Nirenberg--Campanato space,
which is new even in a special case: localized John--Nirenberg spaces.
Now, we recall the definition of the localized Campanato space,
which was first introduced by Goldberg in \cite[Theorem 5]{G79}.
In what follows, for any $s\in\zz_{+}$ and $c_0\in(0,\ell(\cx))$, let
$$ P^{(s)}_{Q,c_0}(f):=
\begin{cases}
P^{(s)}_{Q}(f),&\ell(Q)< c_0,\\
0,&\ell(Q)\ge c_0,
\end{cases}$$
where $P^{(s)}_{Q}(f)$ is as in \eqref{PQsf}.

\begin{definition}
\label{def.cam}
Let $q\in[1,\fz)$, $s\in\zz_{+}$, and $\az\in[0,\fz)$.
Fix $c_0\in (0,\ell(\cx))$.
The \emph{local Campanato space}
$\Lambda_{(\az,q,s)}(\cx)$ is defined to be the set of all functions
$f \in L^q_{\loc}(\cx)$ such that
$$\|f\|_{\Lambda_{(\az,q,s)}(\cx)}
:=\sup\lf|Q\r|^{-\az}\lf[\fint_Q\lf|f(x)
-P^{(s)}_{Q,c_0}(f)(x)\r|^q\,dx\r]^{\frac{1}{q}}<\fz,$$
where the supremum is taken over all cubes $Q$ of $\cx$.
\end{definition}

Fix the constant $c_0\in(0,\ell(\cx))$.
In Definition \ref{JNpqa},
if $P^{(s)}_{Q_j}(f)$ were replaced by $P^{(s)}_{Q_j,c_0}(f)$,
then we obtain the following localized John--Nirenberg--Campanato space.
As was mentioned in Remark \ref{rem-ainrr},
we naturally extend the ranges of $\az$ and $p$ similarly to
Subsection \ref{Subsec-JNC}; we omit some identical proofs.

\begin{definition}\label{def.jnpqs}
Let $p$, $q\in[1,\fz)$, $s\in\zz_{+}$, and $\az\in\rr$.
Fix the constant $c_0\in(0,\ell(\cx))$.
The \emph{local John--Nirenberg--Campanato space}
$jn_{(p,q,s)_{\az,c_0}}(\cx)$ is defined to be the set of all functions $f\in L^q_{\loc}(\cx)$ such that
$$\|f\|_{jn_{(p,q,s)_{\az,c_0}}(\cx)}
:=\sup\lf[\sum_{j\in\nn}\lf|Q_j\r|\lf\{\lf|Q_j\r|^{-\az}
\lf[\fint_{Q_j}\lf|f(x)-P^{(s)}_{Q_j,c_0}(f)(x)\r|^q\,dx\r]^{\frac{1}{q}}\r\}^p\r]^{\frac{1}{p}}$$
is finite,
where the supremum is taken over all collections of interior pairwise disjoint cubes
$\{Q_j\}_{j\in\nn}$ of $\cx$.
Moreover, the \emph{dual space $(jn_{(p,q,s)_{\az,c_0}}(\cx))^\ast$}
of $jn_{(p,q,s)_{\az,c_0}}(\cx)$
is defined to be the set of all continuous linear functionals on
$jn_{(p,q,s)_{\az,c_0}}(\cx)$
equipped with the weak-$\ast$ topology.
\end{definition}

\begin{remark}\label{quotient}
Notice that the Campanato space and the John--Nirenberg--Campanato space
are quotient spaces, while their localized versions are not.
\end{remark}

Also, in \cite[Proposition 2.5]{SXY19}, Sun et al. showed that
$jn_{(p,q,s)_{\az,c_0}}(\cx)$ in Definition \ref{def.jnpqs}
is independent of the choice of the positive constant $c_0$.
Therefore, in what follows, we write $$jn_{(p,q,s)_{\az}}(\cx):=jn_{(p,q,s)_{\az,c_0}}(\cx).$$
Especially, if $q=1$ and $s=0=\az$, then $jn_{(p,q,s)_{\az}}(\cx)$
becomes the \emph{local  John--Nirenberg space}
$$jn_{p}(\cx):=jn_{(p,1,0)_0}(\cx).$$
The following Banach structure of $jn_{(p,q,s)_{\az}}(\cx)$
is just \cite[Proposition 2.7]{SXY19}.
\begin{proposition}\label{prop.jncom}
Let $p$, $q\in[1,\fz)$, $s\in\zz_{+}$, and $\az\in\rr$.
Then $jn_{(p,q,s)_{\az}}(\cx)$ is a Banach space.
\end{proposition}
In what follows, the \emph{space $jn_{(p,q,s)_{\az}}(Q_0)/\mathcal{P}_s(Q_0)$} is defined by setting
$$jn_{(p,q,s)_{\az}}(Q_0)/\mathcal{P}_s(Q_0):=\lf\{f\in jn_{(p,q,s)_{\az}}(Q_0):\,\, \|f\|_{jn_{(p,q,s)_{\az}}(Q_0)/\mathcal{P}_s(Q_0)}<\fz\r\},$$
where
$$\|f\|_{jn_{(p,q,s)_{\az}}(Q_0)/\mathcal{P}_s(Q_0)}:=\inf_{a\in \mathcal{P}_s(Q_0)}\|f+a\|_{jn_{(p,q,s)_\az}(Q_0)};$$
the \emph{space $JN_{(p,q,s)_{\az}}(\cx)\cap L^p(\cx)$}
is defined by setting
$$JN_{(p,q,s)_{\az}}(\cx)\cap L^p(\cx)
:=\lf\{f\in L^1_\loc(\cx):\
\|f\|_{JN_{(p,q,s)_{\az}}(\cx)\cap L^p(\cx)}<\fz\r\},$$
where
$$\|f\|_{JN_{(p,q,s)_{\az}}(\cx)\cap L^p(\cx)}
:=\max\lf\{\|f\|_{JN_{(p,q,s)_{\az}}(\cx)},\|f\|_{L^p(\cx)}\r\}.$$
Moreover, the relations between $jn_{(p,q,s)_{\az}}(\cx)$
and $JN_{(p,q,s)_{\az}}(\cx)$,
namely, the following Propositions \ref{rem.jnandJN} and \ref{p.a},
are just \cite[Propositions 2.9 and 2.10]{SXY19}, respectively.

\begin{proposition}\label{rem.jnandJN}
Let $p$, $q\in[1,\fz)$, $s\in\zz_{+}$, and $\az\in\rr$. Then
\begin{itemize}
\item[\rm{(i)}]
$jn_{(p,q,s)_{\az}}(\cx)\subset JN_{(p,q,s)_{\az}}(\cx)$;
\item[\rm{(ii)}] if $Q_0$ is a cube of $\rn$, then
$JN_{(p,q,s)_{\az}}(Q_0)=jn_{(p,q,s)_{\az}}(Q_0)/\mathcal{P}_s(Q_0)$ with equivalent norms;
\item[\rm{(iii)}]
$L^p(\rr)\subsetneqq jn_p(\rr)\subsetneqq JN_p(\rr)$ if $p\in(1,\fz)$.
\end{itemize}
\end{proposition}

\begin{proposition}\label{p.a}
Let $p\in[1,\fz)$, $q\in[1,p]$, $s\in\zz_{+}$, and $\az\in(0,\fz)$. Then
\begin{align}\label{jnJNLp}
jn_{(p,q,s)_{\az}}(\cx)=\lf[JN_{(p,q,s)_{\az}}(\cx)\cap L^p(\cx)\r]
\end{align}
with equivalent norms.
\end{proposition}
Also, observe that Proposition \ref{p.a} is the counterpart of
\cite[Theorem 4.1]{JSW84} which says that,
for any $\az\in(0,\fz)$, $q\in [1,\fz)$, and $s\in\zz_{+}$,
\begin{align*}
\Lambda_{(\az,q,s)}(\cx)=\lf[\mathcal{C}_{(\az,q,s)}(\cx)\cap L^{\fz}(\cx)\r].
\end{align*}
However, the case $q\in[p,\fz)$ in Proposition \ref{p.a} is unclear so far;
see Question \ref{openQ-loc-q>p} below.

As an application of Propositions \ref{rem.jnandJN}(ii) and \ref{p.a},
we have the following result.
\begin{proposition}\label{JN<Lp}
Let $p\in[1,\fz)$, $q\in[1,p]$, $s\in\zz_{+}$, $\az\in(0,\fz)$,
and $Q_0$ be a cube of $\rn$.  Then
$$JN_{(p,q,s)_{\az}}(Q_0)\subset \lf[L^p(Q_0)/\mathcal{P}_s(Q_0)\r].$$
\end{proposition}
\begin{proof}
Let $p,\ q,\ s,\ \az$, and $Q_0$ be as in this proposition.
Then, by Propositions \ref{rem.jnandJN}(ii) and \ref{p.a},
we obtain
\begin{align*}
JN_{(p,q,s)_{\az}}(Q_0)
&=\lf[jn_{(p,q,s)_{\az}}(Q_0)/\mathcal{P}_s(Q_0)\r]\\
&=\lf\{JN_{(p,q,s)_{\az}}(Q_0)\cap \lf[L^p(Q_0)/\mathcal{P}_s(Q_0)\r]\r\}
\end{align*}
and
\begin{align*}
\|\cdot\|_{JN_{(p,q,s)_{\az}}(Q_0)}
&\sim\inf_{a\in \mathcal{P}_s(Q_0)}\|\cdot+a\|_{jn_{(p,q,s)_\az}(Q_0)}\\
&\sim\max\lf\{\|\cdot\|_{JN_{(p,q,s)_{\az}}(Q_0)},
\inf_{a\in \mathcal{P}_s(Q_0)}\|\cdot+a\|_{L^p(Q_0)}\r\}.
\end{align*}
This implies that
$JN_{(p,q,s)_{\az}}(Q_0)\subset [L^p(Q_0)/\mathcal{P}_s(Q_0)]$ with
$$\inf_{a\in \mathcal{P}_s(Q_0)}\|\cdot+a\|_{L^p(Q_0)}
\ls\|\cdot\|_{JN_{(p,q,s)_{\az}}(Q_0)},$$
which completes the proof of Proposition \ref{JN<Lp}.
\end{proof}

Propositions \ref{prop.jncam} and \ref{prop.jncam2} below
are just, respectively, \cite[Propositions 2.12 and 2.13]{SXY19}
which show that the localized Campanato space is the limit of
the localized John--Nirenberg--Campanato space.

\begin{proposition}\label{prop.jncam}
Let $q\in[1,\fz)$, $s\in\zz_{+}$, $\az\in[0,\fz)$,
and $Q_0$ be a cube of $\rn$. Then,
for any $f\in L^1(Q_0)$,
$$\|f\|_{\Lambda_{(\az,q,s)}(Q_0)}=\lim_{p\to\fz}\|f\|_{jn_{(p,q,s)_{\az}}(Q_0)}.$$
Moreover,
$${\Lambda_{(\az,q,s)}(Q_0)}=\lf\{f\in\bigcap_{p\in[1,\fz)}jn_{(p,q,s)_{\az}}(Q_0):
\,\,\lim_{p\to\fz}\|f\|_{jn_{(p,q,s)_{\az}}(Q_0)}<\fz\r\}.$$
\end{proposition}

\begin{proposition}\label{prop.jncam2}
Let $q\in[1,\fz)$, $s\in\zz_{+}$, and $\az\in[0,\fz)$.
Then
$$\lim_{p\to\fz}jn_{(p,q,s)_{\az}}(\rn)=\Lambda_{(\az,q,s)}(\rn)$$
in the following sense:
if $f\in jn_{(p,q,s)_{\az}}(\rn)\cap\Lambda_{(\az,q,s)}(\rn)$, then
$$f\in\bigcap_{r\in[p,\fz)}jn_{(r,q,s)_{\az}}(\rn)$$
and
$$\|f\|_{\Lambda_{(\az,q,s)}(\rn)}=\lim_{r\to\fz}\|f\|_{jn_{(r,q,s)_{\az}}(\rn)}.$$
\end{proposition}

As in Proposition \ref{JNpqa=JNp1a},
the following invariance of $jn_{(p,q,s)_{\az}}(\cx)$
on its indices in the appropriate range is just \cite[Proposition 3.1]{SXY19}.
\begin{proposition}\label{prop.jneq1}
Let $p\in(1,\fz)$, $q\in[1,p)$, $s\in\zz_{+}$, and $\az\in[0,\fz)$. Then
$$jn_{(p,q,s)_{\az}}(\cx)=jn_{(p,1,s)_{\az}}(\cx)$$ with equivalent norms.
\end{proposition}

In other range of indices, namely, $q\ge p$,
the following relation between $jn_{(p,q,s)_{\az}}(\cx)$
and the Lebesgue space is just \cite[Proposition 3.4]{SXY19}.

\begin{proposition}\label{prop.jneq2}
Let $s\in\zz_{+}$ and $Q_0$ be a cube of $\rn$.
\begin{itemize}
\item[\rm{(i)}]
If $1\le p\le q<\fz$, then
 $[|Q_0|^{\frac{1}{q}-\frac{1}{p}}jn_{(p,q,s)_{0}}(Q_0)]=L^q(Q_0)$
 with equivalent norms.
\item[\rm{(ii)}]
 If $p\in[1,\fz)$, then $jn_{(p,p,s)_{0}}(\rn)=L^p(\rn)$ with equivalent norms.
\item[\rm{(iii)}]
 If $p$, $q\in[1,\fz)$, $\az\in(-\fz,\frac{1}{p}-\frac{1}{q})$,
 and $f\in jn_{(p,q,s)_{\az}}(\rn)$, then $f=0$ almost everywhere.
 \end{itemize}
\end{proposition}

Using the localized atom, Sun el al. \cite{SXY19} introduced
the localized Hardy-type space and showed that this space is the predual of
the localized John--Nirenberg--Campanato space.
First, recall the definitions of localized atoms, localized polymers, and
localized Hardy-type spaces in order as follows.
\begin{definition}\label{def.atom}
Let $v$, $w\in[1,\fz]$, $s\in \zz_{+}$, and $\az\in\rr$. Fix $c_0\in(0,\ell(\cx))$ and
let $Q$ denote a cube of $\rn$. Then a function
$a$ on $\rn$ is called a \emph{local $(v,w,s)_{\az,c_0}$-atom} supported in $Q$ if
\begin{itemize}
\item[{\rm(i)}]
$\supp(a):=\{x\in\rn:\,\,a(x)\neq 0\}\subset Q$;
\item[{\rm(ii)}]
$\| a\|_{L^w(Q)}\le |Q|^{\frac{1}{w}-\frac{1}{v}-\az}$;
\item[{\rm(iii)}]
when $\ell(Q)<c_0$,
$\int_Qa(x)x^{\beta}dx=0$ for any $\beta\in\zz_{+}^n$ and $|\beta|\le s$.
\end{itemize}
\end{definition}

\begin{definition}\label{def.poly}
Let $v$, $w\in[1,\fz]$, $s\in \zz_{+}$, $\az\in\rr$,
and $c_0\in (0,\ell(\cx))$. The \emph{space}
$\widetilde{hk}_{(v,w,s)_{\az,c_0}}(\cx)$
is defined to be the set of all $g\in (jn_{(v',w',s)_{\az,c_0}}(\cx))^{\ast}$ such that
$$g=\sum_{j\in\nn}\lz_ja_j$$
in  $(jn_{(v',w',s)_{\az,c_0}}(\cx))^{\ast}$,
where $1/v+1/v'=1=1/w+1/w'$,
$\{a_j\}_{j\in\nn}$ are local $(v,w,s)_{\az,c_0}$-atoms supported, respectively,
in interior pairwise disjoint subcubes $\{Q_j\}_{j\in\nn}$ of $\cx$, and
$\{\lambda_j\}_{j\in\nn}\subset\mathbb{C}$ with
$\|\{\lz_j\}_{j\in\nn}\|_{\ell^v}<\fz$
[see \eqref{lp-norm} for the definition of $\|\cdot\|_{\ell^v}$].
Any $g\in\widetilde{hk}_{(v,w,s)_{\az,c_0}}(\cx)$
is called a \emph{local} $(v,w,s)_{\az,c_0}$\emph{-polymer} on $\cx$ and let
$$\| g\|_{\widetilde{hk}_{(v,w,s)_{\az,c_0}}(\cx)}
:=\inf\lf\|\{\lz_j\}_{j\in\nn}\r\|_{\ell^v},$$
where the infimum is taken over all decompositions of $g$ as above.
\end{definition}

\begin{definition}
\label{def.hkvws}
Let $v$, $w\in[1,\fz]$, $s\in \zz_{+}$, $\az\in\rr$,
and $c_0\in(0,\ell(\cx))$.
The \emph{local Hardy-type space $hk_{(v,w,s)_{\az,c_0}}(\cx)$}
is defined to be the set of all
$g\in (jn_{(v',w',s)_{\az,c_0}}(\cx))^{\ast}$ such that
there exists a sequence
$\{g_i\}_{i\in\nn}\subset \widetilde{hk}_{(v,w,s)_{\az,c_0}}(\cx)$ such that
$\sum_{i\in\nn}\|g_i\|_{\widetilde{hk}_{(v,w,s)_{\az,c_0}}(\cx)}<\fz$ and
\begin{align}\label{hc}
g=\sum_{i\in\nn}g_i
\end{align}
in $(jn_{(v',w',s)_{\az,c_0}}(\cx))^{\ast}$.
For any $g\in{hk_{(v,w,s)_{\az,c_0}}(\cx)}$, let
$$\|g\|_{hk_{(v,w,s)_{\az,c_0}}(\cx)}:=\inf \sum_{i\in\nn}\|g_i\|_{\widetilde{hk}_{(v,w,s)_{\az,c_0}}(\cx)},$$
where the infimum is taken over all decompositions of $g$ as in \eqref{hc}.
\end{definition}

Correspondingly, $hk_{(v,w,s)_{\az,c_0}}(\cx)$ is independent of
the choice of the positive constant $c_0$ as well,
which is just \cite[Proposition 4.7]{SXY19}.

\begin{proposition}
\label{prop.hkc}
Let $v\in(1,\fz)$, $w\in(1,\fz]$, $s\in \zz_{+}$, $\az\in\rr$,
and $0<c_1<c_2<\ell(\cx)$.
Then $hk_{(v,w,s)_{\az,c_1}}(\cx)=hk_{(v,w,s)_{\az,c_2}}(\cx)$
with equivalent norms.
\end{proposition}

Henceforth, we simply write
$${\rm local}\ {(v,w,s)_{\az,c_0}}{\rm -atoms},\ \
\widetilde{hk}_{(v,w,s)_{\az,c_0}}(\cx),\ \
{\rm and}\ \ hk_{(v,w,s)_{\az,c_0}}(\cx),$$
respectively, as
$${\rm local}\ {(v,w,s)_{\az}}{\rm -atoms},\ \
\widetilde{hk}_{(v,w,s)_{\az}}(\cx),\ \
{\rm and}\ \ hk_{(v,w,s)_{\az}}(\cx).$$

The corresponding dual theorem (namely, Theorem \ref{theo.dual} below)
is just \cite[Theorem 4.11]{SXY19}.
In what follows, the \emph{space $hk^{\mathrm{fin}}_{(v,w,s)_{\az}}(\cx)$}
is defined to be the set of all finite linear combinations of local
$(v,w,s)_{\az}$-atoms supported, respectively, in cubes of $\cx$.

\begin{theorem}\label{theo.dual}
Let $v$, $w\in(1,\fz)$, $1/v+{1}/{v'}=1={1}/{w}+{1}/{w'}=1$, $s\in \zz_{+}$, and $\az\in\rr$.
Then $ jn_{(v',w',s)_{\az}}(\cx)=(hk_{(v,w,s)_{\az}}(\cx))^{\ast}$ in the following sense:
\begin{itemize}
\item[\rm{(i)}]
For any given $f\in jn_{(v',w',s)_{\az}}(\cx)$, the linear functional
$$\cl_f:\ g\longmapsto\lf\langle \cl_f,g\r\rangle
:=\int_{\cx}f(x)g(x)\,dx, \qquad\forall\,g\in hk^{\mathrm{fin}}_{(v,w,s)_{\az}}(\cx)$$
can be extended to a bounded linear functional on $hk_{(v,w,s)_{\az}}(\cx)$.
Moreover, it holds true that
$\| \cl_f\|_{(hk_{(v,w,s)_{\az}}(\cx))^{\ast}}
\le \|f\|_{jn_{(v',w',s)_{\az}}(\cx)}$.
\item[\rm{(ii)}]
Any bounded linear functional $ \cl$ on $hk_{(v,w,s)_{\az}}(\cx)$ can be represented by a function $f\in jn_{(v',w',s)_{\az}}(\cx)$
in the following sense:
\begin{align*}
\lf\langle \cl,g\r\rangle =\int_{\cx}f(x)g(x)\,dx,\quad
\forall\, g\in hk^{\mathrm{fin}}_{(v,w,s)_{\az}}(\cx).
\end{align*}
Moreover, there exists a positive constant $C$, depending only on $s$, such that
$\|f\|_{jn_{(v',w',s)_{\az}}(\cx)}\le C\|\cl \|_{(hk_{(v,w,s)_{\az}}(\cx))^{\ast}}$.
\end{itemize}
\end{theorem}

As a corollary of Theorem \ref{theo.dual}
as well as a counterpart of Proposition \ref{prop.jneq1},
for any admissible $(v,s,\az)$,
Proposition \ref{prop.hkeq1}, which is just \cite[Proposition 5.1]{SXY19},
shows that $hk_{(v,w,s)_{\az}}(\cx)$
is invariant on $w\in(v,\fz]$.

\begin{proposition}\label{prop.hkeq1}
Let $v\in(1,\fz)$, $w\in(v,\fz]$, $s\in \zz_{+}$, and $\az\in[0,\fz)$. Then
$$hk_{(v,w,s)_{\az}}(\cx)=hk_{(v,\fz,s)_{\az}}(\cx)$$ with equivalent norms.
\end{proposition}

The following proposition, which is just \cite[Proposition 5.6]{SXY19},
might be viewed as a counterpart of Proposition \ref{prop.jneq2}.

\begin{proposition}\label{prop.hkeq2}
Let $v\in(1,\fz)$ and $s\in\zz_{+}$.
\begin{itemize}
\item[{\rm(i)}]
If $w\in(1,v]$ and $Q_0$ is a cube of $\rn$, then
$hk_{(v,w,s)_{0}}(Q_0)=|Q_0|^{\frac{1}{v}-\frac{1}{w}}L^w(Q_0)$
with equivalent norms.
\item[{\rm{(ii)}}]
$L^v(\rn)=hk_{(v,v,s)_{0}}(\rn)$ with equivalent norms.
\end{itemize}
\end{proposition}

Finally, the following relation between $hk_{(v,w,s)_{\az}}(\cx)$
and the atomic localized Hardy space is just \cite[Proposition 5.7]{SXY19}.

\begin{proposition}\label{prop.hkandhp}
Let $w\in(1,\fz]$ and $Q_0$ be a cube of $\rn$.
Then
$$\bigcup_{v\in[1,\fz)}hk_{(v,w,0)_{0}}(Q_0)\subset h^{1,w}_{at}(Q_0).$$
Moreover, if $g\in \bigcup_{v\in[1,\fz)}hk_{(v,w,0)_0}(Q_0)$, then
$$\|g\|_{h^{1,w}_{at}(Q_0)}\le \liminf_{v\to 1^+}\|g\|_{hk_{(v,w,0)_0}(Q_0)},$$
where $v\to 1^+$ means that $v\in(1,\fz)$ and $v\to 1$.
\end{proposition}

We also list some open questions at the end of this subsection.

\begin{question}
	There still exist something \emph{unclear} in Proposition \ref{rem.jnandJN}(iii).
	Precisely, let $p\in(1,\fz)$,
	$$jn_p(\rr)/\cc:=\lf\{f\in L^1_\loc(\rr):\
	\|f\|_{jn_p(\rr)/\cc}:=\inf_{c\in\cc}\|f+c\|_{jn_p(\rr)}
	<\fz\r\}$$
	and
	$$L^p(\rr)/\cc:=\lf\{f\in L^1_\loc(\rr):\
	\|f\|_{L^p(\rr)/\cc}:=\inf_{c\in\cc}\|f+c\|_{L^p(\rr)}
	<\fz\r\}.$$
	Then it is still unknown whether or not
	$$\lf[jn_p(\rr)/\cc\r]\subsetneqq JN_p(\rr)$$
	holds true, namely, it is still unknown whether or not
	there exists some \emph{non-constant} function $h$ such that
	$h\in JN_p(\rr)$ but $h\notin jn_p(\rr)$.
	Moreover, it is still unknown whether or not
	$$\lf[L^p(\rn)/\cc\r]\subsetneqq \lf[jn_p(\rn)/\cc\r]\subsetneqq JN_p(\rn)$$
	holds true.
\end{question}
The following question is on the case $q>p$ corresponding to Proposition \ref{p.a}.
\begin{question}\label{openQ-loc-q>p}
Let $p\in[1,\fz)$, $q\in(p,\fz)$, $s\in\zz_{+}$, and $\az\in(0,\fz)$. Then
it is still unknown whether or not
$$jn_{(p,q,s)_{\az}}(\cx)=\lf[JN_{(p,q,s)_{\az}}(\cx)\cap L^p(\cx)\r]$$
still holds true.
\end{question}

Also, the corresponding localized cases of
Questions \ref{openQ-JN} and \ref{openQ-HK}
are listed as follows.
The following Question \ref{openQ-jn} is a modification of \cite[Remark 3.5]{SXY19},
and Question \ref{openQ-hk} is just \cite[Remark 5.8]{SXY19}.

\begin{question}\label{openQ-jn}
Let $p\in[1,\fz)$, $q\in[1,\fz)$, $s\in\zz_{+}$, and
$\az\in[\frac{1}{p}-\frac{1}{q},\fz)$.
Then the relation between
$jn_{(p,q,s)_{\az}}(\rn)$ and the Riesz--Morrey space
$RM_{p,q,\az}(\rn)$
(see Subsection \ref{Subsec-RM} for its definition)
is still \emph{unclear} so far,
except the identity
$$jn_{(p,p,s)_0}(\rn)=L^p(\rn)=RM_{p,p,0}(\rn)$$
due to Proposition \ref{prop.jneq2}(ii)
and Theorem \ref{thm-ClasRM}(ii),
and the inclusion
$$jn_{(p,q,s)_\az}(\rn)\supset RM_{p,q,\az}(\rn)
\quad{\rm with}\quad
\|\cdot\|_{jn_{(p,q,s)_\az}(\rn)}
\ls \|\cdot\|_{RM_{p,q,\az}(\rn)}$$
due to \eqref{Cs} and their definitions,
where the implicit positive constant
is independent of functions under consideration.
\end{question}

\begin{question}\label{openQ-hk}
Let $v\in(1,\fz)$, $w\in(1,\fz]$, and $Q_0$ be a cube of $\rn$.
\begin{itemize}
\item[\rm{(i)}]
It is interesting to clarify the relation between $\bigcup_{v\in(1,\fz)}hk_{(v,w,0)_{0}}(Q_0)$ and $ h^{1,w}_{at}(Q_0)$, and
to find the condition on $g$ such that $\|g\|_{h^{1,w}_{at}(Q_0)}=\lim_{v\to 1^+}\|g\|_{hk_{(v,w,0)_0}(Q_0)}$.
\item[\rm{(ii)}]
Let $\az\in(0,\fz)$ and $s\in\zz_{+}$.
As $v\to 1^+$,
the relation between the atomic localized Hardy space (see \cite{G79} for the definition)
and $hk_{(v,w,s)_{\az}}(Q_0)$ is still unknown.
\end{itemize}
\end{question}

\subsection{Congruent John--Nirenberg--Campanato spaces}\label{Subsec-ConJNC}

Inspired by the JNC space (see Subsection \ref{Subsec-JNC})
and the space $\mathcal{B}$ (introduced and studied by Bourgain et al. \cite{BBM15}),
Jia et al. \cite{jtyyz21} introduced the special John--Nirenberg--Campanato
spaces via congruent cubes, which are of some amalgam features.
This subsection is devoted to the main properties and some applications
of congruent JNC spaces.

In what follows, for any $m\in \zz$, $\mathcal{D}_m(\rn)$
denotes the set of all subcubes of $\rn$ with side length $2^{-m}$,
$\mathcal{D}_m(Q_0)$ the set of all subcubes of $Q_0$
with side length $2^{-m}\ell(Q_0)$ for any given $m\in\zz_+$,
and $\mathcal{D}_{m}(Q_0):=\emptyset$ for any given $m\in\zz\setminus\zz_+$;
here and thereafter, $\ell(Q_0)$ denotes the side length of $Q_0$.

\begin{definition}\label{Defin.jncc}
	Let $p$, $q\in[1,\infty)$, $s\in\zz_+$, and $\alpha\in \rr$.
	The \emph{special John--Nirenberg--Campanato space via congruent cubes}
	(for short, \emph{congruent $\mathrm{JNC}$ space})
	$JN_{(p, q, s)_{\alpha}}^{\mathrm{con}}(\mathcal{X})$ is defined to be the
	set of all $f\in L_{\mathrm{loc}}^1(\mathcal{X})$ such that
	$$\|f\|_{JN_{(p,q,s)_{\alpha}}^{\mathrm{con}}(\mathcal{X})}
	:=\sup_{m\in\zz}
	\lf\{[f]^{(m)}_{(p,q,s)_{\alpha},\mathcal{X}}\r\}<\infty,$$
	where, for any $m\in\zz$, $[f]^{(m)}_{(p,q,s)_{\alpha},\mathcal{X}}$
	is defined to be
	$$\sup_{\{Q_j\}_j\subset \mathcal{D}_m(\mathcal{X})}
	\lf[\sum_{j}\lf|Q_{j}\r|\lf\{\lf|Q_{j}\r|^{-\alpha}\lf[\fint_{Q_{j}}
	\lf|f(x)-P_{Q_{j}}^{(s)}(f)(x)\r|^{q}\,dx\r]^{\frac{1}{q}}\r\}^{p} \r]^{\frac{1}{p}}$$
	with $P_{Q_j}^{(s)}(f)$ for any $j$ as in \eqref{PQsf} via $Q$ replaced by $Q_j$
	and the supremum taken over all collections of interior pairwise disjoint
	cubes $\{Q_j\}_j\subset\mathcal{D}_m(\mathcal{X})$.
	In particular, let
	$$JN_{p,q}^{\mathrm{con}}(\mathcal{X}):=JN_{(p,q,0)_{0}}^{\mathrm{con}}(\mathcal{X}).$$
\end{definition}

\begin{remark}\label{2eqNORM}
	Let $p$, $q\in[1,\infty)$, $s\in\zz_+$, and $\alpha\in \rr$.
	There exist some useful equivalent norms
	on $JN_{(p,q,s)_\alpha}^{\mathrm{con}}(\cx)$ as follows.
	\begin{enumerate}
		\item[\rm{(i)}] (non-dyadic side length)
		$f\in JN_{(p,q,s)_\alpha}^{\mathrm{con}}(\cx)$ if and only if
		$f\in L^1_\loc(\cx)$ and
		\begin{align*}
			\|f\|_{\widetilde{JN}_{(p,q,s)_{\alpha}}^{\mathrm{con}}(\mathcal{X})}
			:=\sup\lf[\sum_{j}\lf|Q_{j}\r|\lf\{\lf|Q_{j}\r|^{-\alpha}\lf[\fint_{Q_{j}}
			\lf|f(x)-P_{Q_{j}}^{(s)}(f)(x)\r|^{q}\,dx\r]^{\frac{1}{q}}\r\}^{p} \r]^{\frac{1}{p}}
			<\fz
		\end{align*}
	    if and only if $f\in L^1_\loc(\cx)$ and
	    \begin{align}\label{inff-P}
	    	\|f\|_{\widehat{JN}_{(p,q,s)_{\alpha}}^{\mathrm{con}}(\mathcal{X})}
	    	:=\sup\lf[\sum_{j}\lf|Q_{j}\r|\lf\{\lf|Q_{j}\r|^{-\alpha}
	    	\inf_{P\in \mathcal{P}_{s}(Q_j)}\lf[\fint_{Q_{j}}
	    	\lf|f(x)-P(x)\r|^{q}\,dx\r]^{\frac{1}{q}}\r\}^{p} \r]^{\frac{1}{p}}
	    	<\fz,
	    \end{align}
		where the suprema are taken over all collections of interior pairwise disjoint
		cubes $\{Q_j\}_j$ of $\cx$ with the same side length;
		moreover,
		$\|\cdot\|_{JN_{(p,q,s)_\alpha}^{\mathrm{con}}(\cx)}
		\sim\|\cdot\|_{\widetilde{JN}_{(p,q,s)_{\alpha}}^{\mathrm{con}}(\mathcal{X})}
		\sim\|\cdot\|_{\widehat{JN}_{(p,q,s)_{\alpha}}^{\mathrm{con}}(\mathcal{X})}$;
		see \cite[Remark 1.6(ii) and Propositions 2.6 and 2.7]{jtyyz21}.
		
		\item[\rm{(ii)}] (integral representation)
		In what follows, for any $y\in\rn$ and $r\in(0,\fz)$, let 
        $$B(y,r):=\{x\in\rn:\ |x-y|<r\}.$$
		Then
		$f\in JN_{(p,q,s)_\alpha}^{\mathrm{con}}(\rn)$ if and only if
		$f\in L^1_\loc(\rn)$ and
		$$\|f\|_{\ast}:=
		\sup_{r\in(0,\fz)}\lf[\int_{\rn}\lf\{|B(y,r)|^{-\alpha}
		\lf[\fint_{B(y,r)}\lf|f(x)-P_{B(y,r)}^{(s)}(f)(x)\r|^q
		\,dx\r]^{\frac{1}{q}}\r\}^p\,dy\r]^{\frac{1}{p}}<\fz;$$
		moreover,
		$\|\cdot\|_{JN_{(p,q,s)_\alpha}^{\mathrm{con}}(\rn)}
		\sim\|\cdot\|_{\ast}$;
		see \cite[Proposition 2.2]{jtyyz21} for this equivalence
		which plays an essential role when establishing the boundedness
		of operators on congruent JNC spaces;
		see \cite{JTYYZ2,JYYZ} for more details.
	\end{enumerate}
\end{remark}

The following proposition is just \cite[Proposition 2.10]{jtyyz21}.
\begin{proposition}\label{p2.1}
	Let $s\in\zz_+$, $\alpha\in\rr$, and $Q_0$ be any given cube of $\rn$.
	\begin{enumerate}
		\item[\rm (i)] For any given $p\in [1,\infty)$ and $q\in [1,\infty)$,
		$$JN_{(p,q,s)_{\alpha}}^{\mathrm{con}}(Q_0)
		\subset\lf[|Q_0|^{\frac{1}{p}-\frac{1}{q}-\alpha}L^q(Q_0)/\mathcal{P}_s(Q_0)\r].$$
		Moreover, for any $f\in JN_{(p,q,s)_{\alpha}}^{\mathrm{con}}(Q_0)$,
		$$\|f\|_{|Q_0|^{\frac{1}{p}-\frac{1}{q}-\alpha}L^q(Q_0)/\mathcal{P}_s(Q_0)}
		\leq\|f\|_{JN_{(p,q,s)_{\alpha}}^{\mathrm{con}}(Q_0)}.$$
		\item[\rm (ii)] If $\alpha\in(-\infty,0]$, then, for any given $p\in [1,\infty)$ and $q\in[p,\infty)$,
		$$JN_{(p,q,s)_{\alpha}}^{\mathrm{con}}(Q_0)
		=\lf[|Q_0|^{\frac{1}{p}-\frac{1}{q}-\alpha}L^q(Q_0)/\mathcal{P}_s(Q_0)\r]$$
		with equivalent norms.
		\item[\rm (iii)] If $q\in[1,\fz)$ and $1\leq p_1\leq p_2<\fz$, then
		$JN_{(p_2,q,s)_{\alpha}}^{\mathrm{con}}(Q_0)
		\subset JN_{(p_1,q,s)_{\alpha}}^{\mathrm{con}}(Q_0)$. Moreover, for any $f\in JN_{(p_2,q,s)_{\alpha}}^{\mathrm{con}}(Q_0)$,
		$$|Q_0|^{-\frac{1}{p_1}}\|f\|_{JN_{(p_1,q,s)_{\alpha}}^{\mathrm{con}}(Q_0)}
		\leq|Q_0|^{-\frac{1}{p_2}}\|f\|_{JN_{(p_2,q,s)_{\alpha}}^{\mathrm{con}}(Q_0)}.$$
		\item[\rm (iv)] If $p\in[1,\fz)$ and $1\leq q_1\leq q_2<\fz$, then
		$JN_{(p,q_2,s)_{\alpha}}^{\mathrm{con}}(\mathcal{X})\subset JN_{(p,q_1,s)_{\alpha}}^{\mathrm{con}}(\mathcal{X})$.
		Moreover, for any $f\in JN_{(p,q_2,s)_{\alpha}}^{\mathrm{con}}(\mathcal{X})$,
		$$\|f\|_{JN_{(p,q_1,s)_{\alpha}}^{\mathrm{con}}(\mathcal{X})}
		\leq\|f\|_{JN_{(p,q_2,s)_{\alpha}}^{\mathrm{con}}(\mathcal{X})}.$$
	\end{enumerate}
\end{proposition}
The relation of congruent JNC spaces and Campanato spaces
is similar to Proposition \ref{ptofz} and Corollary \ref{p=8},
and hence we omit the statement here;
see \cite[Proposition 2.11]{jtyyz21} for the details.
The relation of congruent JNC spaces and the space $\cb$
was discussed in \cite[Proposition 2.20 and Remark 2.21]{jtyyz21}.
Recall that the \emph{local Sobolev space}
$W_{\mathrm{loc}}^{1,p}(\rn)$ is defined by setting
$$W_{\mathrm{loc}}^{1,p}(\rn):=
\lf\{f\in L_{\mathrm{loc}}^p(\rn):\ |\nabla f|\in L_{\mathrm{loc}}^p(\rn)\r\},$$
here and thereafter, $\nabla f:=(\partial_1f,\ldots,\partial_nf)$,
where, for any $i\in\{1,\ldots,n\}$,
$\partial_i f$ denotes the \emph{weak derivative} of $f$,
namely, a locally integrable function on $\rn$ such that,
for any $\varphi\in C_{\mathrm{c}}^{\fz}(\rn)$
(the set of all infinitely differentiable functions on $\rn$ with
compact support),
$$\int_{\rn}f(x)\partial_i \varphi(x)\,dx=-\int_{\rn}\varphi(x)\partial_i f(x)\,dx.$$
The following proposition is just \cite[Proposition 2.13]{jtyyz21}.
\begin{proposition}\label{W1p}
Let $p\in(1,\infty)$ and $f\in L_{\mathrm{loc}}^p(\rn)$. Then
$|\nabla f|\in L^p(\rn)$ if and only if
$$\liminf_{m\to \fz}\,
[f]_{(p,p,0)_{1/n},\rn}^{(m)}<\infty,$$
where $[f]_{(p,p,0)_{1/n},\rn}^{(m)}$ is as in Definition \ref{Defin.jncc}.
Moreover, for any given $p\in[1,\infty)$,
there exists a constant $C_{(n,p)}\in[1,\infty)$ such that,
for any $f\in W_{\mathrm{loc}}^{1,p}(\rn)$,
$$\frac{1}{C_{(n,p)}}\lf[\int_{\rn}|\nabla f(x)|^p\,dx\r]^{\frac{1}{p}}
\leq\liminf_{m\to\fz}\,
[f]_{(p,p,0)_{1/n},\rn}^{(m)}
\leq C_{(n,p)}\lf[\int_{\rn}|\nabla f(x)|^p\,dx\r]^{\frac{1}{p}}.$$
\end{proposition}

\begin{remark}\label{SobolevBMO}
Fusco et al. studied BMO-type seminorms and Sobolev functions
in \cite{fms18}. Indeed, in \cite[Theorem 2.2]{fms18},
Fusco et al. showed that Proposition \ref{W1p} still holds true with cubes $\{Q_j\}_j$,
in the supremum of $[f]_{(p,p,0)_{1/n},\rn}^{(m)}$,
having the same side length but \emph{arbitrary orientation}.
Later, the main results of \cite{fms18} were further
extended by Di Fratta and Fiorenza in \cite{df20},
via replacing a family of open cubes by a broader class of \emph{tessellations}
(from pentagonal and hexagonal tilings to space-filling polyhedrons
and creative tessellations).
\end{remark}

The following nontriviality is just \cite[Propositions 2.16 and 2.19]{jtyyz21}.
\begin{proposition}\label{t2.5}
	Let $p\in(1,\infty)$ and $q\in[1,p)$.
	\begin{itemize}
		\item[{\rm(i)}]
		Let $I_0$ be any given bounded interval of $\rr$.
		Then
		$$
		JN_{p,q}(I_0)\subsetneqq JN_{p,q}^{\mathrm{con}}(I_0)
		\quad\text{and}\quad
		JN_{p,q}(\rr)\subsetneqq JN_{p,q}^{\mathrm{con}}(\rr).$$
		
		\item[{\rm(ii)}]
		Let $Q_0$ be any given cube of $\rn$. Then
		$$JN_{p,q}(Q_0)\subsetneqq JN_{p,q}^{\mathrm{con}}(Q_0).$$
	\end{itemize}
\end{proposition}

Similarly to Theorem \ref{duality},
the following dual result is just \cite[Theorem 4.10]{jtyyz21}.
Recall that the congruent Hardy-type space $HK_{(u,v,s)_{\alpha}}^{\mathrm{con}}(\mathcal{X})$
is defined as in Definition \ref{HKuvb} with the additional condition that
all cubes of the polymer have the same side length;
see \cite[Definition 4.7]{jtyyz21} for more details.

\begin{theorem}\label{t3.9}
	Let $p$, $q\in(1,\infty)$, $1/p=1/p'=1=1/q+1/q'$, $s\in\zz_+$, and $\alpha\in\rr$.
	If $JN_{(p,q,s)_{\alpha}}^{\mathrm{con}}(\mathcal{X})$ is equipped with the norm
	$\|\cdot\|_{\widehat{JN}_{(p,q,s)_{\alpha}}^{\mathrm{con}}(\mathcal{X})}$
	in \eqref{inff-P},
	then $$\lf(HK_{(p',q',s)_{\alpha}}^{\mathrm{con}}(\mathcal{X})\r)^*
	=JN_{(p,q,s)_{\alpha}}^{\mathrm{con}}(\mathcal{X})$$
	with equivalent norms in the following sense:
	\begin{enumerate}
		\item[\rm (i)] Any $f\in JN_{(p,q,s)_{\alpha}}^{\mathrm{con}}(\mathcal{X})$
		induces a linear functional $\mathcal{L}_f$ which is given by setting,
		for any $g\in HK_{(p',q',s)_{\alpha}}^{\mathrm{con}}(\mathcal{X})$
		and $\{g_i\}_i\subset \widetilde{HK}_{(p',q',s)_{\alpha}}^{\mathrm{con}}(\mathcal{X})$
		with $g=\sum_i g_i$ in
		$(JN_{(p,q,s)_{\alpha}}^{\mathrm{con}}(\mathcal{X}))^*$,
		\begin{align*}
			\mathcal{L}_f(g):=\langle g,f\rangle
			=\sum_i\langle g_i,f\rangle.
		\end{align*}
		Moreover,
		for any $g\in HK_{(p',q',s)_{\alpha}}^{\mathrm{con-fin}}(\mathcal{X})$,
		\begin{align*}
			\mathcal{L}(g)=\int_{\mathcal{X}}f(x)g(x)\,dx
			\quad {and}\quad
			\lf\|\mathcal{L}_f\r\|_{(HK_{(p',q',s)_{\alpha}}^{\mathrm{con}}(\mathcal{X}))^*}
			\leq \|f\|_{\widehat{JN}_{(p,q,s)_{\alpha}}^{\mathrm{con}}(\mathcal{X})}.
		\end{align*}
		\item[\rm (ii)] Conversely, for any continuous
		linear functional $\mathcal{L}$ on
		$HK_{(p',q',s)_{\alpha}}^{\mathrm{con}}(\mathcal{X})$,
		there exists a unique
		$f\in JN_{(p,q,s)_{\alpha}}^{\mathrm{con}}(\mathcal{X})$ such that,
		for any $g\in HK_{(p',q',s)_{\alpha}}^{\mathrm{con-fin}}(\mathcal{X})$,
		$$\mathcal{L}(g)=\int_{\mathcal{X}}f(x)g(x)\,dx
		\quad {and}\quad
		\|f\|_{\widehat{JN}_{(p,q,s)_{\alpha}}^{\mathrm{con}}(\mathcal{X})}
		\leq\|\mathcal{L}\|_{(HK_{(p',q',s)_{\alpha}}^{\mathrm{con}}(\mathcal{X}))^*}.$$
	\end{enumerate}
\end{theorem}

Moreover, when $\cx=Q_0$,
we further have the VMO-$H^1$ type duality for the congruent Hardy-type space;
see Theorem \ref{VMO-H1-duality} below.

Recall that Ess\'en et al. \cite{ejpx00} introduced
and studied the \emph{$Q$} space on $\rn$, which generalizes the space $\BMO(\rn)$.
Later, the $Q$ space  proves very useful in harmonic analysis,
potential analysis, partial differential equations as well as
the closely related fields;
see, for instance, \cite{x19,kxzz17,zsyy17}.
Thus, it is natural to consider some ``new $Q$ space'' corresponding to
the John--Nirenberg space $JN_p$.
Based on Remark \ref{2eqNORM}(ii), Tao et al. \cite{tyyJNQ} introduced
and studied the \emph{John--Nirenberg-$Q$ space} on $\rn$ via congruent cubes,
which contains the congruent John--Nirenberg space on $\rn$ as special cases,
and also sheds some light on the mysterious John--Nirenberg space.

\section{Riesz-type space}\label{Sec-Riesz}

Observe that, if we partially subtract integral means
(or polynomials for high order cases) in $\|f\|_{JN_{(p,q,s)_\az}(\cx)}$,
namely, drop $P_{Q_i}^{(s)}(f)$ in
$$\lf\{\sum_i|Q_i|\lf[|Q_i|^{-\alpha}
\lf\{\fint_{Q_i}\lf|f(x)-P_{Q_i}^{(s)}(f)(x)\r|^q\,dx\r\}
^{\frac 1q}\r]^p\r\}^{\frac 1p}$$
for any $i$ satisfying $\ell(Q_i)\ge c_0$,
then we obtain the localized JNC space as in Definition \ref{def.jnpqs}.
Thus, a natural question arises:
what if we thoroughly drop all $\{P_{Q_i}^{(s)}(f)\}_i$ in $\|f\|_{JN_{(p,q,s)_\az}(\cx)}$?
In this section, we study the space with such norm
(subtracting all $\{P_{Q_i}^{(s)}(f)\}_i$ in the norm of the JNC space).
As a bridge connecting Lebesgue and Morrey spaces via Riesz norms,
it was called the ``Riesz--Morrey space''.
For more studies on the well-known Morrey space,
we refer the reader to, for instance, \cite{hns17,mst18,ms19,hs20}
and, in particular, the recent monographs
by Sawano et al. \cite{sfk20i,sfk20ii}.

\subsection{Riesz--Morrey spaces}\label{Subsec-RM}
As a suitable substitute of $L^\fz(\cx)$, the space $\BMO(\cx)$ proves very useful
in harmonic analysis and partial differential equations.
Recall that
$$\|f\|_{\BMO(\cx)}:=\sup_{{\rm cube}\,Q\subset\cx}\fint_Q\lf|f(x)-f_Q\r|\,dx.$$
Indeed, the only difference between them exists in subtracting integral means,
which is just the following proposition.
In what follows, for any measurable function $f$,
let
$$\|f\|_{L^\fz_\ast(\cx)}:=\sup_{{\rm cube}\,Q\subset \cx}\fint_Q |f(x)|\,dx.$$

\begin{proposition}\label{LfzBMO}
$f\in L^\fz(\cx)$ if and only if $f\in L^1_\loc(\cx)$ and $\|f\|_{L^\fz_\ast(\cx)}<\fz$.
Moreover, $\|\cdot\|_{L^\fz(\cx)}=\|\cdot\|_{L^\fz_\ast(\cx)}$.
\end{proposition}

\begin{proof}
On one hand, for any $f\in L^\fz(\cx)$, it is easy to see that
$f\in L^1_\loc(\cx)$ and
$$\|f\|_{L^\fz_\ast(\cx)}=\sup_{Q\subset\cx}\fint_Q|f(x)|\,dx
\le\sup_{Q\subset\cx}\|f\|_{L^\fz(\cx)}
=\|f\|_{L^\fz(\cx)}<\fz.$$

On the other hand, for any $f\in L^1_\loc(\cx)$ and $\|f\|_{L^\fz_\ast(\cx)}<\fz$,
let $x$ be any Lebesgue point of $f$.
Then, from the Lebesgue differentiation theorem, we deduce that
$$|f(x)|=\lim_{|Q|\to 0^+,\,Q\ni x}\fint_Q|f(y)|\,dy
\le \sup_{Q\subset\cx}\fint_Q|f(y)|\,dy=\|f\|_{L^\fz_\ast(\cx)},$$
which, together with the Lebesgue differentiation theorem again, further implies that
$$\|f\|_{L^\fz(\cx)}\le\|f\|_{L^\fz_\ast(\cx)}$$
and hence $f\in L^\fz(\cx)$.
Moreover, we have $\|\cdot\|_{L^\fz(\cx)}=\|\cdot\|_{L^\fz_\ast(\cx)}$.
This finishes the proof of Proposition \ref{LfzBMO}.
\end{proof}

Also, if we remove integral means in the $JN_p(Q_0)$-norm
$$\|f\|_{JN_p(Q_0)}=\sup \lf[ \sum_i|Q_i|\lf(\fint_{Q_i}\lf|f(x)-f_{Q_i}\r|\,dx
\r)^p \r]^\frac1p,$$
where the supremum is taken over all collections of
cubes $\{Q_i\}_i$ of $Q_0$ with pairwise disjoint interiors,
then we obtain
$$\sup \lf[ \sum_i|Q_i|\lf(\fint_{Q_i}\lf|f(x)\r|\,dx
   \r)^p \r]^\frac1p=:\|f\|_{R_p(Q_0)}$$
which coincides with $\|f\|_{L^p(Q_0)}$ due to Riesz \cite{R1910}.
Corresponding to the JNC space,
the following triple index Riesz-type space $R_{p,q,\az}(\cx)$,
called the Riesz--Morrey space,
was introduced and studied in \cite{tyyBJMA}.

\begin{definition}\label{def-Riesz}
Let $p\in[1,\infty]$, $q\in[1,\infty]$, and $\alpha\in\rr$.
The \emph{Riesz--Morrey space} $ RM_{p,q,\alpha}(\cx)$ is defined by setting
$$  RM_{p,q,\alpha}(\cx):=\lf\{f\in L^q_{\loc}(\cx):\,\,
\|f\|_{  RM_{p,q,\alpha}(\cx)}<\fz\right\},$$
where
\begin{align*}
\|f\|_{  RM_{p,q,\alpha}(\cx)}:=
\begin{cases}
\displaystyle{\sup\lf[\sum_i|Q_i|^{1-p\az-\frac pq}\|f\|_{L^q(Q_i)}^p\right]^{\frac 1p}}
&{\rm if\ }p\in[1,\fz),\ q\in[1,\fz],\\
\displaystyle{\sup\sup_i|Q_i|^{-\alpha-\frac1q}\|f\|_{L^q(Q_i)}}
&{\rm if\ }p=\fz,\ q\in[1,\fz]
\end{cases}
\end{align*}
and the suprema are taken over all collections of
subcubes $\{Q_i\}_i$ of $\cx$ with pairwise disjoint interiors.
In addition, $R_{p,q,0}(\cx)=:R_{p,q}(\cx)$.
\end{definition}

Observe that the Riesz--Morrey norm
$\|\cdot\|_{  RM_{p,q,\alpha}(\cx)}$
differs from the John--Nirenberg--Campanato norm
$\|\cdot\|_{JN_{(p,q,s)_\az}(\cx)}$
with $s=0$ only in subtracting mean oscillations;
see \cite[Remark 2]{tyyBJMA} for more details.
It is easy to see that $\|\cdot\|_{R_{p,1,0}(Q_0)}=\|\cdot\|_{R_p(Q_0)}$
and, as a generalization of the above equivalence in Riesz \cite{R1910},
the following proposition is just \cite[Proposition 1]{tyyBJMA}.

\begin{proposition}\label{Riesz}
Let $p\in[1,\fz]$ and $q\in[1,p]$. Then $f\in L^p(\cx)$
if and only if $f\in R_{p,q}(\cx)$.
Moreover,
$L^p(\cx)= R_{p,q}(\cx)$
with equivalent norms, namely, for any $f\in L^q_\loc(\cx)$,
$\|f\|_{L^p(\cx)}=\|f\|_{ R_{p,q}(\cx)}$.
\end{proposition}

As for the case $1\le p<q\le\fz$, by \cite[Remark 2.3]{tyyBJMA}, we know that
$$R_{p,q}(\rn)=0\neq L^q(\rn)=R_{q,q}(\rn),$$
and
$$\lf[|Q_0|^{-\frac1p} R_{p,q}(Q_0)\r]=\lf[|Q_0|^{-\frac1q}L^q(Q_0)\r]
=\lf[|Q_0|^{-\frac1q} R_{q,q}(Q_0)\r]$$
with equivalent norms.

Moreover, it is shown in \cite[Theorem 1 and Corollary 1]{tyyBJMA}
that the endpoint spaces of Riesz--Morrey spaces
are Lebesgue spaces or Morrey spaces.
In this sense, we regard the Riesz--Morrey space
as a bridge connecting the Lebesgue space and the Morrey space.
Thus, a natural question arises:
whether or not Riesz--Morrey spaces are truly new spaces
different from Lebesgue spaces or Morrey spaces?
Very recently, Zeng et al. \cite{zcty21} give an \emph{affirmative}
answer to this question via constructing two nontrivial functions
over $\rn$ and any given cube $Q$ of $\rn$.
It should be pointed out that the nontrivial function on the cube $Q$
is geometrically similar to the striking function constructed by
Dafni et al. in the proof of \cite[Proposition 3.2]{DHKY18}.
Furthermore, we have the following classifications of Riesz--Morrey spaces,
which is just \cite[Corollary 3.7]{zcty21}.

\begin{theorem}\label{thm-ClasRM}
\begin{itemize}
	\item[{\rm(i)}]
	Let $p\in(1,\fz]$ and $q\in[1,p)$. Then
	$$RM_{p,q,\az}(\rn)
	\begin{cases}
		=L^q(\rn) &{\rm if\ } \az=\frac1p-\frac1q,\\
		\supsetneqq L^{\frac{p}{1-p\az}}(\rn)&{\rm if\ } \az\in\lf(\frac1p-\frac1q,0\r),\\
		=L^p(\rn) &{\rm if\ } \az=0,\\
		=\{0\} &{\rm if\ } \az\in\lf(-\fz,\frac1p-\frac1q\r)\cup(0,\fz).
	\end{cases}$$
	In particular, if $\az\in(-\frac1q,0)$, then
	$RM_{\fz,q,\az}(\rn)=M_{q}^{-1/\az}(\rn)$ which is just the Morrey space
	defined in Remark \ref{rem-Morrey}.
	
	\item[{\rm(ii)}]
	Let $p\in[1,\fz]$ and $q\in[p,\fz]$. Then
	$$RM_{p,q,\az}(\rn)
	\begin{cases}
		=L^q(\rn) &{\rm if\ } \az=\frac1p-\frac1q=0,\\
		=\{0\} &{\rm if\ } \az=\frac1p-\frac1q\neq0,\\
		=\{0\} &{\rm if\ } \az\in\rr\setminus\lf\{\frac1p-\frac1q\r\}.
	\end{cases}$$
	
	\item[{\rm(iii)}]
	Let $p\in(1,\fz]$, $q\in[1,p)$, and
	$Q_0$ be any cube of $\rn$. Then
	$$RM_{p,q,\az}(Q_0)
	\begin{cases}
		=L^q(Q_0) &{\rm if\ } \az=\lf(-\fz,\frac1p-\frac1q\r],\\
		\supsetneqq L^{\frac{p}{1-p\az}}(Q_0)&{\rm if\ } \az\in\lf(\frac1p-\frac1q,0\r),\\
		=L^p(Q_0) &{\rm if\ } \az=0,\\
		=\{0\} &{\rm if\ } \az\in(0,\fz).
	\end{cases}$$
	In particular, $RM_{\fz,q,\az}(Q_0)=M_{q}^{-1/\az}(Q_0)$ if $\az\in(-\frac1q,0)$.

	\item[{\rm(iv)}]
	Let $p\in[1,\fz]$, $q\in[p,\fz]$, and
	$Q_0$ be any cube of $\rn$. Then
	$$RM_{p,q,\az}(Q_0)
	\begin{cases}
		=L^q(Q_0) &{\rm if\ } \az\in(-\fz,0],\\
		=\{0\} &{\rm if\ } \az\in(0,\fz).
	\end{cases}$$
\end{itemize}
\end{theorem}

Recall that, by \cite[Theorem 1]{BRV99},
the predual space of the Morrey space is the so-called block space.
Combining this with the duality of
John--Nirenberg--Campanato spaces in \cite[Theorem 3.9]{tyyNA},
the authors in \cite{tyyBJMA} introduced the block-type space which proves
the predual of the Riesz--Morrey space.
Observe that every $(\fz,v,\az)$-block in Definition \ref{def-block}(i)
is exactly a \emph{$(v,\frac{\az}{n})$-block}
introduced in \cite{BRV99}.

\begin{definition}\label{def-block}
Let $u,\ v\in[1,\fz]$, $1/u+1/u'=1=1/v+1/v'$, and $\az\in\rr$.
Let $(RM_{u',v',\az}(\cx))^\ast$ be the \emph{dual space} of
$RM_{u',v',\az}(\cx)$ equipped with the weak-$\ast$ topology.
\begin{itemize}
\item[{\rm(i)}] A function $b$ is called a \emph{$(u,v,\az)$-block} if
$$\supp(b):=\lf\{x\in\cx:\,\,b(x)\neq0\right\}
\subset Q\quad{\rm and}\quad\|b\|_{L^v(Q)}\le|Q|^{\frac1v-\frac1u-\az}.$$

\item[{\rm(ii)}] The \emph{space of $(u,v,\az)$-chains}, $\widetilde{ B}_{u,v,\az}(\cx)$,
is defined by setting
\begin{align*}
\widetilde{ B}_{u,v,\az}(\cx):=\lf\{h\in\lf( RM_{u',v',\az}(\cx)\right)^\ast:\,\,
h=\sum_{j}\lz_j b_j\ \ {\rm and}\ \ \lf\|\lf\{\lz_j\right\}_j\right\|_{\ell^u}<\fz \right\},
\end{align*}
where $\{b_j\}_j$ are $(u,v,\az)$-blocks supported, respectively,
in subcubes $\{Q_j\}$ of $\cx$ with pairwise disjoint interiors,
and $\{\lz_j\}_j\subset\cc$ with $\|\{\lz_j\}_j\|_{\ell^u}<\fz$
[see \eqref{lp-norm} for the definition of $\|\cdot\|_{\ell^u}$].
Moreover, any $h\in\widetilde{ B}_{u,v,\az}(\cx)$ is called a
\emph{$(u,v,\az)$-chain} and its norm is defined by setting
$$\|h\|_{\widetilde{ B}_{u,v,\az}(\cx)}:=
\inf\lf\|\lf\{\lz_j\right\}_j\right\|_{\ell^u},$$
where the infimum is taken over all decompositions of $h$ as above.

\item[{\rm(iii)}] The \emph{block-type space}  $ B_{u,v,\az}(\cx)$
is defined by setting
\begin{align*}
 B_{u,v,\az}(\cx):=\lf\{g\in\lf( RM_{u',v',\az}(\cx)\right)^\ast:\,\,
g=\sum_{i} h_i\ \ {\rm and}\ \
\sum_i\lf\|h_j\right\|_{\widetilde{ B}_{u,v,\az}(\cx)}<\fz \right\},
\end{align*}
where $\{h_i\}_i$ are $(u,v,\az)$-chains.
Moreover, for any $g\in  B_{u,v,\az}(\cx)$,
$$\|g\|_{ B_{u,v,\az}(\cx)}:=\inf \sum_i\lf\|h_j\right\|_{\widetilde{ B}_{u,v,\az}(\cx)},$$
where the infimum is taken over all decompositions of $g$ as above.

\item[{\rm(iv)}] The \emph{finite block-type space
$ B_{u,v,\az}^{\mathrm{fin}}(\cx)$} is defined to be
the set of all finite summations
$$\sum_{m=1}^M \lz_{m}b_{m},$$
where $M\in\nn$, $\{\lz_{m}\}_{m=1}^M\subset\cc$, and
$\{b_{m}\}_{m=1}^M$ are $(u,v,\az)$-blocks.
\end{itemize}
\end{definition}

The following dual theorem is just \cite[Theorem 2]{tyyBJMA}.

\begin{theorem}\label{dual-RM}
Let $p,\ q\in (1,\fz)$, $1/p+1/p'=1=1/q+1/q'$, and $\alpha\in\rr$.
Then $( B_{p',q',\az}(\cx))^\ast= RM_{p,q,\alpha}(\cx)$
in the following sense:
\begin{enumerate}
\item[{\rm(i)}]If $f\in  RM_{p,q,\alpha}(\cx)$, then $f$ induces a linear
functional $\mathcal{L}_f$ on $ B_{p',q',\az}(\cx)$ with
$$\|\mathcal{L}_f\|_{( B_{p',q',\az}(\cx))^\ast}
  \le C\|f\|_{ RM_{p,q,\alpha}(\cx)},$$
where $C$ is a positive constant independent of $f$.
\item[{\rm(ii)}] If $\mathcal{L}\in( B_{p',q',\az}(\cx))^\ast$,
then there exists some $f\in  RM_{p,q,\alpha}(\cx)$ such that,
for any $g\in B_{p',q',\az}^{\mathrm{fin}}(\cx)$,
$$\mathcal{L}(g)=\int_{\cx}f(x)g(x)\,dx,$$
and
$$\|\mathcal{L}\|_{(B_{p',q',\az}(\cx))^\ast}
\sim\|f\|_{ RM_{p,q,\alpha}(\cx)}$$
with the positive equivalence constants independent of $f$.
\end{enumerate}
\end{theorem}

Also, for the Riesz--Morrey space,
there exist three open questions unsolved so far.
The first question is on the relation between
the Riesz--Morrey space and the weak Lebesgue space.
\begin{question}\label{RM-wLp}
Let $p\in(1,\fz)$, $q\in[1,p)$, and $\az\in(\frac1p-\frac1q,0)$.
Then Zeng et al. \cite[Remark 3.4]{zcty21} showed that
$$RM_{p,q,\az}(\rn)\nsubseteq L^{\frac{p}{1-p\az},\fz}(\rn)
\nsubseteq RM_{p,q,\az}(\rn),$$
which implies that, on $\rn$,
the Riesz--Morrey space and the weak Lebesgue space
do not cover each other.
Also, for any given cube $Q_0$ of $\rn$,
Zeng et al. \cite[Remark 3.6]{zcty21} showed that
$$L^{\frac{p}{1-p\az},\fz}(Q_0)\nsubseteq RM_{p,q,\az}(Q_0).$$
However, it is still unknown whether or not
$$RM_{p,q,\az}(Q_0)\nsubseteq L^{\frac{p}{1-p\az},\fz}(Q_0)$$
still holds true. This question was posed in \cite[Remark 3.6]{zcty21}
and is still \emph{unclear} so far.
\end{question}
The following Questions \ref{RM-JNC} and \ref{M-bdd}
are just \cite[Remarks 4 and 5]{tyyBJMA}, respectively.
\begin{question}\label{RM-JNC}
As a counterpart of \eqref{Mor=Cam},
for any given $p\in[1,\fz)$, $q\in[1,p)$, $s\in\zz_+$, and
$\az\in[\frac1p-\frac1q,0)$,
it is interesting to ask whether or not
$$JN_{(p,q,s)_\alpha}(\cx)= \lf[RM_{p,q,\az}(\cx)/\mathcal{P}_s(\cx)\r]$$
and, for any $f\in JN_{(p,q,s)_\alpha}(\cx)$,
$$\|f\|_{JN_{(p,q,s)_\alpha}(\cx)}
\sim\lf\|f-\sigma(f)\r\|_{RM_{p,q,\az}(\cx)},$$
with the positive equivalence constants independent of $f$,
still hold true. This is still \emph{unclear} so far.
\end{question}

\begin{question}\label{M-bdd}
Recall that, for any given $f\in L^1_\loc(\cx)$ and any $x\in\cx$,
the \emph{Hardy--Littlewood maximal function} $\cm(f)(x)$ is defined by setting
\begin{align}\label{HL}
\cm(f)(x):=\sup_{Q\ni x}\fint_Q |f(y)|\,dy,
\end{align}
where the supremum is taken over all cubes $Q$ containing $x$.
Meanwhile, $\cm$ is called the \emph{Hardy--Littlewood maximal operator}.
It is well known that $\cm$ is bounded on $L^q(\cx)$ for any given $q\in(1,\fz]$;
see, for instance, \cite[p.\,31, Theorem 2.5]{Duo}.
Moreover, $\cm$ is also bounded on $ M_{q}^{-1/\az}(\cx)$
for any given $q\in(1,\fz]$ and $\az\in[-\frac1q,0]$;
see, for instance, \cite[Theorem 1]{CF88}.
To sum up, the boundedness of $\cm$ on
endpoint spaces of Riesz--Morrey spaces
(Lebesgue spaces and Morrey spaces)
have already been obtained. Therefore,
it is very interesting to ask whether or not $\cm$ is bounded
on the Riesz--Morrey space $ RM_{p,q,\az}(\cx)$
with $p\in(1,\fz]$, $q\in[1,p)$, and $\az\in(\frac1p-\frac1q,0)$.
This is a challenging and important problem which is still open so far.
\end{question}

\subsection{Congruent Riesz--Morrey spaces}\label{Subsec-RM-Appl}

To obtain the boundedness of several important operators,
we next consider a special Riesz--Morrey space via congruent cubes,
denoted by $RM_{p,q,\alpha}(\rn)$, as in Subsection \ref{Subsec-ConJNC}.
In this subsection, we first recall the definition of $RM_{p,q,\alpha}^{\rm con}(\rn)$,
and then review the boundedness of the Hardy--Littlewood maximal operator
on this space.

\begin{definition}\label{def-RM}
Let $p,\,q\in[1,\infty]$ and $\alpha\in\rr$.
The \emph{special Riesz--Morrey space via congruent cubes}
(for short, \emph{congruent Riesz--Morrey space})
$RM_{p,q,\alpha}^{\rm con}(\rn)$
is defined to be the set of all locally integrable functions
$f$ on $\rn$ such that
\begin{align*}
\|f\|_{RM_{p,q,\alpha}^{\rm con}(\rn)}:=
\begin{cases}
\displaystyle
\sup\lf[\sum_j|Q_j|^{1-p\az-\frac pq}\|f\|_{L^q(Q_i)}^p\r]^{\frac 1p},
&p\in[1,\fz), \\
\displaystyle
\sup_{{\rm cube\ }Q\subset \rn} |Q|^{-\az-\frac1q}\|f\|_{L^q(Q)},
&p=\fz
\end{cases}
\end{align*}
is finite, where the first supremum is taken over all collections of
interior pairwise disjoint cubes $\{Q_j\}_j$ of $\rn$ with the same side length.
\end{definition}

\begin{remark}\label{rem-RM}
\begin{itemize}
\item [{\rm(i)}]
If we do not require that $\{Q_j\}_j$ have the same size in
the definition of congruent Riesz--Morrey spaces,
then it is just the Riesz--Morrey space $RM_{p,q,\alpha}(\rn)$
in Subsection \ref{Subsec-RM}.

\item [{\rm(ii)}]
If $p=\infty$, $q\in(0,\infty)$, and $\alpha\in[-\frac1q,0)$,
then $RM_{p,q,\alpha}^{\mathrm{con}}(\rn)$ in Definition \ref{def-RM}
coincides with the Morrey space $M_q^{-1/\az}(\rn)$ in Remark \ref{rem-Morrey}.

\item [{\rm(iii)}]
Similarly to Remark \ref{2eqNORM},
for any given $p,\,q\in[1,\infty)$, and $\alpha\in \rr$,
$f\in RM_{p,q,\alpha}^{\mathrm{con}}(\rn)$ if and only if
$f\in L^1_\loc(\rn)$  and
$$\|f\|_{\widetilde{RM}_{p,q,\alpha}^{\mathrm{con}}(\rn)}
:=\sup_{r\in(0,\fz)}\lf[\int_{\rn}\lf\{|B(y,r)|^{-\alpha}
\lf[\fint_{B(y,r)}\lf|f(x)\r|^q dx\r]^{\frac{1}{q}}\r\}^p\,dy\r]^{\frac{1}{p}}$$
is finite; moreover,
$$\|\cdot\|_{RM_{p,q,\alpha}^{\mathrm{con}}(\rn)}
\sim\|\cdot\|_{\widetilde{RM}_{p,q,\alpha}^{\mathrm{con}}(\rn)};$$
see \cite{JTYYZ2} for more details.
Recall that, for any $y\in\rn$ and $r\in(0,\fz)$,
$$B(y,r):=\{x\in\rn:\ |x-y|<r\}.$$
\end{itemize}
\end{remark}

The following boundedness of the Hardy--Littlewood maximal operator
on congruent Riesz--Morrey spaces was obtained in \cite{JTYYZ2}.

\begin{theorem}\label{bdd-HLM}
Let  $p,\,q\in(1,\infty)$, $\alpha\in\rr$, and $\cm$ be
the Hardy--Littlewood maximal operator as in \eqref{HL}.
Then $\cm$ is bounded on $RM_{p,q,\alpha}^{\rm con}(\rn)$.
\end{theorem}

Moreover, via Theorem \ref{bdd-HLM},
Jia et al. \cite{JTYYZ2} also establishes the boundedness of
Calder\'on--Zygmund operators on congruent Riesz--Morrey spaces.

Finally, since a congruent Riesz--Morrey space is a
\emph{ball Banach function space},
we refer the reader to \cite{tyyz21} for
the equivalent characterizations of the boundedness and
the compactness of Calder\'on--Zygmund commutators
on ball Banach function spaces.
It should be mentioned that, a crucial assumption in \cite{tyyz21}
is the boundedness of $\cm$,
and hence Theorem \ref{bdd-HLM} provides an essential
tool when studying the boundedness of operators on congruent Riesz--Morrey spaces.

\section{Vanishing subspace}\label{Sec-Van}

In this section, we focus on several vanishing subspaces
of aforementioned John--Nirenberg-type spaces.
In what follows,
$C^\fz(\rn)$ denotes the set of all infinitely differentiable functions on $\rn$;
$\mathbf{0}$ denotes the origin of $\rn$;
for any $\az:=(\az_1,\ldots,\az_n)\in\zz_+^n:=(\mathbb{Z}_+)^n$, let
$\partial^\az:=(\frac\partial{\partial x_1})^{\az_1}\cdots(\frac\partial{\partial x_n})^{\az_n}$;
for any given normed linear space $\cy$ and
any given its subset $\cx$,
$\overline{\cx}^{\cy}$ denotes the \emph{closure} of the set $\cx$ in $\cy$
in terms of the topology of $\cy$, and,
if $\mathcal{Y}=\rn$, we then denote $\overline{\cx}^{\cy}$ simply by $\overline{\cx}$.

\subsection{Vanishing BMO spaces}\label{Subsec-VanBMO}

We now recall several vanishing subspaces of the space $\BMO(\rn)$.
\begin{enumerate}
\item[$\bullet$] $\VMO(\rn)$, introduced by Sarason \cite{Sarason75},
is defined by setting
$$\VMO(\rn):=\overline{C_{\rm u}(\rn)\cap\BMO(\rn)}^{\BMO(\rn)},$$
where $C_{\rm u}(\rn)$ denotes the set of all uniformly continuous functions on $\rn$.

\item[$\bullet$] $\CMO(\rn)$, announced in Neri \cite{N75},
is defined by setting
$$\CMO(\rn):=\overline{C_{\rm c}^\fz(\rn)}^{\BMO(\rn)},$$
where $C_{\rm c}^\fz(\rn)$ denotes the set of all infinitely differentiable functions on $\rn$ with
compact support.
In addition, by approximations of the identity, it is easy to find that
\begin{align}\label{Cc=C0}
\CMO(\rn)=\overline{C_{\rm c}(\rn)}^{\BMO(\rn)}=\overline{C_0(\rn)}^{\BMO(\rn)},
\end{align}
where $C_{\rm c}(\rn)$ denotes the set of all functions on $\rn$ with
compact support, and
$C_0(\rn)$ the set of all continuous functions on $\rn$
which vanish at the infinity.

\item[$\bullet$] $\MMO(\rn)$, introduced by Torres and Xue \cite{TX19},
is defined by setting
$$\MMO(\rn):=\overline{A_\fz(\rn)}^{\BMO(\rn)},$$
where
$$A_\fz(\rn):=\lf\{b\in C^\fz(\rn)\cap L^\fz(\rn):
\,\,\forall\ \az\in\zz^n_+\setminus\{\mathbf{0}\},
\lim_{|x|\to\fz}\partial^\az b(x)=0\r\}.$$

\item[$\bullet$] $\XMO(\rn)$, introduced by Torres and Xue \cite{TX19},
is defined by setting
$$\XMO(\rn):=\overline{B_\fz(\rn)}^{\BMO(\rn)},$$
where
$$B_\fz(\rn):=\lf\{b\in C^\fz(\rn)\cap\BMO(\rn):
\,\,\forall\ \az\in\zz^n_+\setminus\{\mathbf{0}\},
\lim_{|x|\to\fz}\partial^\az b(x)=0\r\}.$$

\item[$\bullet$] $\xMO(\rn)$, introduced by Tao el al. \cite{txyyJFAA},
is defined by setting
$$\xMO(\rn):=\overline{B_1(\rn)}^{\BMO(\rn)},$$
where
$$B_1(\rn):=\lf\{b\in C^1(\rn)\cap\BMO(\rn):\,\,
\lim_{|x|\to\fz}|\nabla b(x)|=0\r\}$$
with $C^1(\rn)$ being the set of all functions $f$ on $\rn$ whose gradients
$\nabla f:=(\frac{\partial f}{\partial x_1},\dots,\frac{\partial f}{\partial x_n})$
are continuous.
\end{enumerate}

The relation of these vanishing subspaces reads as follows.

\begin{proposition}\label{subset}
$\CMO(\rn)\subsetneqq\MMO(\rn)\subsetneqq\XMO(\rn)=\xMO(\rn)\subsetneqq\VMO(\rn).$
\end{proposition}

Indeed, $$\CMO(\rn)\subsetneqq\MMO(\rn)\subsetneqq\XMO(\rn)$$
was obtained in \cite[p.\,5]{TX19}.
Moreover, $$\XMO(\rn)=\xMO(\rn)\subsetneqq\VMO(\rn)$$
was obtained in \cite[Corollary 1.3]{txyyJFAA},
which completely answered the open question proposed in
\cite[p.\,6]{TX19}.

Next, we investigate the mean oscillation characterizations
of these vanishing subspace.
Recall that, for any cube $Q$ of $\rn$, and any $f\in L_{\loc}^1(\rn)$,
the \emph{mean oscillation} $\co(f;Q)$ is defined by setting
\begin{align*}
\co(f;Q):=\fint_Q\lf|f(x)-f_Q\r|\,dx
=\frac1{|Q|}\int_Q\lf|f(x)-\frac1{|Q|}\int_Q f(y)\,dy\r|\,dx.
\end{align*}
The earliest results of $\VMO(\rn)$ was obtained by
Sarason in \cite{Sarason75}, and Theorem \ref{VMO-char} below
is a part of \cite[Theorem 1]{Sarason75}.
In what follows, $a\to0^+$ means $a\in(0,\fz)$ and $a\to0$.
\begin{theorem}\label{VMO-char}
$f\in\VMO(\rn)$ if and only if $f\in\BMO(\rn)$ and
$$\lim_{a\to0^+}\sup_{|Q|=a}\co(f; Q)=0.$$
\end{theorem}

The following equivalent characterization of $\CMO(\rn)$
is just Uchiyama \cite[p.\,166]{U78}.
\begin{theorem}\label{CMO-char}
$f\in\CMO(\rn)$ if and only if $f\in\BMO(\rn)$ and
satisfies the following three conditions:
\begin{itemize}
\item [{\rm(i)}] $\displaystyle{\lim_{a\to0^+}\sup_{|Q|=a}\co(f; Q)=0};$
\item [{\rm (ii)}]for any cube $Q$ of $\rn$,
$\displaystyle{\lim_{|x|\to\fz}\co(f; Q+x)=0};$
\item [{\rm(iii)}]$\displaystyle{\lim_{a\to\fz}\sup_{|Q|=a}\co(f; Q)=0}.$
\end{itemize}
\end{theorem}

Very recently,
Tao el al. obtained the following equivalent characterization
of both $\XMO(\rn)$ and $\xMO(\rn)$,
which is just \cite[Theorem 1.2]{txyyJFAA}.

\begin{theorem}\label{xMO-char}
The following statements are mutually equivalent:
\begin{itemize}
\item [{\rm(i)}] $f\in\xMO(\rn)$;
\item [{\rm (ii)}] $f\in\BMO(\rn)$ and enjoys the properties that
\begin{itemize}
\item [{\rm (ii)$_1$}] $\displaystyle{\lim_{a\to0^+}\sup_{|Q|=a}\co(f; Q)=0};$
\item [{\rm (ii)$_2$}]for any cube $Q$ of $\rn$,
$\displaystyle{\lim_{|x|\to\fz}\co(f; Q+x)=0}.$
\end{itemize}
\item [{\rm (iii)}] $f\in\XMO(\rn)$.
\end{itemize}
\end{theorem}

\begin{remark}\label{xMO-replace}
Proposition \ref{CMO-char}{\rm (ii)} can be replaced by
\begin{itemize}
\item [{\rm(ii')}]
$\displaystyle{\lim_{M\to\fz}\sup_{Q\cap Q(\mathbf{0},M)=\emptyset}\co(f;Q)=0},$
\end{itemize}
where $Q(\mathbf{0},M)$ denotes the cube centered at $\mathbf{0}$ with the side length $M$.
But, ${\rm(ii)_2}$ of Theorem \ref{xMO-char}{\rm (ii)} can not be replaced by {\rm (ii')};
see \cite[Proposition 2.5]{txyyJFAA} for more details.
\end{remark}

However, the equivalent characterization of $\MMO(\rn)$ is still unknown;
see \cite[Proposition 2.5 and Remark 2.6]{txyyJFAA} for more details
of the following open question.

\begin{question}\label{openQ-MMO}
It is interesting to find the equivalent characterization of $\MMO(\rn)$,
as well as its localized counterpart (see Question \ref{openQ-localXMO}),
via the mean oscillations.
\end{question}

As for the applications of these vanishing subspaces,
we know that the commutator $[b,T]$,
generated by $b\in\BMO(\rn)$ and the Calder\'on--Zygmund operator $T$,
plays an important role in harmonic analysis, complex analysis,
partial differential equations, and other fields in mathematics.
Here, we only list several typical \emph{bilinear} results;
other \emph{linear} and \emph{multi-linear}
results can be founded, for instance, in
\cite{tyyMMAS,tyyPA,L20} and their references.

In what follows, let
$\mathbb{Z}_+^{3n}:=(\mathbb{Z}_+)^{3n}$
and $L_{\rm c}^\fz(\rn)$ denote the set of all functions
$f\in L^\fz(\rn)$ with compact support.
We now consider the following particular type
bilinear Calder\'on--Zygmund operator $T$, whose kernel $K$ satisfies
\begin{itemize}
\item[{\rm(i)}]The standard \emph{size} and \emph{regularity} conditions:
for any multi-index $\az:=(\az_1,\ldots,\az_{3n})\in\zz_+^{3n}$ with
$|\az|:=\az_1+\cdots+\az_{3n}\le1$, there exists a positive constant
$C_{(\az)}$, depending on $\az$, such that,
for any $x,\ y,\ z\in\rn$ with $x\neq y$ or $x\neq z$,
\begin{align}\label{sizeregular}
|\partial^\az K(x,y,z)|\le C_{(\az)} (|x-y|+|x-z|)^{-2n-|\az|}.
\end{align}
Here and thereafter,
$\partial^\az:=(\frac\partial{\partial x_1})^{\az_1}\cdots
 (\frac\partial{\partial x_{3n}})^{\az_{3n}}$.

\item[{\rm(ii)}]The additional decay condition:
there exist positive constants $C$ and $\dz$ such that,
for any $x,\ y,\ z\in\rn$ with $|x-y|+|x-z|>1$,
\begin{align}\label{decay}
|K(x,y,z)|\le C (|x-y|+|x-z|)^{-2n-2-\dz}
\end{align}
\end{itemize}
and, for any $f,\ g\in L_{\rm c}^\fz(\rn)$ and $x\notin\supp(f)\cap\supp(g)$,
$T$ is supposed to have the following usual representation:
\begin{align*}
T(f,g)(x)=\int_{\rnn}K(x,y,z)f(y)g(z)\,dy\,dz,
\end{align*}
here and thereafter, $\supp(f):=\{x\in\rn:\ f(x)\neq 0\}$.
Notice that the (inhomogeneous) Coifman--Meyer bilinear
Fourier multipliers and the bilinear pseudodifferential
operators with certain symbols satisfy the above two conditions;
see, for instance, \cite{TX19} and its references.

Recall that, usually, a non-negative measurable function $w$ on $\rn$
is called a \emph{weight} on $\rn$.
For any given $\mathbf{p}:=(p_1,p_2)\in(1,\fz)\times(1,\fz)$,
let $p$ satisfy $\frac1p=\frac1{p_1}+\frac1{p_2}$.
Following \cite{BDMT15},  we call $\mathbf{w}:=(w_1,w_2)$ a
\emph{vector} $\mathbf{A}_{\mathbf{p}}(\rn)$ \emph{weight},
denoted by $\mathbf{w}:=(w_1,w_2)\in \mathbf{A}_{\mathbf{p}}(\rn)$,
if
\begin{align*}
[\mathbf{w}]_{\mathbf{A}_{\mathbf{p}}(\rn)}&:=
\sup_Q\lf[\frac 1{|Q|}\int_Q w(x)\,dx\r]\lf\{\frac 1{|Q|}
\int_Q \lf[w_1(x)\r]^{1-p_1'}\,dx\r\}^{\frac p{p_1'}}\\
&\quad\times\lf\{\frac 1{|Q|}\int_Q \lf[w_2(x)\r]^{1-p_2'}\,dx\r\}^{\frac p{p_2'}}<\fz,
\end{align*}
where $w:=w_1^{p/p_1}w_2^{p/p_2}$,
$ 1/{p_1}+1/{p_1'}=1=1/{p_2}+1/{p_2'}$,
and the supremum is taken over all cubes $Q$ of $\rn$.
In what follows, for any given weight $w$ on $\rn$
and any measurable subset $E\subseteqq\rn$,
the \emph{symbol $L^p_w(E)$}, with $p\in(0,\fz)$, denotes the set of all measurable
functions $f$ on $E$ such that
$$\|f\|_{L^p_w(E)}:=\lf[\int_E |f(x)|^p w(x)\,dx\r]^\frac1p<\fz,$$
and, when $w\equiv1$, we write $L^p_w(E)=:L^p(E)$.
Also, $\|\cdot\|_{L^\fz(E)}$ represents
the essential supremum on $E$.

Recall also that the \emph{bilinear commutators}
$[b,T]_1$ and $[b,T]_2$ are defined, respectively,
by setting, for any $f,\ g\in L_{\rm c}^\fz(\rn)$ and
$x\notin\supp(f)\cap\supp(g)$,
\begin{align}\label{c1}
[b,T]_1(f,g)(x)
:=&\,\lf(bT(f,g)-T(bf,g)\r)(x)\noz\\
=&\,\int_{\rnn}[b(x)-b(y)]K(x,y,z)f(y)g(z)\,dy\,dz
\end{align}
and
\begin{align}\label{c2}
[b,T]_2(f,g)(x)
:=&\,\lf(bT(f,g)-T(f,bg)\r)(x)\noz\\
=&\,\int_{\rnn}[b(x)-b(z)]K(x,y,z)f(y)g(z)\,dy\,dz.
\end{align}

The following theorem, obtained in \cite[Theorem 1]{BT13}
for any given $p\in(1,\fz)$
and in \cite[Theorem 1]{txy18}
for any given $p\in(\frac12,1]$,
showed that the bilinear compact commutators $\{[b,T]_i\}_{i=1,2}$
are compact for $b\in\CMO(\rn)$.

\begin{theorem}\label{compact-thm-CMO}
Let $(p_1,p_2)\in(1,\fz)\times(1,\fz)$,
$p\in(\frac12,\fz)$ with $\frac1p=\frac1{p_1}+\frac1{p_2}$,
$b\in\CMO(\rn)$, and
$T$ be a bilinear Calder\'on--Zygmund operator whose kernel satisfies
\eqref{sizeregular}.
Then, for any $i\in\{1,2\}$,
the bilinear commutator $[b,T]_i$ as in \eqref{c1} or \eqref{c2}
is compact from $L^{p_1}(\rn)\times L^{p_2}(\rn)$ to $L^{p}(\rn)$.
\end{theorem}

If we require an extra additional decay \eqref{decay}
for the Calder\'on--Zygmund kernel in Theorem \ref{compact-thm-CMO},
we can then replace $\CMO(\rn)$ by $\XMO(\rn)$,
that is, deletes condition (iii) in Theorem \ref{CMO-char} of $\CMO(\rn)$.
This new compactness result was first obtained in \cite[Theorem 1.1]{TX19}
and then generalized into the weighted case,
namely, the following Theorem \ref{compact-thm}
which is just \cite[Theorem 1.4]{txyyJFAA}.

\begin{theorem}\label{compact-thm}
Let ${\bf p}:=(p_1,p_2)\in(1,\fz)\times(1,\fz)$,
$p\in(\frac12,\fz)$ with $\frac1p=\frac1{p_1}+\frac1{p_2}$,
${\bf w}:=(w_1,w_2)\in {\bf A}_{{\bf p}}(\rn)$,
$w:=w_1^{p/p_1}w_2^{p/p_2}$, $b\in\XMO(\rn)$, and
$T$ be a bilinear Calder\'on--Zygmund operator whose kernel satisfies
\eqref{sizeregular} and \eqref{decay}. Then, for any $i\in\{1,2\}$,
the bilinear commutator $[b,T]_i$ as in \eqref{c1} or \eqref{c2}
is compact from $\Lpwa\times\Lpwb$ to $\Lpw$.
\end{theorem}

On the other hand, if the kernel behaves ``good'', such as the
\emph{Riesz transforms} $\{\mathcal{R}_j\}_{j=1}^n$:
$$\mathcal{R}_j(f)(x):={\rm \,p.\,v.\,}
\pi^{-\frac{n+1}{2}}\Gamma\lf(\frac{n+1}{2}\r)
\int_{\rn}\frac{y_j}{|y|^{n+1}}f(x-y)\,dy,$$
then the reverse of Theorem \ref{compact-thm-CMO}
holds true as well; see, for instance, the following
Theorem \ref{compact-thm-CMO-char} which is
just \cite[Theorem 3.1]{CCHTW18}.
Besides, it should be mentioned that the linear case of
Theorem \ref{compact-thm-CMO-char}
was obtained by Uchiyama \cite[Theorem 2]{U78}.

\begin{theorem}\label{compact-thm-CMO-char}
Let $(p_1,p_2)\in(1,\fz)\times(1,\fz)$ and
$p\in(\frac12,\fz)$ with $\frac1p=\frac1{p_1}+\frac1{p_2}$.
Then, for any $i\in\{1,2\}$ and $j\in\{1,\dots,n\}$,
the bilinear commutator $[b,\mathcal{R}_j]_i$
is compact from $L^{p_1}(\rn)\times L^{p_2}(\rn)$ to $L^{p}(\rn)$
if and only if $b\in\CMO(\rn)$.
\end{theorem}

However, the corresponding equivalent characterization of $\XMO(\rn)$ is still unknown.
For simplicity, we state this question in unweighted case.

\begin{question}\label{openQ-XMO}
Let $(p_1,p_2)\in(1,\fz)\times(1,\fz)$,
and $p\in(\frac12,\fz)$ be such that $\frac1p=\frac1{p_1}+\frac1{p_2}$.
Then it is interesting to find some bilinear
Calder\'on--Zygmund operator $T$ such that,
for any $i\in\{1,2\}$, the bilinear commutator $[b,T]_i$
is compact from $L^{p_1}(\rn)\times L^{p_2}(\rn)$ to $L^{p}(\rn)$
\emph{if and only if} $b\in\XMO(\rn)$.
\end{question}

Next, recall the Riesz transform characterizations of
$\BMO(\rn)$ and its vanishing subspaces.
\begin{theorem}\label{Riesz-trans-char}
Let $f\in L^1_\loc(\rn)$. Then
\begin{itemize}
\item[{\rm(i)}](\cite[Theorem 3]{FS72})
$f\in \BMO(\rn)$
if and only if there exist
functions $\{f_j\}_{j=0}^n\subset L^\fz(\rn)$ such that
$$f=f_0+\sum_{j=1}^n \mathcal{R}_j(f_j)$$
and
\begin{align}\label{RjBMO}
C^{-1}\|f\|_{\BMO(\rn)}
\le \sum_{j=0}^n\lf\|f_j\r\|_{L^\fz(\rn)}
\le C\|f\|_{\BMO(\rn)}
\end{align}
for some positive constant $C$ independent of $f$ and $\{f_j\}_{j=0}^n$.

\item[{\rm(ii)}](\cite[Theorem 1]{Sarason75})
$f\in \VMO(\rn)$
if and only if there exist
functions $\{f_j\}_{j=0}^n\subset [C_{\rm u}(\rn)\cap L^\fz(\rn)]$ such that
$$f=f_0+\sum_{j=1}^n \mathcal{R}_j(f_j)$$
and \eqref{RjBMO} holds true in this case.

\item[{\rm(iii)}](\cite[p.\,185]{N75})
$f\in \CMO(\rn)$
if and only if there exist
functions $\{f_j\}_{j=0}^n\subset C_0(\rn)$ such that
$$f=f_0+\sum_{j=1}^n \mathcal{R}_j(f_j)$$
and \eqref{RjBMO} holds true in this case.
\end{itemize}
\end{theorem}

\begin{question}\label{openQ-Riesz-trans-char}
Since the Riesz transform is well defined on $L^\fz(\rn)$,
it is interesting to find the counterpart of Theorem \ref{Riesz-trans-char}
when $f\in\MMO(\rn)$.
Moreover, since the Riesz transform characterization is useful when
proving the duality of CMO-$H^1$ type,
it is also interesting to find the dual spaces of $\MMO(\rn)$ and $\XMO(\rn)$.
\end{question}

When $\rn$ is replaced by some cube $Q_0$ with finite length, we have
$\VMO(Q_0)=\CMO(Q_0)$; see \cite{D02} for more details.
Moreover, the vanishing subspace on the \emph{spaces of homogeneous type},
denoted by $\mathfrak{X}$,
was studied in Coifman et al. \cite{CRW1976} and they proved
$(\mathcal{VMO}(\mathfrak{X}))^\ast=H^1(\mathfrak{X})$,
where $\mathcal{VMO}(\mathfrak{X})$ denotes
the closure in $\BMO(\mathfrak{X})$
of continuous functions on $\mathfrak{X}$ with compact support.
Notice that, when $\mathfrak{X}=\rn$, by \eqref{Cc=C0}, we have
$\mathcal{VMO}(\mathfrak{X})=\mathcal{VMO}(\rn)=\CMO(\rn)$.

Finally, we consider the localized version of these vanishing subspaces.
The following characterization of local $\VMO(\rn)$
is a part of \cite[Theorem 1]{B02}.
\begin{proposition}\label{vmo=VMO}
Let $\vmo(\rn)$ be the closure of $C_{\rm u}(\rn)\cap\bmo(\rn)$ in $\bmo(\rn)$.
Then $f\in\vmo(\rn)$ if and only if $f\in\bmo(\rn)$ and
$$\lim_{a\to0^+}\sup_{|Q|=a}\co(f;Q)=0.$$
\end{proposition}

Moreover, the following localized result of $\CMO(\rn)$ is just
Dafni \cite[Theorem 6]{D02};
see also \cite[Theorem 3]{B02}.

\begin{theorem}\label{cmo-char}
Let $\cmo(\rn)$ be the closure of $C_0(\rn)$ in $\bmo(\rn)$ .
Then $f\in\cmo(\rn)$ if and only if $f\in\bmo(\rn)$ and
$$\lim_{a\to0^+}\sup_{|Q|=a}\co(f;Q)
=0=
\lim_{M\to\fz}\sup_{|Q|>1,\,Q\cap Q(\mathbf{0},M)=\emptyset}\fint_Q|f|.$$
\end{theorem}

In addition, the localized version of Theorem \ref{Riesz-trans-char}
can be found in \cite[Corollary 1]{G79} for $\bmo(\rn)$,
and in \cite[Theorems 1 and 3]{B02} for $\vmo(\rn)$ and $\cmo(\rn)$,
respectively.

\begin{question}\label{openQ-localXMO}
Let ${\rm mmo\,}(\rn)$, $\xmo(\rn)$, and ${\rm x_1mo\,}(\rn)$
be, respectively, the closure in $\bmo(\rn)$ of $A_\fz(\rn)$,
$B_\fz(\rn)$, and $B_1(\rn)$.
It is interesting to find the counterparts of
\begin{itemize}
\item[{\rm(i)}] Theorem \ref{cmo-char}
with $\cmo(\rn)$ replaced by $\xmo(\rn)$;

\item[{\rm(ii)}] Theorem \ref{xMO-char} with $\XMO(\rn)$ and $\xMO(\rn)$
replaced, respectively, by $\xmo(\rn)$ and ${\rm x_1mo\,}(\rn)$;

\item[{\rm(iii)}] Question \ref{openQ-Riesz-trans-char} with $\MMO(\rn)$
replaced by ${\rm mmo\,}(\rn)$;

\item[{\rm(iv)}] the dual result $(\cmo(\rn))^\ast=h^1(\rn)$,
in \cite[Theorem 9]{D02},
with $\cmo(\rn)$ replaced by ${\rm mmo\,}(\rn)$ or $\xmo(\rn)$,
where $h^1(\rn)$ is the \emph{localized Hardy space};

\item[{\rm(v)}] the equivalent characterizations for ${\rm mmo\,}(\rn)$ and
$\xmo(\rn)$ via \emph{localized Riesz transforms}.
\end{itemize}
\end{question}

\begin{remark}\label{VMorrey}
For the studies of vanishing Morrey spaces,
we refer the reader to \cite{AS18,AAS19,A20,AAS20}.
\end{remark}

\subsection{Vanishing John--Nirenberg--Campanato spaces}\label{Subsec-VanJNp}

Very recently, the vanishing subspaces of John--Nirenberg spaces
were also studied in \cite{tyyACV,BB20}.
Indeed, as a counterpart of Subsection \ref{Subsec-VanBMO},
the vanishing subspaces of JNC spaces enjoy similar characterizations
which are summarized in this subsection.

\begin{definition}\label{def-VJNp}
Let $p\in(1,\fz)$, $q\in[1,\fz)$, $s\in\zz_+$, and $\az\in\rr$.
The \emph{vanishing subspace $VJN_{(p,q,s)_\az}(\cx)$} is defined by setting
$$VJN_{(p,q,s)_\az}(\cx):=\lf\{f\in JN_{(p,q,s)_\az}(\cx):\,\,
\limsup_{a\to0^+}\sup_{{\rm size}\le a}\widetilde{\co}_{(p,q,s)_\az}(f;\{Q_i\}_i)=0  \r\},$$
where
\begin{align*}
\widetilde{\co}_{(p,q,s)_\az}(f;\{Q_i\}_i)
:=\lf\{\sum_i|Q_i|\lf[|Q_i|^{-\alpha}\lf\{\fint_{Q_i}
\lf|f(x)-P_{Q_i}^{(s)}(f)(x)\r|^q\,dx\r\}^{\frac 1q}\r]^p\r\}^{\frac 1p}
\end{align*}
and the supremum is taken over all collections of interior pairwise
disjoint cubes $\{Q_i\}_i$ of $\cx$ with side lengths no more than $a$.
To simplify the notation, write
$VJN_{p,q}(\cx):=VJN_{(p,q,0)_0}(\cx)$
and $VJN_{p}(\cx):=VJN_{p,1}(\cx)$.
\end{definition}

On the unit cube $[0,1]^n$, the space $VJN_{(p,q,s)_\az}([0,1]^n)$
was studied by A. Brudnyi and Y. Brudnyi in \cite{BB20} with different symbols.
The following characterization (Theorem \ref{VJNp-BB})
and duality (Theorem \ref{Vdual-BB}) are just, respectively,
\cite[Theorem 3.14 and 3.7]{BB20}.
Notice that, when $\az\ge\frac{s+1}n$, from \cite[Lemma 4.1]{BB20},
we deduce that
$JN_{(p,q,s)_\alpha}([0,1]^n)=\mathcal{P}_s([0,1]^n)$ is trivial.

\begin{theorem}\label{VJNp-BB}
Let $p$, $q\in[1,\fz)$, $s\in\zz_+$, and $\az\in(-\fz,\frac{s+1}n)$. Then
$$VJN_{(p,q,s)_\az}([0,1]^n)=\overline{C^\fz([0,1]^n)\cap
JN_{(p,q,s)_\az}([0,1]^n)}^{JN_{(p,q,s)_\az}([0,1]^n)},$$
where $C^\fz([0,1]^n):=C^\fz(\rn)|_{[0,1]^n}$ denotes the restriction of
infinitely differentiable functions from $\rn$ to $[0,1]^n$.
\end{theorem}

\begin{theorem}\label{Vdual-BB}
Let $p$, $q\in(1,\fz)$, $s\in\zz_+$, and $\az\in(-\fz,\frac{s+1}n)$. Then
$$\lf(VJN_{(p,q,s)_\az}([0,1]^n) \r)^\ast=HK_{(p',q',s)_\az}([0,1]^n),$$
where $\frac1{p}+\frac1{p'}=1=\frac1{q}+\frac1{q'}$.
\end{theorem}

It is obvious that Theorems \ref{VJNp-BB} and \ref{Vdual-BB} hold true
with $[0,1]^n$ replaced by any cube $Q_0$ of $\rn$.
As an application of the duality,
Tao et al. \cite[Proposition 5.7]{tyyACV} showed that,
for any $p\in(1,\fz)$ and any given cube $Q_0$ of $\rn$,
$$\lf[L^p(Q_0)/\cc\r]\subsetneqq VJN_p(Q_0)$$
which proves the \emph{nontriviality} of $VJN_p(Q_0)$,
here and thereafter,
$$L^p(\cx)/\cc:=\lf\{f\in L^1_{\loc}(\cx):\ \|f\|_{L^p(\cx)/\cc}<\fz \r\}$$
with
$$\|f\|_{L^p(\cx)/\cc}:=\inf_{c\in\cc} \|f+c\|_{L^p(\cx)}.$$
\begin{remark}
There exists a gap in the proof of \cite[Proposition 5.7]{tyyACV}:
We can not deduce
\begin{align}\label{VJNpQ0}
	\lf(VJN_p(Q_0)\r)^{\ast\ast}=JN_p(Q_0),
\end{align}
namely, \cite[(5.2)]{tyyACV},
directly from Theorems \ref{Vdual-BB} and \ref{duality}
because, in the statements of these dual theorems,
$q$ can not equal to $1$.
Indeed, \eqref{VJNpQ0} still holds true
due to the equivalence of $JN_{p,q}(Q_0)$ with $q\in[1,p)$.
Precisely, let $p\in(1,\fz)$ and $q\in(1,p)$.
By Theorems \ref{Vdual-BB} and \ref{duality},
we obtain
\begin{align*}
\lf(VJN_{p,q}(Q_0)\r)^{\ast\ast}=JN_{p,q}(Q_0),
\end{align*}
which, together with Theorems \ref{JNpqa=JNp1a}
and \ref{VJNp-char} below, further implies that
\begin{align*}
\lf(VJN_{p}(Q_0)\r)^{\ast\ast}
=\lf(VJN_{p,q}(Q_0)\r)^{\ast\ast}
=JN_{p,q}(Q_0)=JN_{p}(Q_0),
\end{align*}
and hence \eqref{VJNpQ0} holds true.
This fixes the gap in the proof of \cite[(5.2)]{tyyACV}.
\end{remark}

Next, we consider the case $\cx=\rn$.
The following proposition indicates that the
convolution is a suitable tool
when approximating functions in $JN_p(\rn)$,
which is a counterpart of \cite[Lemma 1]{Sarason75}.
Indeed, the approximate functions in the proofs of
both Theorems \ref{VJNp-char} and \ref{CJNp-char} are
constructed via the convolution; see \cite{tyyACV} for more details.

\begin{proposition}\label{f-ast-phi}
Let $p\in(1,\fz)$ and $\varphi\in L^1(\rn)$
with compact support.
If $f\in JN_p(\rn)$, then $f\ast\varphi\in JN_p(\rn)$ and
$$\|f\ast\varphi\|_{JN_p(\rn)}\le2\|\varphi\|_{L^1(\rn)}\|f\|_{JN_p(\rn)}.$$
\end{proposition}

\begin{proof}
Let $p,\ \varphi$, and $f$ be as in this lemma.
Then, for any cube $Q$ of $\rn$, by the Fubini theorem, we have
\begin{align}\label{ineq-ast}
\co(f\ast\varphi;Q)&=\fint_Q\lf|f\ast\varphi(x)-(f\ast\varphi)_Q \r|\,dx\notag\\
&=\fint_Q\lf|\fint_Q\int_{\rn}\varphi(z)[f(x-z)-f(y-z)]\,dz\,dy \r|\,dx\notag\\
&\le\int_{\rn}\fint_Q\fint_Q|\varphi(z)|\lf|f(x-z)-f(y-z) \r|\,dy\,dx\,dz\notag\\
&=\int_{\rn}|\varphi(z)|\fint_{Q-z}\fint_{Q-z}\lf|f(x)-f(y) \r|\,dy\,dx\,dz\notag\\
&\le2\int_{\rn}|\varphi(z)|\co(f;Q-z)\,dz,
\end{align}
where $Q-z:=\{w-z:\,\,w\in Q\}$.
Therefore, for any interior pairwise disjoint subcubes $\{Q_i\}_i$ of $\rn$,
by \eqref{ineq-ast} and the Minkowski generalized integral inequality,
we conclude that
\begin{align*}
&\lf\{\sum_i|Q_i|\lf[\co(f\ast\varphi;Q_i)\r]^p\r\}^\frac1p\\
&\quad\le2\lf\{\sum_i|Q_i|\lf[\int_{\rn}|\varphi(z)|\co(f;Q-z)\,dz\r]^p\r\}^\frac1p\\
&\quad=2\lf\{\sum_i\lf[\int_{\rn}|Q_i|^\frac1p|\varphi(z)|\co(f;Q_i-z)\,dz\r]^p\r\}^\frac1p\\
&\quad\le2\int_{\rn}\lf\{\sum_i\lf[|Q_i|^\frac1p|\varphi(z)|\co(f;Q_i-z)\r]^p\r\}^\frac1p dz\\
&\quad=2\int_{\rn}|\varphi(z)|\lf\{\sum_i|Q_i-z|\lf[\co(f;Q_i-z)\r]^p\r\}^\frac1p dz\\
&\quad\le2\|\varphi\|_{L^1(\rn)}\|f\|_{JN_p(\rn)},
\end{align*}
where $Q_i-z:=\{w-z:\,\,w\in Q_i\}$ for any $i$.
This further implies that
$$\|f\ast\varphi\|_{JN_p(\rn)}\le2\|\varphi\|_{L^1(\rn)}\|f\|_{JN_p(\rn)}$$
and hence finishes the proof of Proposition \ref{f-ast-phi}.
\end{proof}

The following equivalent characterization is just \cite[Theorem 3.2]{tyyACV}.
\begin{theorem}\label{VJNp-char}
Let $p\in(1,\fz)$.
Then the following three statements are mutually equivalent:
\begin{itemize}
\item [{\rm(i)}] $f\in \overline{D_p(\rn)\cap JN_p(\rn)}^{JN_p(\rn)}
=:VJN_p(\rn)$,
where
$$D_p(\rn):=\lf\{f\in C^\fz(\rn):\,\,|\nabla f|\in L^p(\rn)\r\}$$
and $\nabla f$ denotes the gradient of $f$;

\item [{\rm(ii)}] $f\in JN_{p}(\rn)$ and, for any given $q\in[1,p)$,
$$\lim_{a\to0^+}\sup_{\{\{Q_i\}_i:\,\,\ell(Q_i)\le a,\ \forall\,i\}}
\lf\{ \sum_i|Q_i|\lf[\fint_{Q_i}\lf|f(x)-f_{Q_i}\r|^q\, dx\r]^\frac pq \r\}^\frac1p=0,$$
where the supremum is taken over all collections $\{Q_i\}_i$ of interior pairwise
disjoint subcubes of $\rn$ with side lengths no more than $a$;

\item [{\rm(iii)}] $f\in JN_{p}(\rn)$ and
$$\lim_{a\to0^+}\sup_{\{\{Q_i\}_i:\,\,\ell(Q_i)\le a,\ \forall\,i\}}
\lf\{ \sum_i|Q_i|\lf[\fint_{Q_i}\lf|f(x)-f_{Q_i}\r|\,dx\r]^p \r\}^\frac1p=0,$$
where the supremum is taken over all collections $\{Q_i\}_i$ of interior pairwise
disjoint subcubes  of $\rn$ with side lengths no more than $a$.
\end{itemize}
\end{theorem}

Now, we recall another vanishing subspace of $JN_p(\rn)$ introduced in \cite{tyyACV},
which is of CMO type.
\begin{definition}\label{def-CJNp}
Let $p\in(1,\fz)$. The \emph{vanishing subspace $CJN_p(\rn)$} of $JN_p(\rn)$
is defined by setting
$$CJN_p(\rn):=\overline{C_{\rm c}^\fz(\rn)}^{JN_p(\rn)},$$
where $C_{\rm c}^\fz(\rn)$ denotes the set of
all infinitely differentiable functions on $\rn$
with compact support.
\end{definition}

The following theorem is just \cite[Theorem 4.3]{tyyACV}.
\begin{theorem}\label{CJNp-char}
Let $p\in(1,\fz)$. Then $f\in CJN_p(\rn)$
if and only if $f\in JN_p(\rn)$ and $f$ satisfies the following two conditions:
\begin{itemize}
\item [{\rm(i)}]
$$\lim_{a\to0^+}\sup_{\{\{Q_i\}_i:\,\,\ell(Q_i)\le a,\ \forall\,i\}}
\lf\{ \sum_i|Q_i|\lf[\fint_{Q_i}\lf|f(x)-f_{Q_i}\r|\,dx\r]^p \r\}^\frac1p=0,$$
where the supremum is taken over all collections $\{Q_i\}_i$ of interior pairwise
disjoint subcubes  of $\rn$ with side lengths $\{\ell(Q_i)\}_i$ no more than $a$;

\item [{\rm(ii)}]
$$\lim_{a\to\fz}\sup_{\{Q\subset\rn:\,\,\ell(Q)\ge a\}}
|Q|^{1/p}\fint_Q\lf|f(x)-f_Q\r|\,dx=0,$$
where the supremum is taken over all cubes $Q$ of $\rn$ with side lengths
$\ell(Q)$ no less than $a$.
\end{itemize}
\end{theorem}
Moreover, Tao et al. \cite[Theorem 4.4]{tyyACV} showed that
Theorem \ref{CJNp-char}(ii) can be replaced by the following statement:
$$\lim_{a\to\fz}\sup_{\{\{Q_i\}_i:\,\,\ell(Q_i)\ge a,\ \forall\,i\}}
\lf\{ \sum_i|Q_i|\lf[\fint_{Q_i}\lf|f(x)-f_{Q_i}\r|\,dx\r]^p \r\}^\frac1p=0,$$
where the supremum is taken over all collections $\{Q_i\}_i$ of interior pairwise
disjoint subcubes of $\rn$ with side lengths $\{\ell(Q_i)\}_i$ greater than $a$.

Furthermore, Tao et al. \cite[Corollary 4.5]{tyyACV} showed that
Theorem \ref{CJNp-char} holds true with
$$\fint_Q\lf|f(x)-f_Q\r|\,dx
\quad{\rm and}\quad
\fint_{Q_i}\lf|f(x)-f_{Q_i}\r|\,dx$$
in (i) and (ii) replaced, respectively, by
$$\lf[\fint_Q\lf|f(x)-f_Q\r|^q\,dx\r]^\frac1q
\quad{\rm and}\quad
\lf[\fint_{Q_i}\lf|f(x)-f_{Q_i}\r|^q\,dx\r]^\frac1q$$
for any $q\in[1,p)$.

However, there still exist some unsolved questions on the vanishing
John--Nirenberg space. The first question is on the case $p=1$.
\begin{question}\label{VJNS-p=1}
	The proof of \cite[Theorem 3.2]{tyyACV} indicates that
	(i) and (iii) of Theorem \ref{VJNp-char} are equivalent when $p=1$.
	However, the corresponding equivalent characterization of
	$CJN_1(\rn)$ is still \emph{unclear} so far.
\end{question}
The following question is just \cite[Question 5.5]{tyyACV}.
\begin{question}\label{openQ-CVJNp}
\begin{itemize}
\item[\rm{(i)}] It is still unknown whether or not
Theorems \ref{VJNp-char} and \ref{CJNp-char} hold true
with $JN_p(\rn)$ replaced by $JN_{(p,q,s)_\az}(\rn)$
when $p$, $q\in[1,\fz)$, $s\in\zz_+$, and $\az\in\rr\setminus\{0\}$.

\item[\rm{(ii)}] It is interesting to ask whether or not,
for any given $p\in(1,\fz)$, $q\in[1,\fz)$, $s\in\zz_+$, and $\az\in\rr$,
$$\lf(CJN_{(p,q,s)_\az}(\rn)\r)^\ast=HK_{(p',q',s)_\az}(\rn)
{\quad\rm or\quad}
\lf(CJN_{(p,q,s)_\az}(\rn)\r)^{\ast\ast}=JN_{(p,q,s)_\az}(\rn)$$
still holds true, where $1/p+1/p'=1=1/q+1/q'$, $CJN_{(p,q,s)_\az}(\rn)$ denotes
the closure of $C_{\rm c}^\fz(\rn)$ in $JN_{(p,q,s)_\az}(\rn)$,
and $HK_{(p',q',s)_\az}(\rn)$ the Hardy-type space introduced in
\cite[Definition 3.6]{tyyNA}.
\end{itemize}
\end{question}

Obviously, $[L^p(\rn)/\cc]\subset CJN_p(\rn)\subset VJN_p(\rn)\subset JN_p(\rn)$.
Then the last question naturally arises,
which is just \cite[Questions 5.6 and 5.8]{tyyACV}.
\begin{question}\label{Q-nontrivial}
Let $p\in(1,\fz)$.
It is interesting to ask whether or not
$$\lf[L^p(\rn)/\cc\r]\subsetneqq CJN_p(\rn)\subsetneqq VJN_p(\rn)\subsetneqq JN_p(\rn)$$
holds true. This is still \emph{unclear} so far.
\end{question}

\subsection{Vanishing congruent John--Nirenberg--Campanato spaces}\label{Subsec-VanConJNp}

As a counterpart of Subsection \ref{Subsec-VanJNp},
the vanishing subspace of congruent John--Nirenberg--Campanato spaces
$VJN_{(p,q,s)_{\alpha}}^{\mathrm{con}}(\mathcal{X})$
was studied in \cite{jtyyz21}.

\begin{definition}\label{d4.1}
Let $p$, $q\in[1,\infty)$, $s\in\zz_+$, and $\alpha\in\rr$.
The \emph{space}
$VJN_{(p,q,s)_{\alpha}}^{\mathrm{con}}(\mathcal{X})$ is defined by setting
$$VJN_{(p,q,s)_{\alpha}}^{\mathrm{con}}(\mathcal{X})
:=\overline{D_p(\mathcal{X})\cap
JN_{(p,q,s)_{\alpha}}^{\mathrm{con}}(\mathcal{X})}
^{JN_{(p,q,s)_{\alpha}}^{\mathrm{con}}(\mathcal{X})},$$
where
$$D_p(\cx):=\lf\{f\in C^\fz(\cx):\,\,|\nabla f|\in L^p(\cx)\r\}.$$
Also, simply write
$VJN_{p,q}^{\mathrm{con}}(\mathcal{X})
:=VJN_{(p,q,0)_{0}}^{\mathrm{con}}(\mathcal{X})$
and
$VJN_p^{\mathrm{con}}(\mathcal{X})
:=VJN_{p,1}^{\mathrm{con}}(\mathcal{X})$.
\end{definition}
\begin{remark}
Let $p$, $q\in[1,\infty)$, $s\in\zz_+$, $\alpha\in\rr$,
and $Q_0$ be any cube of $\rn$.
Then the observation $D_p(Q_0)=C^{\infty}(Q_0)$ implies that
$$VJN_{(p,q,s)_{\alpha}}^{\mathrm{con}}(Q_0)
=\overline{C^{\infty}(Q_0)\cap
	JN_{(p,q,s)_{\alpha}}^{\mathrm{con}}(Q_0)}^{JN_{(p,q,s)_{\alpha}}^{\mathrm{con}}(Q_0)}.$$
\end{remark}

Recall that $\mathcal{D}_m(\cx)$ with $m\in\zz$
is defined in the beginning of Subsection \ref{Subsec-ConJNC}.
The following characterizations, namely,
Theorems \ref{t4.3} and \ref{C-infty}, are
just \cite[Theorems 3.5 and 3.9]{jtyyz21}, respectively.

\begin{theorem}\label{t4.3}
Let $p$, $q\in[1,\infty)$, $s\in\zz_+$,
$\alpha\in(-\infty,\frac{s+1}{n})$,
and $Q_0$ be a cube of $\rn$. Then
$f\in VJN_{(p,q,s)_\alpha}^{\mathrm{con}}(Q_0)$
if and only if $f\in L^q(Q_0)$ and
\begin{align}\label{CJN00}
\limsup_{m\rightarrow\infty}\sup_{\{Q_j\}_j\subset \mathcal{D}_m(Q_0)}
\left[\sum_{j}\left|Q_{j}\right|\lf\{\left|Q_{j}\right|^{-\alpha}\left[\fint_{Q_{j}}
\left|f-P_{Q_{j}}^{(s)}(f)\right|^{q}\right]^{\frac{1}{q}}\r\}^{p} \right]^{\frac{1}{p}}=0,
\end{align}
where the second supremum is taken over all collections of interior pairwise disjoint
cubes $\{Q_j\}_j\subset\mathcal{D}_m(Q_0)$ for any $m\in \zz$.
\end{theorem}
\begin{corollary}\label{p1a0}
Let $p=1$, $q\in[1,\fz)$, $s\in\zz_+$, $\az=0$,
and $Q_0$ be a cube of $\rn$.
Then \eqref{CJN00} holds true for any $f\in L^q(Q_0)$.
\end{corollary}
\begin{proof}
By Proposition \ref{p2.1}(ii) and the definition of
$VJN_{(p,q,s)_{\alpha}}^{\mathrm{con}}(Q_0)$,
we have
$$\lf[L^q(Q_0)/\mathcal{P}_s(Q_0)\r]=VJN_{(p,q,s)_{\alpha}}^{\mathrm{con}}(Q_0)
=JN_{(p,q,s)_{\alpha}}^{\mathrm{con}}(Q_0),$$
which, combined with Theorem \ref{t4.3},
then completes the proof of Corollary \ref{p1a0}.
\end{proof}

\begin{theorem}\label{C-infty}
Let $p\in[1,\infty)$ and $q\in[1,p]$.
Then $f\in VJN_{p,q}^{\mathrm{con}}(\rn)$ if and only if $f\in JN_{p,q}^{\mathrm{con}}(\rn)$ and
$$\limsup_{m\rightarrow\infty}\sup_{\{Q_j\}_j\subset \mathcal{D}_m(\rn)}
  \lf[\sum_{j}|Q_j|\lf(\fint_{Q_j}\lf|f-f_{Q_j}\r|^q\r)^\frac{p}{q}\r]^{\frac{1}{p}}=0,$$
where the second supremum is taken over all collections of interior pairwise disjoint
cubes $\{Q_j\}_j\subset\mathcal{D}_m(\rn)$ for any $m\in \zz$.
\end{theorem}

We can partially answer Question \ref{Q-nontrivial} in the congruent JNC space as follows.
\begin{proposition}\label{nontrivial}
Let $I_0$ be any given bounded interval of $\rr$,
and $Q_0$ any given cube of $\rn$.
\begin{enumerate}
\item[{\rm(i)}] (\cite[Proposition 3.11]{jtyyz21})
If $p\in(1,\fz)$ and $q\in[1,p)$, then
$[L^p(\rr)/\cc]\subsetneqq VJN_{p,q}^{\mathrm{con}}(\rr)$.

\item[{\rm(ii)}] (\cite[Proposition 3.12]{jtyyz21})
If $p\in(1,\fz)$ and $q\in[1,p)$, then
$VJN_{p,q}^{\mathrm{con}}(\rr)\subsetneqq JN_{p,q}^{\mathrm{con}}(\rr)$
and
$VJN_{p,q}^{\mathrm{con}}(I_0)\subsetneqq JN_{p,q}^{\mathrm{con}}(I_0)$.

\item[{\rm(iii)}] (\cite[Proposition 4.40]{jtyyz21})
If $p\in(1,\fz)$ and $q\in(1,p)$, then
$[L^p(Q_0)/\cc]\subsetneqq VJN_{p,q}^{\mathrm{con}}(Q_0)$.
\end{enumerate}
\end{proposition}
Also, it is easy to show that $[L^1(Q_0)/\cc]= VJN_{1}^{\mathrm{con}}(Q_0)
=JN_{1}^{\mathrm{con}}(Q_0)$; see Remark \ref{LpJNpWLp}(ii).

The following VMO-$H^1$ type duality is just \cite[Theorem 4.39]{jtyyz21}.
\begin{theorem}\label{VMO-H1-duality}
	Let $p$, $q\in(1,\infty)$, $s\in\zz_+$, $\frac1{p}+\frac1{p'}=1=\frac1{q}+\frac1{q'}$,
	$\alpha\in(-\infty,\frac{s+1}{n})$,
	and $Q_0$ be any given cube of $\rn$. Then
	$$\lf(VJN_{(p,q,s)_{\alpha}}^{\mathrm{con}}(Q_0)\r)^*
	=HK_{(p',q',s)_{\alpha}}^{\mathrm{con}}(Q_0)$$
	in the following sense: there exists an isometric isomorphism
	$$K:\ HK_{(p',q',s)_{\alpha}}^{\mathrm{con}}(Q_0)
	\longrightarrow \lf(VJN_{(p,q,s)_{\alpha}}^{\mathrm{con}}(Q_0)\r)^*$$
	such that, for any $g\in HK_{(p',q',s)_\alpha}^{\mathrm{con}}(Q_0)$
	and $f\in VJN_{(p,q,s)_\alpha}^{\mathrm{con}}(Q_0)$,
	$$
	\langle Kg,f\rangle=\langle g,f\rangle.
	$$
\end{theorem}

Similarly to Question \ref{openQ-CVJNp}(ii), the following question,
posed in \cite[Remark 4.41]{jtyyz21}, is still unsolved so far.

\begin{question}
For any given $p$, $q\in(1,\infty)$,
$s\in\zz_+$, and $\alpha\in(-\infty,\frac{s+1}{n})$,
it is interesting to ask whether or not
$$
\lf(CJN_{(p,q,s)_{\alpha}}^{\mathrm{con}}(\rn)\r)^{\ast}
=HK_{(p',q',s)_{\alpha}}^{\mathrm{con}}(\rn)
\quad\text{and}\quad
\lf(CJN_{(p,q,s)_{\alpha}}^{\mathrm{con}}(\rn)\r)^{\ast\ast}
=JN_{(p,q,s)_{\alpha}}^{\mathrm{con}}(\rn)$$
hold true, where $CJN_{(p,q,s)_{\alpha}}^{\mathrm{con}}(\rn)$ denotes the closure
of $C_{\mathrm{c}}^{\infty}(\rn)$ in $JN_{(p,q,s)_{\alpha}}^{\mathrm{con}}(\rn)$
and $\frac1{p}+\frac1{p'}=1=\frac1{q}+\frac1{q'}$.
This is still \emph{unclear} so far.
\end{question}

\noindent\textbf{Acknowledgements}.
Jin Tao would like to thank Hongchao Jia and Jingsong Sun
for some useful discussions on this survey.

\bigskip

\noindent Jin Tao, Dachun Yang (Corresponding author) and Wen Yuan

\medskip

\noindent Laboratory of Mathematics and Complex Systems
(Ministry of Education of China),
School of Mathematical Sciences, Beijing Normal University,
Beijing 100875, People's Republic of China

\smallskip

\noindent{\it E-mails:} \texttt{jintao@mail.bnu.edu.cn} (J. Tao)

\noindent\phantom{{\it E-mails:}} \texttt{dcyang@bnu.edu.cn} (D. Yang)

\noindent\phantom{{\it E-mails:}} \texttt{wenyuan@bnu.edu.cn} (W. Yuan)

\end{document}